\documentclass[reqno, 11pt, a4paper]{amsart}
\usepackage{amsfonts, amsmath, amssymb, amsthm}
\usepackage[hmargin=2.75cm, vmargin=2.9cm]{geometry}
\usepackage[dvipsnames]{xcolor}
\usepackage{graphicx, setspace}
\usepackage[colorlinks=true]{hyperref}

\newtheorem{thm}{Theorem}[section]
\newtheorem{lemma}[thm]{Lemma}
\newtheorem{prop}[thm]{Proposition}
\newtheorem{cor}[thm]{Corollary}

\theoremstyle{definition}
\newtheorem{defn}[thm]{Definition}
\newtheorem{rmk}[thm]{Remark}

\newcommand{\ep}{\epsilon}
\newcommand{\ga}{\gamma}
\newcommand{\ls}{{\lambda,\sigma}}
\newcommand{\vep}{\varepsilon}
\newcommand{\vrh}{\varrho}
\newcommand{\vth}{\vartheta}
\newcommand{\vp}{\varphi}
\newcommand{\pa}{\partial}

\newcommand{\mn}{\mathbb{N}}
\newcommand{\mr}{\mathbb{R}}
\newcommand{\ms}{\mathbb{S}}
\newcommand{\mbp}{\mathbf{p}}
\newcommand{\mbx}{\mathbf{x}}
\newcommand{\mca}{\mathcal{A}}
\newcommand{\mcb}{\mathcal{B}}
\newcommand{\mcc}{\mathcal{C}}
\newcommand{\mce}{\mathcal{E}}
\newcommand{\mcf}{\mathcal{F}}

\newcommand{\mch}{\mathcal{H}}
\newcommand{\mci}{\mathcal{I}}
\newcommand{\mcm}{\mathcal{M}}
\newcommand{\mcn}{\mathcal{N}}
\newcommand{\mcp}{\mathcal{P}}
\newcommand{\mcq}{\mathcal{Q}}

\newcommand{\mcu}{\mathcal{U}}
\newcommand{\mcw}{\mathcal{W}}
\newcommand{\mfb}{\mathfrak{B}}
\newcommand{\mfl}{\mathfrak{L}}

\newcommand{\bg}{\bar{g}}
\newcommand{\bh}{\bar{h}}
\newcommand{\bu}{\bar{u}}
\newcommand{\bx}{\bar{x}}
\newcommand{\chg}{\check{g}}
\newcommand{\chh}{\check{h}}
\newcommand{\chpi}{\check{\pi}}

\newcommand{\hf}{\hat{f}}
\newcommand{\hg}{\hat{g}}
\newcommand{\hr}{\hat{r}}
\newcommand{\hw}{\hat{w}}
\newcommand{\hep}{\hat{\epsilon}}
\newcommand{\hmfl}{\hat{\mfl}}

\newcommand{\tif}{\tilde{f}}
\newcommand{\tig}{\tilde{g}}
\newcommand{\tih}{\tilde{h}}
\newcommand{\tiu}{\tilde{u}}

\newcommand{\tiy}{\tilde{y}}
\newcommand{\tirh}{\tilde{\rho}}
\newcommand{\tipi}{\tilde{\pi}}
\newcommand{\tich}{\tilde{\chi}}

\newcommand{\obr}{\overline{B_R}}
\newcommand{\obrt}{\overline{B_{R/2}}}
\newcommand{\opbr}{\overline{\pa B_R'}}
\newcommand{\ox}{\overline{X}}
\newcommand{\whe}{\widehat{E}}
\newcommand{\whf}{\widehat{F}}

\newcommand{\whq}{\widehat{Q}}

\newcommand{\whw}{\widehat{W}}

\newcommand{\wtc}{\widetilde{C}}
\newcommand{\wtf}{\widetilde{F}}
\newcommand{\wtk}{\widetilde{K}}

\newcommand{\wtt}{\widetilde{T}}
\newcommand{\wtu}{\widetilde{U}}
\newcommand{\wtv}{\widetilde{V}}
\newcommand{\wth}{\widetilde{\mch}}
\newcommand{\wtps}{\widetilde{\Psi}}

\newcommand{\mrg}{(\mr^N_+;x_N^{1-2\ga})}
\newcommand{\loc}{\textnormal{loc}}

\renewcommand{\(}{\left(}
\renewcommand{\)}{\right)}
\newcommand{\la}{\left\langle}
\newcommand{\ra}{\right\rangle}

\begin{document}
\title[A compactness theorem of the fractional Yamabe problem]{A compactness theorem of the fractional Yamabe problem, Part I: The non-umbilic conformal infinity}

\author{Seunghyeok Kim}
\address[Seunghyeok Kim]{Department of Mathematics and Research Institute for Natural Sciences, College of Natural Sciences, Hanyang University, 222 Wangsimni-ro Seongdong-gu, Seoul 04763, Republic of Korea}
\email{shkim0401@gmail.com}

\author{Monica Musso}
\address[Monica Musso]{Department of Mathematical Sciences, University of Bath, Bath BA2 7AY, United Kingdom, and Facultad de Matem\'{a}ticas, Pontificia Universidad Cat\'{o}lica de Chile, Avenida Vicu\~{n}a Mackenna 4860, Santiago, Chile}
\email{m.musso@bath.ac.uk}

\author{Juncheng Wei}
\address[Juncheng Wei] {Department of Mathematics, University of British Columbia, Vancouver, B.C., Canada, V6T 1Z2}
\email{jcwei@math.ubc.ca}

\begin{abstract}
Assume that $(X, g^+)$ is an asymptotically hyperbolic manifold, $(M, [\bh])$ is its conformal infinity, $\rho$ is the geodesic boundary defining function associated to $\bh$ and $\bg = \rho^2 g^+$.
For any $\ga \in (0,1)$, we prove that the solution set of the $\ga$-Yamabe problem on $M$ is compact in $C^2(M)$
provided that convergence of the scalar curvature $R[g^+]$ of $(X, g^+)$ to $-n(n+1)$ is sufficiently fast as $\rho$ tends to 0 and the second fundamental form on $M$ never vanishes.
Since most of the arguments in blow-up analysis performed here is irrelevant to the geometric assumption imposed on $X$,
our proof also provides a general scheme toward other possible compactness theorems for the fractional Yamabe problem.
\end{abstract}

\date{\today}
\maketitle

\allowdisplaybreaks
\numberwithin{equation}{section}

\section{Introduction}
Given any $n \in \mn$, let $(X^{n+1}, g^+)$ be an asymptotically hyperbolic manifold with conformal infinity $(M^n, [\bh])$.
According to \cite{GZ, MM, JS, CG, GQ}, there exists a family of self-adjoint conformal covariant pseudo-differential operators $P^{\ga}[g^+, \bh]$ on $M$ in general
whose principal symbols are the same as those of $(-\Delta_{\bh})^{\ga}$.
If $(X, g^+)$ is Poincar\'e-Einstein and $\ga \in \mn$,
the operator $P^{\ga}[g^+, \bh]$ coincides with the GJMS operator constructed by Graham et al. \cite{GJMS} via the ambient metric; refer to Graham and Zworski \cite{GZ}.
In particular, $P^{\ga}[g^+, \bh]$ is equal to the conformal Laplacian or the Paneitz operator for $\ga = 1$ or 2, respectively.

Let us call $Q^{\ga}[g^+, \bh] = P^{\ga}[g^+, \bh](1)$ the {\it $\ga$-scalar curvature}.
One natural question is if a conformal metric $\bh'$ to $\bh$ on $M$ exists whose $\ga$-scalar curvature $Q^{\ga}[g^+, \bh']$ is constant.
By virtue of the conformal covariance property of $P^{\ga}$, it is reduced to search a smooth solution of the equation
\begin{equation}\label{eq_Yamabe}
P^{\ga}[g^+, \bh] u = c u^p \quad \text{and} \quad u > 0 \quad \text{on } (M^n, \bh)
\end{equation}
for some constant $c \in \mr$ provided $n > 2\ga$ and $p = 2_{n,\ga}^*-1 = (n+2\ga)/(n-2\ga)$.

For $\ga = 1$, the study on the existence of a solution to \eqref{eq_Yamabe} was initiated by Yamabe \cite{Ya}
and completely solved through the successive works of Trudinger \cite{Tr}, Aubin \cite{Au} and Schoen \cite{Sc3}.
See also Lee and Parker \cite{LP} and Bahri \cite{Ba} where a unified proof based on the use of conformal normal coordinates
and a proof not depending on the positive mass theorem are devised, respectively.
If $\ga = 2$ (and $n \ge 5$), existence theory of \eqref{eq_Yamabe} becomes considerably harder because of the lack of a maximum principle.
Up to now, only partial results are available such as Qing and Raske \cite{QR}, Gursky and Malchiodi \cite{GM} and Hang and Yang \cite{HY}.
In \cite{HY}, the authors could treat a general class of manifolds having the property that
the Yamabe constant \eqref{eq_Y_cons} is positive and there exists a representative of the conformal class $[\bh]$ with semi-positive $Q$-curvature $Q^2$.
Meanwhile, equation \eqref{eq_Yamabe} with $\ga = 1/2$ has deep relationship with the boundary Yamabe problem
(or the higher dimensional Riemann mapping theorem) formulated by Escobar \cite{Es}; refer to Remark \ref{rmk_main} (1).
Its existence theory has been completed due to the effort of Escobar \cite{Es, Es2}, Marques \cite{Ma2, Ma3}, Almaraz \cite{Al2}, Chen \cite{Ch} and Mayer and Ndiaye \cite{MN2}.

If $\ga \notin \mn$, it is not a simple task to solve \eqref{eq_Yamabe} directly, since the operator $P^{\ga}[g^+, \bh]$ is nonlocal and defined in a rather abstract way.
However, Chang and Gonz\'alez \cite{CG} discovered that \eqref{eq_Yamabe} can be interpreted as a Caffarelli-Silvestre type degenerate elliptic equation in \cite{CaS}, which is indeed local,
for which a variety of well-known techniques like constraint minimization and the Moser iteration technique can be applied; see Proposition \ref{prop_CG} for a more precise description.
From this observation, Gonz\'alez and Qing \cite{GQ} succeeded to find solutions to \eqref{eq_Yamabe} for $\ga \in (0,1)$
under the hypothesis that the dimension $n$ of the underlying manifold $M$ is sufficiently high and $M$ is the non-umbilic boundary of $X$.
Their approach was further developed in the works of Gonz\'alez and Wang \cite{GW} and the authors of this paper \cite{KMS}, which cover most cases when the local geometry dominates.
In \cite{KMS}, the author also established the existence result for 2-dimensional or locally conformally flat manifolds
provided that a certain technical assumption on the Green's function of $P^{\ga}$ holds.
Recently, Mayer and Ndiaye \cite{MN} and Daskalopoulos et al. \cite{DSV} pursued the critical point at infinity approach and the flow approach, respectively, removing the technical condition on the Green's function.

Furthermore, Case and Chang \cite{CC} obtained an extension result for $\ga \in (1, n/2)$ which generalizes \cite{CG}.
By utilizing it and adapting the argument in Gursky and Malchiodi \cite{GM}, they also deduced that $P^{\ga}[g^+, \bh]$ satisfies a strong maximum principle for $\ga \in (1,\min\{2, n/2\})$
when $[\bh]$ possesses a metric whose scalar curvature is nonnegative and $\ga$-curvature is semi-positive.
It is plausible that this with some further ideas in \cite{GM} may allow one to get certain existence results of \eqref{eq_Yamabe} under the prescribed conditions.

\medskip
In many cases, \eqref{eq_Yamabe} may contain higher Morse index (or energy) solutions as shown in \cite{Sc4, Po} for $\ga = 1$.
When $\ga = 1$ and $(M, \bh)$ is the round sphere, Obata \cite{Ob} found all the solutions to \eqref{eq_Yamabe} and concluded that there is no uniform $L^{\infty}(M)$-bound on them.
Analogously, if $\ga \in (0,1)$ and $(X, g^+)$ is the Poincar\'e ball whose conformal infinity is the standard sphere,
the classification result \cite[Theorem 1.8]{JLX} of Jin et al. shows that the solution set of \eqref{eq_Yamabe} is not bounded in $L^{\infty}(M)$.

In this regards, for manifolds $M$ that are non-conformally diffeomorphic to the standard sphere,
Schoen \cite{Sc} raised a question on the $C^2(M)$-compactness of the solution set to \eqref{eq_Yamabe} with $\ga = 1$ and suggested a general strategy towards its proof.
The first affirmative answer was given by Schoen himself \cite{Sc2} in the locally conformally flat case.
Li and Zhu \cite{LZhu} obtained it in $n = 3$ and Druet \cite{Dr} did it in $n \le 5$.
If $n \ge 6$, the analysis is more delicate because one needs to prove that
the Weyl tensor vanishes at an order greater than $\lfloor (n-6)/2 \rfloor$ at a blow-up point.
By solving this technical difficulty, Marques \cite{Ma} and Li and Zhang \cite{LZ} could deal with the situation that either $n \le 7$ holds or the Weyl tensor never vanishes on $M$.
Assuming the validity of the positive mass theorem and performing refined blow-up analysis on the basis of the linearized problem (as in our Section \ref{sec_lin}),
Li and Zhang \cite{LZ2} extended the result up to dimension 11, and Khuri et al. \cite{KMS} finally verified it for $n \le 24$.
Surprisingly, according to Brendle \cite{Br} and Brendle and Marques \cite{BM},
there are $C^{\infty}$-metrics on the sphere $S^n$ with $n \ge 25$ such that
even though they are not conformally equivalent to the canonical metric, a blowing-up family of solutions to \eqref{eq_Yamabe} does exist.

For $\ga = 2$, Y. Li and Xiong \cite{LX} obtained the $C^4(M)$-compactness result if $5 \le n \le 9$ or $M$ is locally conformally flat,
and the positive mass theorem holds for the Paneitz operator $P^2$; see also the previous works \cite{HR, QR, Li}.
On the other hand, Wei and Zhao \cite{WZ} established a non-compactness result for $n \ge 25$.
While it is believed that $C^4(M)$-compactness holds in general up to dimension 24, a rigorous proof is not known yet.

For the boundary Yamabe problem, corresponding to the case $\ga = 1/2$, compactness results were deduced
when $X^{n+1}$ is locally conformally flat \cite{FO} or $n+1 = 3$ \cite{FO2}.
Compactness of the solution set also follows under the assumption that the second fundamental form on $M$ vanishes nowhere \cite{Al2}; refer to \cite{HL, DK} for more results.
Almaraz \cite{Al3} showed that a blow-up phenomenon still happens if $n+1 \ge 25$.

If $\ga$ is non-integer value, only a few has been revealed up to now.
As far as we know, the only article that investigates compactness of the solution set to \eqref{eq_Yamabe} for $\ga \in [1, n/2)$ is \cite{QR2} due to Qing and Raske,
which concerns with locally conformally flat manifolds $M$ with the positive Yamabe constant and the Poincar\'e exponent less than $(n-2\ga)/2$.
For the non-compactness, the authors of this paper \cite{KMW} constructed asymptotic hyperbolic manifolds
which are small perturbations of the Poincar\'e ball and exhibit a blow-up phenomenon for $n \ge 24$ if $\ga \in (0, \ga^*)$
and $n \ge 25$ if $\ga \in [\ga^*, 1)$ where $\ga^*$ is a number close to $0.940197$.
However, the compactness issue on \eqref{eq_Yamabe} for $\ga \in (0,1)$ has not been discussed in the literature so far, unless the underlying manifold is the Poincar\'e ball; see \cite{JLX}.

\medskip
In this paper, we are concerned with the compactness of the solution set to the {\it $\ga$-Yamabe problem} \eqref{eq_Yamabe} provided $\ga \in (0,1)$ and $c > 0$.
As can be predicted from the representation theorem of Palais-Smale sequences associated with fractional Yamabe-type equations in \cite{FG},
the conformal covariance property of $P^{\ga}$ makes us perform local analysis even though it is a nonlocal operator.

We will state the main theorem under a slightly more general setting; more precisely, we will allow $p$ to be subcritical.
Since we always assume that the metric $g^+$ in $X$ is fixed, we write $P^{\ga}_{\bh} = P^{\ga}[g^+, \bh]$ and $Q^{\ga}_{\bh} = Q^{\ga}[g^+, \bh]$.
\begin{thm} \label{thm_main_2}
Let $\ga \in (0,1)$,
\begin{equation}\label{eq_n}
n \ge \begin{cases}
7 &\text{for } 0 < \ga \le \sqrt{1 \over 19},\\
6 &\text{for } \sqrt{1 \over 19} < \ga \le {1 \over 2},\\
5 &\text{for } {1 \over 2} < \ga \le \sqrt{5 \over 11},\\
4 &\text{for } \sqrt{5 \over 11} < \ga < 1
\end{cases} \qquad \(\sqrt{1 \over 19} \simeq 0.229 \text{ and } \sqrt{5 \over 11} \simeq 0.674\)
\end{equation}
and $(X^{n+1}, g^+)$ be an asymptotically hyperbolic manifold with conformal infinity $(M^n, [\bh])$.
Denote by $\rho$ a geodesic defining function associated to $M$, i.e., a unique defining function splitting the metric $\bg = \rho^2 g^+$ as $d\rho^2 + h_{\rho}$ near $M$
where $\{h_{\rho}\}_{\rho}$ is a family of metrics on $M$ such that $h_0 = \bh$.

Assume that the first $L^2(X)$-eigenvalue $\lambda_1(-\Delta_{g^+})$ of the Laplace-Beltrami operator $-\Delta_{g^+}$ satisfies the inequality
\begin{equation}\label{eq_eig}
\lambda_1(-\Delta_{g^+}) > {(n-1)^2 \over 4} - \ga^2,
\end{equation}
the $\ga$-Yamabe constant defined as
\begin{equation}\label{eq_Y_cons}
\Lambda^{\ga}(M, [\bh]) = \inf_{u \in H^{\ga}(M) \setminus \{0\}} {\int_M u P^{\ga}_{\bh} u\, dv_{\bh} \over \(\int_M |u|^{2n \over n-2\ga} dv_{\bh}\)^{n-2\ga \over n}}
\end{equation}
is positive and
\begin{equation}\label{eq_hyp}
R[g^+] + n(n+1) = o(\rho^2) \quad \text{as } \rho \to 0 \text{ uniformly on } M
\end{equation}
where $R[g^+]$ is the scalar curvature of $(X, g^+)$.

If the second fundamental form $\pi$ of $(M,\bh) \subset (\ox,\bg)$ never vanishes, then for any $\vep_0 > 0$ small,
there exists a constant $C > 1$ depending only on $X^{n+1},\, g^+,\, \bh,\, \ga$ and $\vep_0$ such that
\begin{equation}\label{eq_u_est}
C^{-1} \le u \le C \quad \text{on } M \quad \text{and} \quad \|u\|_{C^2(M)} \le C
\end{equation}
for any solution in $H^{\ga}(M)$ to \eqref{eq_Yamabe} with $1 + \vep_0 \le p \le 2_{n,\ga}^*-1$.
\end{thm}
\noindent Here $H^{\ga}(M)$ is the fractional Sobolev space defined via a partition of unity of $M$.
A couple of remarks regarding Theorem \ref{thm_main_2} are in order.
\begin{rmk}\label{rmk_main}
(1) Condition \eqref{eq_hyp} makes the boundary Yamabe problem and the 1/2-Yamabe problem identical (modulo the remainder term) in view of the energy expansion.
Also, \eqref{eq_hyp} and the property that $\pi \ne 0$ at each point of $M$ are intrinsic in the sense
that they are not relevant to the choice of a representative of the class $[\bh]$.
Refer to Lemma 2.1 and Subsections 2.3, 2.4 of \cite{KMW2}.

\medskip \noindent (2) Hypothesis \eqref{eq_hyp} implies that the mean curvature $H$ is identically 0 on $M$; see \cite[Lemma 2.3]{KMW2}.
As a particular consequence, $\pi = \pi - H \bg$ on $M$, the latter tensor being the trace-free second fundamental form.
Thus our theorem generalizes the result of Almaraz \cite[Theorem 1.2]{Al} on the boundary Yamabe problem (corresponding to the case $\ga = 1/2$) under the further assumption that $H = 0$ on $M$.

\medskip \noindent (3) The standard transversality argument implies that the set of Riemannian metrics on $M$
whose second fundamental form is nonzero everywhere is open and dense in the space of all Riemannian metrics on $M$.
On the other hand, there exists an asymptotically hyperbolic manifold $X^{n+1}$ that can be realized as a small perturbation of the Poincar\'e half-space,
for which the solution set of the $\ga$-Yamabe problem is non-compact provided that $n \ge 24$ or $25$ according to the magnitude of $\ga \in (0,1)$; refer to \cite{KMW}.
In this example, the conformal infinity $M$ is the totally geodesic (in particular, umbilic) boundary of $X$.

\medskip \noindent (4) It is remarkable that the dimension restriction \eqref{eq_n} is exactly same as
the one appeared in the existence result \cite[Corollary 2.7]{KMW2} for equation \eqref{eq_Yamabe} on non-umbilic conformal infinities.

\medskip \noindent (5) We believe that the same type of a compactness result as Theorem \ref{thm_main_2} can be achieved
for any umbilic non-locally conformally flat conformal infinity whose Weyl tensor never vanishes,
whenever $n \ge 7$ for $\ga \in [1/2,1)$ and $n \ge 8$ for $\ga \in (0,1)$.
The condition on the dimension is suggested by an existence result \cite[Corollary 3.4]{KMW2}.
Also, if a suitable condition on the Green's function on $P_{\bh}^{\ga}$ is assumed,
a compactness result may be obtained provided that $M$ is either locally conformally flat or 2-dimensional as in Felli and Ould Ahmedou \cite{FO, FO2}.
An interesting question is to confirm whether the dimension assumption in \cite{KMW} is optimal to form a blow-up phenomenon of the solution set to \eqref{eq_Yamabe}.
\end{rmk}

One can establish Theorem \ref{thm_main_2} from the next theorem and elliptic regularity theory; see Subsection \ref{subsec_pf_m} and Appendix \ref{sec_reg}.
\begin{thm}[Vanishing theorem of the second fundamental form] \label{thm_main_1}
For $\ga \in (0,1)$ and $n \in \mn$ satisfying \eqref{eq_n}, let $(X^{n+1}, g^+)$ be an asymptotically hyperbolic manifold with conformal infinity $(M^n, [\bh])$ such that \eqref{eq_eig} is valid.
Moreover assume that $\rho$ is a geodesic defining function associated to $M$, $\bg = \rho^2 g^+$, $\Lambda^{\ga}(M, [\bh]) > 0$, and \eqref{eq_hyp} holds.
If $\{(u_m, y_m)\}_{m \in \mn}$ is a sequence of pairs in $C^{\infty}(M) \times M$ such that
each $u_m$ is a solution of \eqref{eq_Yamabe}, $y_m$ is a local maximum point of $u_m$ satisfying $u_m(y_m) \to \infty$ and $y_m \to y_0 \in M$ as $m \to \infty$,
then the second fundamental form $\pi$ at $y_0$ vanishes.
\end{thm}

As a corollary of Theorem \ref{thm_main_2}, we can compute the Leray-Schauder degree of all solutions to equation \eqref{eq_Yamabe}
if every hypothesis imposed in the theorem holds.
Since $P^{\ga}_{\bh}$ is a self-adjoint operator as shown in \cite{GZ},
any $L^{p+1}(M)$-normalized minimizer of
\[\inf_{u \in H^{\ga}(M) \setminus \{0\}} {\int_M u P^{\ga}_{\bh} u\, dv_{\bh} \over \(\int_M |u|^{p+1} dv_{\bh}\)^{2 \over p+1}} \quad \text{for } 1 \le p \le 2_{n,\ga}^*-1\]
(which is the same as \eqref{eq_Y_cons} if $p = 2_{n,\ga}^*-1$) solves
\begin{equation}\label{eq_Yamabe_1}
P^{\ga}_{\bh} u = \mce(u) |u|^{p-1}u \quad \text{on } (M, \bh) \quad \text{where} \quad \mce(u) = \int_M u P^{\ga}_{\bh} u\, dv_{\bh} \quad \text{(the energy of $u$)}.
\end{equation}
Furthermore, if the $\ga$-Yamabe constant $\Lambda^{\ga}(M, [\bh])$ is positive, then it is simple to check that the operator $T: L^{2n/(n+2\ga)}(M) \to H^{\ga}(M)$ is well-defined by the relation
\[T(v) = u \quad \text{if and only if} \quad P^{\ga}_{\bh} u = v \quad \text{on } M.\]
Hence, it is natural to define the map $\mcf_p: D_{\Lambda} \to L^{\infty}(M)$ by $\mcf_p(u) = u - T(\mce(u) u^p)$ where
\[D_{\Lambda} = \left\{ u \in L^{\infty}(M): u > \Lambda^{-1} \text{ and } \|u\|_{L^{\infty}(M)} < \Lambda \right\} \quad \text{for each } \Lambda > 1,\]
for which it holds that $\mcf_p(u) = 0$ if and only if $u$ is a solution of \eqref{eq_Yamabe}.
Elliptic estimate in Lemma \ref{lemma_reg_0} below implies that $\mcf_p$ is the sum of the identity and a compact map.
Besides we infer from Lemma \ref{lemma_main_3}, a consequence of Theorem \ref{thm_main_2}, that $0 \notin \mcf_p(\pa D_{\Lambda})$ for all $1 \le p \le 2_{n,\ga}^*-1$ if $\Lambda$ is sufficiently large.
Therefore the Leray-Schauder degree $\deg(\mcf_p, D_{\Lambda}, 0)$ of the map $\mcf_p$ in the domain $D_{\Lambda}$ with respect to the point $0 \in L^{\infty}(M)$ is well-defined.
\begin{thm} \label{thm_main_3}
Under the assumptions of Theorem \ref{thm_main_2}, it holds that
\[\deg(\mcf_p, D_{\Lambda}, 0) = -1.\]
In particular, the fractional Yamabe equation \eqref{eq_Yamabe} possesses a solution.
\end{thm}
\noindent Theorem \ref{thm_main_3} gives a new proof on the existence of a solution to \eqref{eq_Yamabe} under rather restrictive assumptions.
Compare it with \cite{GQ, KMW2}.
We also expect that there exists the strong Morse inequality in our framework; refer to \cite[Theorem 1.4]{KMS}.

\medskip
Our proof of the main theorems relies on Schoen's argument \cite{Sc} for the proof of the compactness theorem of the classical Yamabe problem.
It has been further developed through the works of Li and Zhu \cite{LZhu}, Druet \cite{Dr}, Marques \cite{Ma}, Khuri et al. \cite{KMS} (for the Yamabe problem),
Han and Li \cite{HL}, Felli and Ould Ahmedou \cite{FO, FO2}, Almaraz \cite{Al} (for the boundary Yamabe problem), G. Li \cite{Li}, Y. Y. Li and Xiong \cite{LX} (for the Q-curvature problem),
Schoen and Zhang \cite{SZ}, Li \cite{Li2}, Jin et al. \cite{JLX2} (for the classical and fractional Nirenberg problem)
and Niu et al. \cite{NPX} (for the critical Lane-Emden equation involving the regional fractional Laplacian) among others.

Although certain parts of the proof can be achieved from minor modification of the classical arguments, there are still plenty of technical difficulties which demand new ideas.
We will pay attention to, for instance, the following features.
\begin{itemize}
\item[-] In our setting, the freedom of conformal transform is limited on the boundary $M$ and one cannot use the standard conformal normal coordinate on the whole manifold $\ox$.
    To handle the vertical direction to $M$, we use the geometric assumption \eqref{eq_hyp}
    and examine the first-order partial differential equation satisfied by the geodesic defining function;
\item[-] We largely depend on the extension result of Chang and Gonz\'alez \cite{CG} to analyze solutions.
    Because of the degeneracy of the extended problem \eqref{eq_Yamabe_2},
    it is not simple to study the asymptotic behavior of the Green's function near its singularity;
    see Appendix \ref{subsec_Green} where some of its qualitative properties are obtained.
    Hence, in order to show the decay property of rescaled solutions, we do not use potential analysis,
but iteratively apply the rescaling argument based on the maximum principle;
\item[-] Regularity theory which we require is technically more difficult to deduce than ones for the classical local problems, or even nonlocal problems on the Euclidean space;
\item[-] Because the bubbles have no explicit expression in general (see Subsection \ref{subsec_bubble}), we have to put some extra efforts compute integrals involving them.
\end{itemize}
To reduce overlaps, we will omit the proofs of several intermediate results which closely follow the standard arguments, leaving appropriate references.
Our main concern is to clarify the novelty of the nonlocal problems defined on general conformal infinities.

\medskip
The paper is organized as follows: In Section \ref{sec_pre}, we recall some analytic and geometric tools necessary to investigate the fractional Yamabe problem \eqref{eq_Yamabe}.
In Section \ref{sec_blow}, we introduce some concepts regarding a blowing-up sequence $\{u_m\}_{m \in \mn}$ of solutions to \eqref{eq_Yamabe}
and perform an asymptotic analysis near each blow-up point of $\{u_m\}_{m \in \mn}$.
Section \ref{sec_lin} is devoted to deducing sharp pointwise estimate of $u_m$ near each isolated simple blow-up point.
This allows one to establish the vanishing theorem of the second fundamental form at any isolated simple blow-up point, which is discussed in Section \ref{sec_van}.
Finally, the main theorems are proved in Section \ref{sec_main} with the aid of a local Pohozaev sign condition which guarantees that every blow-up point is isolated simple.
In the appendices, we provide technical results needed in the main body of the proof as well as their proofs.
Firstly, in Appendix \ref{sec_reg}, we present several elliptic regularity results.
Then we study the asymptotic behavior of the Green's functions near its singularity in Appendix \ref{subsec_Green}.
We also derive a fractional B\^ocher's theorem in Appendix \ref{subsec_Bocher}.
Finally, a number of integrals involving the {\it standard bubble} $W_{1,0}$, whose precise definition is given in Subsection \ref{subsec_bubble}, will be computed in Appendix \ref{sec_W_int}.

\bigskip \noindent \textbf{Notations.}

\medskip \noindent - The Einstein convention is adopted throughout the paper. We shall use the indices $1 \le i,\, j,\, k,\, l \le n$.

\medskip \noindent - For any $t \in \mr$, set $t_+ = \max\{t, 0\}$ and $t_- = \max\{-t,0\}$. Clearly, we have $t = t_+ - t_-$.

\medskip \noindent - Let $N = n+1$. Also, for any $x \in \mr^N_+ = \{(x_1, \cdots, x_n, x_N) \in \mr^N: x_N > 0\}$,
we denote $\bx = (x_1, \cdots, x_n) \in \mr^n \simeq \pa \mr^N_+$.

\medskip \noindent - For a function $f$ on $\mr^N_+$, we often write $\pa_i f = {\pa f \over \pa x_i}$ and $\pa_N f = {\pa f \over \pa x_N}$.

\medskip \noindent - For any $\bx \in \mr^n$, $x \in \mr^N_+$ and $r > 0$, $B^n(\bx,r)$ signifies the $n$-dimensional ball whose center and radius are $\bx$ and $r$, respectively.
Similarly, we define $B^N_+(x,r)$ by the $N$-dimensional upper half-ball centered at $x$ having radius $r$.
We often identify $B^n(\bx,r) = \pa B^N_+((\bx,0),r) \cap \pa \mr^N_+$.
Set $\pa_I B^N_+((\bx,0),r) = \pa B^N_+((\bx,0),r) \cap \mr^N_+$.

\medskip \noindent - $|\ms^{n-1}|$ is the surface area of the unit $(n-1)$-sphere $\ms^{n-1}$.

\medskip \noindent - The spaces $W^{1,2}\mrg$ and $D^{1,2}\mrg$ are the completions of $C_c^{\infty}(\overline{\mr^N_+})$ with respect to the norms
\begin{align*}
\|U\|_{W^{1,2}\mrg} &= \( \int_{\mr^N_+} x_N^{1-2\ga} \(|\nabla U|^2 + U^2\) dx \)^{1 \over 2}
\intertext{and}
\|U\|_{D^{1,2}\mrg} &= \( \int_{\mr^N_+} x_N^{1-2\ga} |\nabla U|^2 dx \)^{1 \over 2},
\end{align*}
respectively.
The natural function space $W^{1,2}(X;\rho^{1-2\ga})$ for the fractional Yamabe problem \eqref{eq_Yamabe_2} is analogously defined.

\medskip \noindent - For any number $\beta \in (0, \infty) \setminus \mn$ and domain $\Omega$, we write $C^{\beta}(\Omega)$
to refer the H\"older space $C^{\lfloor \beta \rfloor, \beta - \lfloor \beta \rfloor}(\Omega)$ where $\lfloor \beta \rfloor$ is the greatest integer that does not exceed $\beta$.

\medskip \noindent - Assume that $(M, \bh)$ and $(\ox, \bg)$ are compact Riemannian manifolds.
Then $B_{\bh}(y, r) \subset (M, \bh)$ stands for the geodesic ball centered at $y \in M$ of radius $r > 0$.
Besides, $dv_{\bg}$ is the volume form of $(\ox, \bg)$ and $d\sigma$ represents a surface measure.

\medskip \noindent - $C > 0$ denotes a generic constant possibly depending on the dimension $n$ of an underlying manifold $M$,
the order $\ga$ of the conformal fractional Laplacian $P^{\ga}$ and so on. It may vary from line to line.
Moreover, a notation $C(\alpha, \beta, \cdots)$ means that the constant $C$ depends on $\alpha, \beta, \cdots$.

\begin{rmk}
By \cite[Lemma 3.1]{FF} (or we may follow the proof of \cite[Proposition 2.1.1]{DMV} by replacing \cite[Theorem 1.2]{FKS} with \cite[Lemma 2.2]{TX}) and \cite[Section 3.2]{CaS}, we have
\begin{equation}\label{eq_emb}
D^{1,2}\mrg \hookrightarrow L^{2(n-2\ga+2) \over n-2\ga}\mrg \quad \text{and} \quad D^{1,2}\mrg \hookrightarrow H^{\ga}(\mr^n).
\end{equation}
Hence we infer from \cite[Corollary 7.2]{DPV} that the trace embedding
$D^{1,2}\mrg \hookrightarrow L^q(\Omega)$ is compact for any $q \in [1, 2_{n,\ga}^*)$ and a smooth bounded domain $\Omega \subset \mr^n$.
\end{rmk}

\section{Preliminaries}\label{sec_pre}
\subsection{Geometric background}
We recall the extension result involving the conformal fractional Laplacian $P^{\ga}$ obtained by Chang and Gonz\'alez \cite{CG}; see also \cite{CaS, GQ}.
\begin{prop} \label{prop_CG}
Suppose that $\ga \in (0,1)$, $n > 2\ga$, $(X, g^+)$ is an asymptotically hyperbolic manifold with conformal infinity $(M, [\bh])$.
Also, assume that $\rho$ is a geodesic defining function associated to $M$, $\bg = \rho^2 g^+$ and the mean curvature $H$ is $0$ on $M$. Set $s = n/2 + \ga$ and
\[E_{\bg}(\rho) = \rho^{-1-s} (-\Delta_{g^+}-s(n-s))\rho^{n-s} \quad \text{in } X.\]
Then we have
\begin{equation}\label{eq_E_exp}
\begin{aligned}
E_{\bg}(\rho) &= \({n-2\ga \over 4n}\) \left[ R[\bg] - (n(n+1)+R[g^+]) \rho^{-2} \right] \rho^{1-2\ga} \\
&= -\({n-2\ga \over 2}\) \({\pa_{\rho} \sqrt{|\bg|} \over \sqrt{|\bg|}}\) x_N^{-2\ga} \quad \text{(by } \textnormal{(2.5)} \text{ of \cite{KMW2})}
\end{aligned}
\end{equation}
in $M \times (0,r_0)$ for some small $r_0 > 0$, where $R[\bg]$ and $R[g^+]$ are the scalar curvatures of $(\ox, \bg)$ and $(X, g^+)$, respectively,
and $|\bg|$ is the determinant of $\bg$. In addition,

\medskip \noindent $(1)$ If a positive function $U \in W^{1,2}(X;\rho^{1-2\ga})$ satisfies
\begin{equation}\label{eq_ext}
\begin{cases}
-\textnormal{div}_{\bg} (\rho^{1-2\ga}\nabla U) + E_{\bg}(\rho)U = 0 &\text{in } (X, \bg),\\
U = u &\text{on } M,
\end{cases}
\end{equation}
then
\begin{equation}\label{eq_diff}
\pa^{\ga}_{\nu} U = - \kappa_{\ga} \(\lim_{\rho \to 0+} \rho^{1-2\ga}{\pa U \over \pa \rho}\) = P^{\ga}_{\bh} u \quad \text{on } M,
\quad \kappa_{\ga} = {2^{-(1-2\ga)} \Gamma(\ga) \over \Gamma(1-\ga)} > 0,
\end{equation}
where $\nu$ denotes the outward unit normal vector with respect to $X$ and $\Gamma(z)$ is the Gamma function.

\medskip \noindent $(2)$ If \eqref{eq_eig} is also true, then there is a special defining function $\rho^*$ such that $E_{\bg^*}(\rho^*) = 0$ in $X$ and $\rho^*(\rho) = \rho\, (1 + O(\rho^{2\ga}))$ near $M$.
Besides the function $U^* = (\rho/\rho^*)^{(n-2\ga)/2}U$ solves
\[\begin{cases}
-\textnormal{div}_{\bg^*} (\(\rho^*\)^{1-2\ga}\nabla U^*) = 0 &\text{in } (X, \bg^*),\\
\pa^{\ga}_{\nu} U^* = P^{\ga}_{\bh} u - Q^{\ga}_{\bh} u &\text{on } M.
\end{cases}\]
Here $\bg^* = (\rho^*)^2 g^+$ and $Q^{\ga}_{\bh} = P^{\ga}_{\bh}(1)$ are called the adapted metric on $\ox$ and the fractional scalar curvature on $(M, \bh)$, respectively.
\end{prop}
\noindent Note that the condition $\Lambda^{\ga}(M, [\bh]) > 0$ (see \eqref{eq_Y_cons}) implies that the functional
\[J^{\ga}(U) = \int_X \(\rho^{1-2\ga} |\nabla U|^2_{\bg} + E_{\bg}(\rho) U^2\) dv_{\bg} \quad \text{for } U \in W^{1,2}(X;\rho^{1-2\ga})\]
is coercive, that is, there exists $C > 0$ independent of $U$ such that $J^{\ga}(U) \ge C \|U\|_{W^{1,2}(X;\rho^{1-2\ga})}^2$.
See \cite[Lemma 2.5]{DKP} for its proof.
Therefore, given any $u \in H^{\ga}(M)$, the standard minimization argument guarantees the existence and uniqueness of the extension $U \in W^{1,2}(X;\rho^{1-2\ga})$ of $u$ which satisfies \eqref{eq_ext}.
Furthermore, testing $u_-$ in \eqref{eq_ext}, we easily observe that if $u \ge 0$ on $M$, then $U \ge 0$ in $X$.
If it holds that $u > 0$ on $M$, then the strong maximum principle for (non-degenerate) elliptic operators gives $U > 0$ on $\ox$.

On the other hand, without loss of generality, we can always assume that the constant $c > 0$ in equation \eqref{eq_Yamabe} is exactly 1.
As a result, \eqref{eq_Yamabe} is equivalent to the degenerate elliptic problem
\begin{equation}\label{eq_Yamabe_2}
\begin{cases}
-\textnormal{div}_{\bg} (\rho^{1-2\ga}\nabla U) + E_{\bg}(\rho)U = 0 &\text{in } (X, \bg),\\
U > 0 &\text{on } \ox,\\
U = u &\text{on } M,\\
\pa^{\ga}_{\nu} U = u^p &\text{on } M.
\end{cases}
\end{equation}

\medskip
Next, choose any $y \in M$ and identify it with $0 \in \mr^n$.
Also, let $x = (\bx, x_N) \in \mr_+^N$ be {\it Fermi coordinates} on $\ox$ around $y$, i.e., $\bx = (x_1, \cdots, x_n)$ normal coordinates on $M$ at $y$ and $x_N = \rho$.
In \cite[Lemma 3.1]{Es}, the following expansion of the metric $\bg$ near $y$ is given.
\begin{lemma}\label{lemma_metric}
In terms of Fermi coordinates $x$ on $\ox$ around $y \in M$,
\[\sqrt{|\bg|}(x) = 1 - nH x_N + {1 \over 2} \(n^2H^2 - \|\pi\|^2 - R_{NN}[\bg]\) x_N^2 - nH_{,i}x_ix_N - {1 \over 6} R_{ij}[\bh] x_ix_j + O(|x|^3)\]
and
\[\bg^{ij}(x) = \delta_{ij} + 2 \pi_{ij} x_N + {1 \over 3} R_{ikjl}[\bh] x_kx_l + \bg^{ij}_{\phantom{ij},Nk} x_Nx_k + (3\pi_{ik}\pi_{kj} + R_{iNjN}[\bg]) x_N^2 + O(|x|^3).\]
Here
\begin{enumerate}
\item[-] $\delta_{ij}$ is the Kronecker delta;
\item[-] $\|\pi\|^2 = \bh^{ik}\bh^{jl}\pi_{ij}\pi_{kl}$ is the square of the norm of the second fundamental form $\pi$;
\item[-] $R_{ikjl}[\bh]$ is a component of the Riemannian curvature tensor on $M$ and $R_{iNjN}[\bg]$ is that of the Riemannian curvature tensor on $\ox$;
\item[-] $R_{ij}[\bh] = R_{ikjk}[\bh]$ and $R_{NN}[\bg] = R_{NiNi}[\bg]$.
\end{enumerate}
Every tensor in the expansion is evaluated at $y = 0$ and commas denote partial differentiation.
\end{lemma}

\noindent If $H = 0$ on $M$, which is the case when \eqref{eq_hyp} holds, all the coefficients of the terms $x_N$, $x_i x_N$, $x_i x_j x_N$, etc. in the expansion of $\sqrt{|\bg|}$ are 0.
In particular, condition \eqref{eq_sqrt_bg} holds.

\medskip
As pointed out in papers on the fractional Yamabe problem \cite{GQ, GW, KMW2}, only the boundary metric can be controlled through conformal changes.
It is one of the main differences compared to the boundary Yamabe problem treated in e.g. \cite{Es, Ma2, Ch, MN, Al}.
The following lemma is a reformulation of Lemmas 2.4 and 3.2 in \cite{KMW2}.
\begin{lemma}\label{lemma_rep}
Let $(X, g^+)$ be an asymptotically hyperbolic manifold such that \eqref{eq_hyp} holds.
Then the conformal infinity $(M, [\bh])$ admits a representative $\tih \in [\bh]$, its corresponding geodesic boundary defining function $\tirh$ and the metric $\tig = \tirh^2 g^+$ such that
\begin{itemize}
\item[(1)] $R_{ij}[\tih](y) = R_{ij;k}[\tih](y) + R_{jk;i}[\tih](y) + R_{ki;j}[\tih](y) = 0$,
\item[(2)] $H = 0$ on $M$ and $R_{\tirh\tirh}[\tig](y) = \dfrac{1-2n}{2(n-1)} \|\pi(y)\|^2$
\end{itemize}
for a fixed point $y \in M$. Here the semicolon designates the covariant differentiation.
\end{lemma}

\subsection{Definition and properties of bubbles}\label{subsec_bubble}
Suppose that $\ga \in (0,1)$ and $n > 2\ga$.
For arbitrary $\lambda > 0$ and $\sigma \in \mr^n$, let $w_\ls$ be the {\it bubble} defined as
\begin{equation}\label{eq_bubble}
w_\ls(\bx) = \alpha_{n,\ga} \({\lambda \over \lambda^2 + |\bx - \sigma|^2}\)^{n-2\ga \over 2} \quad \text{for } \bx \in \mr^n,
\quad \alpha_{n,\ga} = 2^{n-2\ga \over 2} \({\Gamma\({n+2\ga \over 2}\) \over \Gamma\({n-2\ga \over 2}\)}\)^{n-2\ga \over 4\ga}.
\end{equation}
We also introduce the {\it $\ga$-harmonic extension} $W_\ls$ of $w_\ls$, namely, the unique solution of
\begin{equation}\label{eq_bubble_2}
\begin{cases}
-\text{div} (x_N^{1-2\ga} \nabla W_\ls) = 0 &\text{in } \mr^N_+,\\
W_\ls = w_\ls &\text{on } \mr^n.
\end{cases}
\end{equation}
Then it is known that
\begin{equation}\label{eq_CS}
\pa^{\ga}_{\nu} W_\ls = - \kappa_{\ga} \(\lim\limits_{x_N \to 0+} x_N^{1-2\ga} \pa_N W_\ls\)
\underset{\text{(by \cite{CaS})}}{=} (-\Delta)^{\ga} w_\ls  = w_\ls^{n+2\ga \over n-2\ga} \quad \text{on } \mr^n
\end{equation}
where $\kappa_{\ga} > 0$ is the number appeared in \eqref{eq_diff} and $\nu$ is the outward unit normal vector with respect to $\mr^N_+$.
\begin{lemma}\label{lemma_W}
\textnormal{(1) [Symmetry]} The value of $W_{1,0}(\bx, x_N)$ for $(\bx, x_N) \in \mr^N_+$ is governed by $|\bx|$ and $x_N$.
In particular, $\pa_i W_{1,0}(\bx, x_N) = - \pa_i W_{1,0}(-\bx, x_N)$ for each $1 \le i \le n$.

\medskip \noindent \textnormal{(2) [Decay]} There exists a constant $C > 0$ depending only on $n$, $\ga$ and $\ell$ such that
\begin{equation}\label{eq_W_dec}
\left| \nabla_{\bx}^{\ell} W_{1,0}(x) \right| \le {C \over 1+|x|^{n-2\ga+\ell}}
\quad \text{and} \quad
\left| x_N^{1-2\ga} \pa_N W_{1,0}(x) \right| \le {C \over 1+|x|^n}
\end{equation}
for all $x \in \mr^N_+$ and $\ell \in \mn \cup \{0\}$.

\medskip \noindent \textnormal{(3) [Classification]} Suppose that $\Phi \in W^{1,2}_{\loc}\mrg$ is a nontrivial solution of
\[\begin{cases}
-\textnormal{div} (x_N^{1-2\ga} \nabla \Phi) = 0 &\text{in } \mr^N_+,\\
\Phi \ge 0 &\text{in } \mr^N_+,\\
\pa^{\ga}_{\nu} \Phi = \Phi^p &\text{on } \mr^n.
\end{cases}\]
Then there is no such a function $\Phi$ for $1 \le p < 2_{n,\ga}^*-1$.
If $p = 2_{n,\ga}^*-1$, then $\Phi(\bx, 0) = w_\ls(\bx)$ for all $\bx \in \mr^n$ and some $(\ls) \in (0,\infty) \times \mr^n$.
Particularly, $\Phi(x) = W_\ls(x)$ for all $x \in \mr^N_+$.

\medskip \noindent \textnormal{(4) [Nondegeneracy]} The solution space of the linear problem
\[\begin{cases}
-\textnormal{div} (x_N^{1-2\ga} \nabla \Phi) = 0 &\text{in } \mr^N_+,\\
\pa^{\ga}_{\nu} \Phi = \({n+2\ga \over n-2\ga}\) w_\ls^{4\ga \over n-2\ga} \Phi &\text{on } \mr^n,\\
\|\Phi(\cdot, 0)\|_{L^{\infty}(\mr^n)} < \infty
\end{cases}\]
is spanned by
\begin{equation}\label{eq_Z}
Z_{\ls}^1 = {\pa W_\ls \over \pa \sigma_1},\ \cdots,\ Z_{\ls}^n = {\pa W_\ls \over \pa \sigma_n} \quad \text{and} \quad Z_{\ls}^0 = - {\pa W_\ls \over \pa \lambda}.
\end{equation}
\end{lemma}
\begin{proof}
Since $w_{1,0}(\bx)$ depends only on $|\bx|$, claim (1) follows from the uniqueness of $\ga$-extension.
Moreover the sharp decay estimate \cite[Section A]{KMW2} for $W_{1,0}$ gives (2).
Assertions (3) and (4) are implied by the results of Jin et al. \cite[Theorem 1.8, Remark 1.9]{JLX} and D\'avila et al. \cite{DDS}, respectively.
\end{proof}

\subsection{Modification of equation \eqref{eq_Yamabe_2}} \label{subsec_Y_31}
Suppose that $(X, g^+)$ is an asymptotically hyperbolic manifold with conformal infinity $(M, [\bh])$.
Consider sequences of parameters $\{p_m\}_{m \in \mn} \subset [1+\vep_0, 2_{n,\ga}^*-1]$ for any fixed $\vep_0 > 0$, metrics $\{\bh_m\}_{m \in \mn} \subset [\bh]$ on $M$,
corresponding geodesic boundary defining function $\{\rho_m\}_{m \in \mn}$ and positive functions $\{f_m\}_{m \in \mn}$ on $M$.
Assume that $p_m \to p_0$, $\bg_m = \rho_m^2 g^+ \to \bg_0$ in $C^4(\ox, \mr^{N \times N})$ for a metric $\bg_0$ on $\ox$ and $f_m \to f_0 > 0$ in $C^2(M)$ as $m \to \infty$.
It turns out that it is convenient to deal with the following form of the equation
\begin{equation}\label{eq_Yamabe_3}
\begin{cases}
-\text{div}_{\bg_m} (\rho_m^{1-2\ga}\nabla U_m) + E_{\bg_m}(\rho_m) U_m = 0 &\text{in } (X, \bg_m),\\
U_m > 0 &\text{on } \ox,\\
U_m = u_m &\text{on } M,\\
\pa^{\ga}_{\nu} U_m = f_m^{-\delta_m} u_m^{p_m} &\text{on } M
\end{cases}
\end{equation}
rather than \eqref{eq_Yamabe_2}. Here $\delta_m = (2_{n,\ga}^*-1) - p_m \ge 0$.

\medskip
Suppose that $\tih_m = w_m^{4/(n-2\ga)} \bh_m$ on $M$ for a positive function $w_m$ on $M$ such that
\begin{equation}\label{eq_w_m_2}
w_m(y_m) = 1, \quad {\pa w_m \over \pa x_i}(y_m) = 0 \quad \text{for each } i = 1, \cdots, n.
\end{equation}
If $\tirh_m$ is the geodesic boundary defining function associated to $\tih_m$ and $\tig_m = \tirh_m^2 g^+$,
then $\tig_m = (\tirh_m/\rho_m)^2 \bg_m$ on $\ox$ and so $w_m = (\tirh_m/\rho_m)^{(n-2\ga)/2}$ on $M$.
Furthermore, a direct computation using \cite[Lemma 4.1]{CG} yields that $\wtu_m = (\rho_m/\tirh_m)^{(n-2\ga)/2} U_m$ solves
\begin{equation}\label{eq_Yamabe_30}
\begin{cases}
-\text{div}_{\tig_m} (\tirh_m^{1-2\ga}\nabla \wtu_m) + E_{\tig_m}(\tirh_m) \wtu_m = 0 &\text{in } (X, \tig_m),\\
\wtu_m > 0 &\text{on } \ox,\\
\wtu_m = \tiu_m &\text{on } M,\\
\pa^{\ga}_{\nu} \wtu_m = w_m^{-{n+2\ga \over n-2\ga}} \pa^{\ga}_{\nu} U_m = \tif_m^{-\delta_m} \tiu_m^{p_m} &\text{on } M
\end{cases}
\end{equation}
where $\tif_m = w_m f_m$, which is the same form as that of \eqref{eq_Yamabe_3}.

\subsection{Pohozaev's identity} \label{subsec_poho}
Pick a small number $r_1 \in (0, r_0)$ (see \eqref{eq_E_exp})
such that $\bg_m$-Fermi coordinate centered at $y \in M$
is well-defined in the closed geodesic half-ball $\overline{B^N_+(y,r_1)} \subset \ox$ for every $m \in \mn$ and $y \in M$.

In this subsection, we provide a local version of Pohozaev's identity for
\begin{equation}\label{eq_Yamabe_31}
\begin{cases}
-\text{div} (x_N^{1-2\ga}\nabla U) = x_N^{1-2\ga} Q &\text{in } B^N_+(0,r_1) \subset \mr^N_+,\\
U = u > 0 &\text{on } B^n(0,r_1) \subset \mr^n,\\
\pa^{\ga}_{\nu} U = f^{-\delta} u^p &\text{on } B^n(0,r_1)
\end{cases}
\end{equation}
where $p \in [1, 2_{n,\ga}^*-1]$, $Q \in L^{\infty}(B^N_+(0,r_1))$ and $f \in C^1(B^n(0,r_1))$.
\begin{lemma}\label{lemma_poho}
Let $U \in W^{1,2}(B^N_+(0,r_1); x_N^{1-2\ga})$ be a solution to \eqref{eq_Yamabe_31} such that $U$, $\pa_i U$ and $x_N^{1-2\ga} \pa_N U$ are H\"older continuous on $\overline{B^N_+(0,r_1)}$.
Given any number $r \in (0, r_1)$, we define
\begin{multline}\label{eq_poho_1}
\mcp(U, r) = \kappa_{\ga} \int_{\pa_I B^N_+(0,r)} x_N^{1-2\ga} \left[\({n-2\ga \over 2}\) u {\pa u \over \pa r} - {r \over 2} |\nabla u|^2
+ r \left| {\pa u \over \pa r} \right|^2 \right] d\sigma_{x} \\
+ {r \over p+1} \int_{\pa B^n(0,r)} f^{-\delta} u^{p+1} d\sigma_{\bx}
\end{multline}
where $\kappa_{\ga} > 0$ is the constant in \eqref{eq_diff}. Then we have
\begin{multline*}
\mcp(U, r) = - \kappa_{\ga} \int_{B^N_+(0,r)} x_N^{1-2\ga} Q \cdot \left[x_i \pa_i U + x_N \pa_N U + \({n-2\ga \over 2}\) U \right] dx \\
- {\delta \over p+1} \int_{B^n(0,r)} x_i \pa_if f^{-(\delta+1)} u^{p+1} d\bx + \({n \over p+1} - {n-2\ga \over 2}\) \int_{B^n(0,r)} f^{-\delta} u^{p+1} d\bx
\end{multline*}
for all $r \in (0, r_1)$.
\end{lemma}
\begin{proof}
The proof is similar to that of \cite[Proposition 4.7]{JLX}.
\end{proof}

\section{Basic properties of blow-up points}\label{sec_blow}
\subsection{Various types of blow-up points}\label{subsec_blow}
We start this section by recalling the notion of blow-up, isolated blow-up and isolated simple blow-up.
Our definition is a slight modification of the one introduced in \cite[Section 4]{Al} (cf. \cite{JLX, KMS, LX}).
\begin{defn}\label{def_blow}
As before, let $(X, g^+)$ be an asymptotically hyperbolic manifold with conformal infinity $(M, [\bh])$.
Here we use the notations in Subsection \ref{subsec_Y_31} and the small number $r_1 > 0$ picked in Subsection \ref{subsec_poho}.

\medskip \noindent (1) $y_0 \in M$ is called a {\it blow-up point} of $\{U_m\}_{m \in \mn} \subset W^{1,2}(X;\rho^{1-2\ga})$ if there exists a sequence of points $\{y_m\}_{m \in \mn} \subset M$
such that $y_m$ is a local maximum of $u_m = U_m|_M$ satisfying that $u_m(y_m) \to \infty$ and $y_m \to y_0$ as $m \to \infty$.
For simplicity, we will often say that $y_m \to y_0 \in M$ is a blow-up point of $\{U_m\}_{m \in \mn}$.

\medskip \noindent (2) $y_0 \in M$ is an {\it isolated blow-up point} of $\{U_m\}_{m \in \mn}$ if $y_0$ is a blow-up point such that
\begin{equation}\label{eq_blow_1}
u_m(y) \le C d_{\bh_m}(y, y_m)^{-{2\ga \over p_m-1}} \quad \text{for any } y \in M \setminus \{y_m\},\ d_{\bh_m}(y, y_m) < r_2
\end{equation}
for some $C > 0$, $r_2 \in (0, r_1]$ where $\bh_m = \bg_m|_{TM}$ and $d_{\bh_m}$ is the distance function in the metric $\bh_m$.

\medskip \noindent (3) Define a weighted spherical average
\begin{equation}\label{eq_bu}
\bu_m(r) = r^{2\ga \over p_m-1} \({\int_{\pa B^n(y_m, r)} u_m\, d\sigma_{\bh_m} \over \int_{\pa B^n(y_m, r)} d\sigma_{\bh_m}}\), \quad r \in (0, r_1)
\end{equation}
of $u_m$.
We say that an isolated blow-up point $y_0$ of $\{U_m\}_{m \in \mn}$ is {\it simple} if there exists a number $r_3 \in (0, r_2]$
such that $\bu_m$ possesses exactly one critical point in the interval $(0, r_3)$ for large $m \in \mn$.
\end{defn}
\noindent Roughly speaking, Item (2) (or (3), respectively) in the above definition depicts the situation
when {\it clustering of bubbles} (or {\it bubble towers}, respectively) is excluded among various blow-up scenarios.

\medskip
Hereafter, we always assume that $\{(u_m, y_m)\}_{m \in \mn}$ is a sequence of pairs in $C^{\infty}(M) \times M$ such that
$u_m$ is a solution to equation \eqref{eq_Yamabe} with $c = 1$, $y_m$ is a local maximum point of $u_m$ satisfying $u_m(y_m) \to \infty$ and $y_m \to y_0 \in M$ as $m \to \infty$.
Then $y_m \to y_0 \in M$ becomes a blow-up point of a sequence of functions $\{U_m\}_{m \in \mn} \subset W^{1,2}(X;\rho^{1-2\ga})$
where each $U_m$ is a solution to \eqref{eq_Yamabe_3} with $\bg_m = \bg$, $\bh_m = \bh$, $\rho_m = \rho$ and $f_m = 1$.
Let us set $M_m = u_m(y_m)$ and $\ep_m = M_m^{-{(p_m-1)/(2\ga)}}$ for each $m \in \mn$.
Obviously, $M_m \to \infty$ and $\ep_m \to 0$ as $m \to \infty$.

Also, denote by $\tih_m$ a representative of the class $[\bh_m]$ satisfying properties (1) and (2) in Lemma \ref{lemma_rep} with $y = y_m$,
and by $\wtu_m$ a solution to \eqref{eq_Yamabe_30}.
We shall often use $x \in \mr^N_+$ to denote $\tig_m$-Fermi coordinates on $\ox$ around $y_m$
so that $\wtu_m$ can be regarded as a function in $\mr^N_+$ near the origin.

\subsection{Blow-up analysis}
We study asymptotic behavior of a sequence of solutions $\{U_m\}_{m \in \mn}$ to \eqref{eq_Yamabe_3} near blow-up points.

\begin{prop}\label{prop_blow_a}
Assume that $p \in [1+\vep_0, 2_{n,\ga}^*-1]$.
For arbitrary small $\vep_1 > 0$ and large $R > 0$, there are constants $C_0,\, C_1 > 0$ depending only on $(X, g^+), \bh,\, n,\, \ga,\, \vep_0,\, \vep_1$ and $R$ such that
if $U \in W^{1,2}(X;\rho^{1-2\ga})$ is a solution to \eqref{eq_Yamabe_2} with the property that $\max_M U \ge C_0$,
then $(2_{n,\ga}^*-1) - p < \vep_1$ and $U|_M$ possesses local maxima $y_1,\, \cdots y_{\mcn} \in M$ for some $1 \le \mcn = \mcn(U) \in \mn$, for which the following statements hold:

\medskip \noindent \textnormal{(1)} It is valid that
\[\overline{B_{\bh}(y_{m_1}, \hr_{m_1})} \cap \overline{B_{\bh}(y_{m_2}, \hr_{m_2})} = \emptyset \quad \text{for } 1 \le m_1 \ne m_2 \le \mcn\]
where $\hr_m = R\, \alpha_{n,\ga}^{(p-1)/(2\ga)} u(y_m)^{-{(p-1)/(2\ga)}}$.

\medskip \noindent \textnormal{(2)} For each $m = 1, \cdots, \mcn$ and some $\beta = \beta(N, \ga) \in (0,1)$, we have
\begin{multline}\label{eq_blow_a_1}
\left\| \alpha_{n,\ga} U(y_m)^{-1} U\(\alpha_{n,\ga}^{p-1 \over 2\ga} U(y_m)^{-{p-1 \over 2\ga}} \cdot \) - W_{1,0} \right\|_{C^{\beta}(\overline{B^N_+(0,2R)})}\\
+ \left\| \alpha_{n,\ga}\, u(y_m)^{-1} u\(\alpha_{n,\ga}^{p-1 \over 2\ga} u(y_m)^{-{p-1 \over 2\ga}} \cdot \) - w_{1,0} \right\|_{C^{2+\beta}(\overline{B^n(0,2R)})} \le \vep_1
\end{multline}
in $\bg$-Fermi coordinates centered in $y_m$.

\medskip \noindent \textnormal{(3)} It holds that
\[U(y)\, d_{\bh}(y, \{y_1, \cdots, y_{\mcn}\})^{2\ga \over p-1} \le C_1 \quad \text{for } y \in M.\]
\end{prop}
\begin{proof}
The validity of this proposition comes from the Liouville type theorem \cite[Theorem 1.8]{JLX} (see our Lemma \ref{lemma_W} (3) for its statement) and an induction argument.
See the proof of \cite[Proposition 5.1]{LZhu} for a detailed account.
\end{proof}
We have a remark on \eqref{eq_blow_a_1}:
According to Proposition \ref{prop_reg}, only the $C^{\beta}$-convergence is guaranteed on the closed half-ball $\overline{B^N_+(0,2R)}$.
However, we have the $C^{2+\beta}$-convergence on its bottom $\overline{B^n(0,2R)}$.

\medskip
Lemma \ref{lemma_reg_0} and the standard rescaling argument readily give the annular Harnack inequality around an isolated blow-up point.
\begin{lemma}\label{lemma_har}
Suppose that $y_m \to y_0 \in M$ is an isolated blow-up point of a sequence of solutions $\{U_m\}_{m \in \mn}$ to equation \eqref{eq_Yamabe_3}.
If we are in $\bg_m$-Fermi coordinate system centered at $y_m$, then there exists $C > 0$ independent of $m \in \mn$ and $r > 0$ such that
\[\max_{B^N_+(0,2r) \setminus B^N_+(0,r/2)} U_m \le C \min_{B^N_+(0,2r) \setminus B^N_+(0,r/2)} U_m\]
for any $r \in (0, r_2/3)$ where $r_2 > 0$ is a number defined in Definition \ref{def_blow} (2).
\end{lemma}
\begin{proof}
The proof is similar to that of \cite[Lemma 4.3]{JLX}.
\end{proof}

If $y_m \to y_0 \in M$ is an isolate blow-up point of solutions $\{U_m\}_{m \in \mn}$ to \eqref{eq_Yamabe_3},
Proposition \ref{prop_blow_a} can be extended in the following manner.
\begin{lemma}\label{lemma_conv}
Let $y_m \to y_0 \in M$ be an isolated blow-up point of a sequence of solutions $\{U_m\}_{m \in \mn}$ to \eqref{eq_Yamabe_3} with $f_m > 0$ in $B^n(0,r_2)$.
In addition, suppose that $\{R_m\}_{m \in \mn}$ and $\{\tau_m\}_{m \in \mn}$ are arbitrary sequences of positive numbers such that $R_m \to \infty$ and $\tau_m \to 0$ as $m \to \infty$.
Then $p_m \to 2^*_{n,\ga}-1$, and $\{U_\ell\}_{\ell \in \mn}$ and $\{p_\ell\}_{\ell \in \mn}$
have subsequences $\{U_{\ell_m}\}_{m \in \mn}$ and $\{p_{\ell_m}\}_{m \in \mn}$ such that for some $\beta \in (0,1)$
\begin{equation}\label{eq_conv}
\left\| \hep_{\ell_m}^{2\ga \over p_{\ell_m}-1} U_{\ell_m}\( \hep_{\ell_m} \cdot \) - W_{1,0} \right\|_{C^{\beta}(\overline{B^N_+(0,R_m)})}
+ \left\| \hep_{\ell_m}^{2\ga \over p_{\ell_m}-1} u_{\ell_m}\( \hep_{\ell_m} \cdot \) - w_{1,0} \right\|_{C^{2+\beta}(\overline{B^n(0,R_m)})} \le \tau_m
\end{equation}
in $\bg_m$-Fermi coordinates centered in $y_m$ and $R_m \hep_{\ell_m} \to 0$ as $m \to \infty$.
Here
\[\hep_{\ell_m} = \alpha_{n,\ga}^{p_{\ell_m}-1 \over 2\ga} M_{\ell_m}^{-{p_{\ell_m}-1 \over 2\ga}} = \alpha_{n,\ga}^{p_{\ell_m}-1 \over 2\ga} u_{\ell_m}(y_{\ell_m})^{-{p_{\ell_m}-1 \over 2\ga}}\]
for all $m \in \mn$.
\end{lemma}
\noindent In order to prove it, we first need the following type of the Hopf lemma.
\begin{lemma}\label{lemma_hopf}
Suppose that $\bg$ is a smooth metric on $\overline{B^N_+(0,1)}$, $A \in C^0(\overline{B^N_+(0,1)})$ and $U \in W^{1,2}(B^N_+(0,1); x_N^{1-2\ga})$ is a solution to
\[\begin{cases}
-\textnormal{div}_{\bg} (x_N^{1-2\ga}\nabla U) + x_N^{1-2\ga} A U = 0 &\text{in } B^N_+(0,1),\\
U \ge c_0 > 0 &\text{on } \overline{B^N_+(0,1)}
\end{cases}\]
such that $x_N^{1-2\ga} \pa_N U \in C^0(\overline{B^N_+(0,1)})$.
Assume also that there exist a small number $r > 0$ and a point $\bx_0 \in B^n(0,r) \setminus \overline{B^n(0,r/2)}$  such that $U(\bx_0, 0) = c_0$ and $U(\bx,0) > c_0$ on $\{\bx \in \mr^n: |\bx| = r/2\}$. Then
\[\lim_{x_N \to 0} x_N^{1-2\ga} \pa_N U(\bx_0,x_N) > 0.\]
\end{lemma}
\begin{proof}
Our proof is in the spirit of those in \cite[Theorem 3.5]{GQ} and \cite[Proposition 7.1]{CC}. Let
\[W(x) = x_N^{-(1-2\ga)} \(x_N + \wtc_1 x_N^2\) \(e^{-\wtc_2|\bx|} - e^{-\wtc_2 r}\) \quad \text{in } B^N_+(0,1)\]
with $\wtc_1$, $\wtc_2 > 0$ sufficiently large.
Then there exists a small number $\delta > 0$ such that
\[\begin{cases}
-\textnormal{div}_{\bg} (x_N^{1-2\ga}\nabla W) + x_N^{1-2\ga} A W \le 0 &\text{in } B^N_+(0,1),\\
U - \delta W \ge c_0/2 > 0 &\text{on } \overline{B^N_+(0,1)}.
\end{cases}\]
Therefore
\[\begin{cases}
-\textnormal{div}_{\bg} (x_N^{1-2\ga}\nabla (U - \delta W)) + x_N^{1-2\ga} A_+ (U - \delta W) \ge 0 &\text{in } \Gamma_1,\\
U - \delta W \ge c_0 &\text{on } \pa \Gamma_1
\end{cases}\]
where $\Gamma_1 = (B^n(0,1) \setminus \overline{B^n(0,1/2)}) \times (0,1/2)$.
By the maximum principle, $U-\delta W > c_0$ in $\Gamma_1$. Since $(U-\delta W )(\bx_0,0) = c_0$, the assertion follows.
\end{proof}

\begin{proof}[Proof of Lemma \ref{lemma_conv}]
We set
\[V_m(x) = \hep_m^{2\ga \over p_m-1} U_m\(\hep_m x\) \quad \text{for all } x \in B^N_+(0, r_2 \hep_m^{-1}).\]
Then we infer from \eqref{eq_Yamabe_3} and \eqref{eq_blow_1} that
\begin{equation}\label{eq_whu_2}
\begin{cases}
-\text{div}_{\bg_m(\hep_m \cdot)} (x_N^{1-2\ga} \nabla V_m) + \hep_m^2 x_N^{1-2\ga} A_m(\hep_m \cdot) V_m = 0 &\text{in } B^N_+(0, r_2 \hep_m^{-1}),\\
\pa^{\ga}_{\nu} V_m = f_m^{-\delta_m}(\hep_m \cdot) V_m^{p_m} &\text{on } B^n(0, r_2 \hep_m^{-1})
\end{cases}
\end{equation}
where
\[A_m = x_N^{-(1-2\ga)} E_{\bg_m}(x_N^{1-2\ga})\]
in $\bg_m(\hep_m \cdot)$-coordinate centered at $y_m$, and
\[V_m(\bx,0) \le C |\bx|^{-{2\ga \over p_m-1}} \quad \text{for all } \bx \in B^n(0, r_2 \hep_m^{-1}).\]
Thus Lemma \ref{lemma_har} leads us that
\begin{equation}\label{eq_whu_1}
V_m(x) \le C |x|^{-{2\ga \over p_m-1}} \quad \text{for all } x \in B^N_+(0, r_2 \hep_m^{-1}).
\end{equation}
Also, by Definition \ref{def_blow} (1), it holds that
\begin{equation}\label{eq_whu_3}
V_m(0) = \alpha_{n,\ga},\ \nabla_{\bx} V_m(0) = 0 \quad \text{for all } m \in \mn.
\end{equation}

On the other hand, Lemma \ref{lemma_hopf} implies
\begin{equation}\label{eq_whu_5}
\inf_{x \in \pa_I B^N_+(0,r)} V_m = \inf_{x \in B^N_+(0,r)} V_m \quad \text{for each } r \in (0,1].
\end{equation}
Indeed, since $V_m$ is positive in its domain, we have
\[-\text{div}_{\bg_m(\hep_m \cdot)} (x_N^{1-2\ga} \nabla V_m) + \hep_m^2 x_N^{1-2\ga} (A_m)_+(\hep_m \cdot) V_m \ge 0 \quad \text{in } B^N_+(0,1).\]
Because of the classical maximum principle, $V_m$ does not attain its infimum in the interior of $B^N_+(0,r)$ for any $r \in (0,1]$, unless it is a constant function.
However, it cannot be constant, otherwise we get an absurd relation
\[0 = \pa^{\ga}_{\nu} V_m = f_m^{-\delta_m}(\hep_m \cdot) V_m^{p_m} > 0 \quad \text{on } B^n(0,r).\]
Moreover, the infimum of $V_m$ is not achieved on the bottom $B^n(0,r)$, because the existence of a minimum point $\bx_m \in B^n(0,r)$ of $V_m$
and the Hopf lemma produce a contradictory relation
\[0 > \pa^{\ga}_{\nu} V_m(\bx_m,0) = f_m^{-\delta_m}(\hep_m \bx_m) V_m^{p_m}(\bx_m,0) > 0.\]
Therefore \eqref{eq_whu_5} must be true.

Now one observes from \eqref{eq_whu_3}, \eqref{eq_whu_5} and Lemma \ref{lemma_har} that $V_m(x) \le C$ for $|x| \le 1$. In light of \eqref{eq_whu_1}, it reads
\begin{equation}\label{eq_whu_4}
V_m(x) \le C \quad \text{for } x \in B^N_+(0,r_2 \hep_m^{-1})
\end{equation}
where $C > 0$ is a constant independent of $m \in \mn$.

Accordingly, by making use of \eqref{eq_whu_2}, \eqref{eq_whu_3}, \eqref{eq_whu_4}, \eqref{eq_DNM}, \eqref{eq_har_2}, \eqref{eq_reg_2}, \eqref{eq_Sch} and Lemma \ref{lemma_W} (3),
we deduce the existence of $\beta \in (0,1)$ such that
\[p_m \to 2^*_{n,\ga}-1, \quad V_m \to W_{1,0} \quad \text{in } C^{\beta}_{\loc}(\mr^N_+)
\quad \text{and} \quad
V_m(\cdot, 0) \to w_{1,0} \quad \text{in } C^{2+\beta}_{\loc}(\mr^n)\]
passing to a subsequence. The assertion of the lemma is true.
\end{proof}

Keeping in mind that our proof is not affected by the act of picking a subsequence of $\{U_\ell\}_{\ell \in \mn}$,
we always select $\{R_m\}_{m \in \mn}$ first and then $\{U_{\ell_m}\}_{m \in \mn}$ satisfying \eqref{eq_conv} and $R_m \hep_{\ell_m} \to 0$.
From now on, we write $\{U_m\}_{m \in \mn}$ to denote $\{U_{\ell_m}\}_{m \in \mn}$ and so on to simplify notations.

The next result is a simple consequence of the previous lemma under the choice $\tau_m = w_{1,0}(R_m)/2$.
\begin{cor}\label{cor_i_blow}
(1) Suppose that $y_m \to y_0 \in M$ is an isolated blow-up point of a sequence $\{U_m\}_{m \in \mn}$ of solutions to \eqref{eq_Yamabe_3}.
If $\{\wtu_m\}_{m \in \mn}$ is a sequence of solutions to \eqref{eq_Yamabe_30} constructed as in Subsection \ref{subsec_blow},
then $y_m \to y_0 \in M$ is an isolated blow-up point of $\{\wtu_m\}_{m \in \mn}$.

\medskip \noindent (2) Assume that $y_0 \in M$ is an isolated blow-up point of $\{U_m\}_{m \in \mn}$.
Then the function $\bu_m$ defined in \eqref{eq_bu} possesses exactly one critical point in the interval $(0, R_m \hep_m)$ for large $m \in \mn$.
In particular, if the isolated blow-up point $y_0 \in M$ of $\{U_m\}_{m \in \mn}$ is also simple, then $\bu_m'(r) < 0$ for all $r \in [R_m \hep_m, r_3)$; see Definition \ref{def_blow} (3).
\end{cor}
\begin{proof}
Choose numbers $l > 0$ and $L > 0$ such that $l \le w_m^{-1} = \tiu_m/u_m \le L$ on $M$ for all $m \in \mn$.
If we use the normal coordinate on $M$ at $y_m$, it follows from \eqref{eq_conv} that
\begin{equation}\label{eq_tiu_m}
l M_m \(w_{1,0}(\hep_m^{-1} \cdot) - \tau_m\) \le \alpha_{n,\ga}\, \tiu_m \le L M_m \(w_{1,0}(\hep_m^{-1} \cdot) + \tau_m\) \quad \text{in } \overline{B^n(0, R_m \hep_m)}.
\end{equation}
Hence there exists a sufficiently large number $R > 0$ independent of $m \in \mn$ such that $R < R_m$ and $\tiu_m(\cdot) \le \tiu_m(0)/2$ on $\pa B^n(0, R \hep_m)$ for each large $m \in \mn$,
from which we infer that $\tiu_m$ has a local maximum $\tiy_m$ on $M$ satisfying $d_{\bh_m}(y_m, \tiy_m) = |\tiy_m| \le R \hep_m$.
Now it is easy to check that $\tiy_m \to y_0$ is a blow-up point of $\{\wtu_m\}_{m \in \mn}$.
Furthermore, since \eqref{eq_tiu_m} implies the existence of a constant $C > 0$ depending on $R > 0$ such that
\[\tiu_m \le C M_m \le C (|y| + |\tiy_m|)^{-{2\ga \over p_m-1}} \le C |y-\tiy_m|^{-{2\ga \over p_m-1}} \quad \text{if } |y| \le R \hep_m,\]
one finds
\[\tiu_m(y) \le C d_{\tih_m}(y, \tiy_m)^{-{2\ga \over p_m-1}} \quad \text{for any } y \in M \setminus \{\tiy_m\},\ d_{\tih_m}(y, \tiy_m) < r_2\]
where the magnitude of $r_2 > 0$ may be reduced if necessary.
As a result, an application of the proof of Lemma \ref{lemma_conv} to $\tiu_m$ shows that
for $R' \gg R$ large enough, $\tiu_m$ is $C^2(\overline{B_{\tih_m}(\tiy_m, R' \hep_m)})$-close to a suitable rescaling of the standard bubble $w_{1,0}$
so that it has the unique critical point on $B_{\tih_m}(\tiy_m, R' \hep_m)$, i.e., the local maximum $\tiy_m$.
However, by \eqref{eq_w_m_2}, $y_m \in B_{\tih_m}(\tiy_m, R' \hep_m)$ is already a critical point of $\tiu_m$, and so it should be equal to $\tiy_m$.
This completes the proof of (1). The verification of (2) is plain.
\end{proof}

\subsection{Isolated simple blow-up points}
Let $y_m \to y_0 \in M$ be an isolated simple blow-up point of $\{U_m\}_{m \in \mn}$.
By Corollary \ref{cor_i_blow} (1), $y_m \to y_0$ is an isolated blow-up point of $\{\wtu_m\}_{m \in \mn}$.
The objective of this subsection is to show that the behavior of each $\wtu_m$ in the geodesic half-ball $B_{\tig_m}(y_m, r) \cap \ox$ can be controlled whenever $r > 0$ is chosen to be sufficiently small.
We will use $\tig_m$-Fermi coordinate system centered at $y_m$, so $B_{\tig_m}(y_m, r) \cap X$ is identified with $B^N_+(0,r) \subset \mr^N_+$.

\begin{prop}\label{prop_iso}
Assume $n > 2 + 2\ga$ and $y_m \to y_0 \in M$ is an isolated simple blow-up point of a sequence of solutions $\{U_m\}_{m \in \mn}$ to \eqref{eq_Yamabe_3}.
Then one can choose $C > 0$ large and $r_4 \in (0, \min\{r_3, R_0\}]$ small (refer to Definition \ref{def_blow} (3) and Proposition \ref{prop_loc_G}), independent of $m \in \mn$, such that
\begin{equation}\label{eq_iso}
\begin{cases}
M_m \left| \nabla_{\bx}^{\ell} \wtu_m(x) \right| \le C|x|^{-(n-2\ga+\ell)} \ (\ell = 0, 1, 2),\\
M_m \left| x_N^{1-2\ga} \pa_N \wtu_m(x) \right| \le C|x|^{-n}
\end{cases}
\quad \text{for } 0 < |x| \le r_4
\end{equation}
where $M_m = u_m(y_m)$ and $\wtu_m$ is the function constructed in Subsection \ref{subsec_blow}.
\end{prop}
\begin{proof}[Proof of Proposition \ref{prop_iso}]
Its proof consists of six steps. Let $o(1)$ denote any sequence tending to 0 as $m \to \infty$.

\medskip \noindent \textsc{Step 1 (Rough upper decay estimate of $\{\wtu_m\}_{m \in \mn}$).}
For any fixed sufficiently small number $\eta > 0$, set $\lambda_m = (n-2\ga-\eta)(p_m-1)/(2\ga)-1$.
We shall show that there exist a number $r_4' = r_4'(\eta) \in (0, r_3]$ and a large constant $C > 0$ independent of $m \in \mn$ such that
\begin{equation}\label{eq_U_m_est}
M_m^{\lambda_m} \wtu_m(x) \le C |x|^{-(n-2\ga)+\eta}
\quad \text{in } \Gamma_2 = B^N_+(0,r_4') \setminus B^N_+(0, R_m \hep_m).
\end{equation}
It can be proved as in \cite[Lemma 2.7]{FO} and \cite[Lemma 4.6]{JLX}.
However, since this is one of the places where hypothesis \eqref{eq_hyp} is used, we sketch the proof.

Because $\wtu_m \le C U_m$ in $B^N_+(0,r_3)$ for some $C > 0$, it suffices to verify that
\begin{equation}\label{eq_U_m_est_3}
M_m^{\lambda_m} U_m(x) \le C |x|^{-(n-2\ga)+\eta} \quad \text{in } \Gamma_2.
\end{equation}
Thanks to Lemma \ref{lemma_har} and Corollary \ref{cor_i_blow} (2), it turns out that
\begin{equation}\label{eq_U_m^p_m}
U_m^{p_m-1}(x) \le C R_m^{-2\ga+o(1)} |x|^{-2\ga} \quad \text{in } B^N_+(0,r_3) \setminus B^N_+(0, R_m \hep_m).
\end{equation}
Let
\[\begin{cases}
\mfl_m(U) = -\text{div}_{\bg_m} (x_N^{1-2\ga}\nabla U) + E_{\bg_m}(x_N) U &\text{in } B^N_+(0,r_3),\\
\mfb_m(U) = \pa^{\ga}_{\nu} U - f_m^{-\delta_m} u_m^{p_m-1} u &\text{on } B^n(0,r_3)
\end{cases}\]
where $u = U$ on $B^n(0,r_3)$.
By \eqref{eq_Yamabe_3}, it clearly holds that $U_m > 0$, $\mfl_m(U_m) = 0$ in $B^N_+(0,r_3)$ and $\mfb_m(U_m) = 0$ in $B^n(0,r_3)$.

Assume that $0 \le \mu \le n-2\ga$. Then one can calculate
\begin{equation}\label{eq_mfl_1}
\mfl_m(|x|^{-\mu}) = x_N^{1-2\ga} \(\mu(n-2\ga-\mu) + O(|x|)\) |x|^{-(\mu+2)}.
\end{equation}
Moreover, \cite[Lemma 2.3]{KMW2} tells us that \eqref{eq_hyp} is a sufficient condition to ensure $H = 0$ on $M$.
Hence we have $\pa_N \sqrt{|\bg_m|} = O(x_N)$ by Lemma \ref{lemma_metric} and
\begin{align*}
&\ \text{div}_{\bg_m} (x_N^{1-2\ga}\nabla (x_N^{2\ga} |x|^{-(\mu+2\ga)})) - \text{div} (x_N^{1-2\ga}\nabla (x_N^{2\ga} |x|^{-(\mu+2\ga)})) \\
&= x_N^{1-2\ga} \left[O(|x|)\, x_N^{2\ga}\, \pa_{ij} |x|^{-(\mu+2\ga)} + O(|x|)\, x_N^{2\ga}\, \pa_i |x|^{-(\mu+2\ga)} + O(x_N) (\pa_N x_N^{2\ga}) |x|^{-(\mu+2\ga)} \right] \\
&= x_N^{1-2\ga} \left[ O(|x|)\, x_N^{2\ga}\, |x|^{-(\mu+2\ga+2)} \right]
\end{align*}
(cf. \eqref{eq_Q_m}). This implies that
\begin{equation}\label{eq_mfl_2}
\mfl_m(x_N^{2\ga} |x|^{-(\mu+2\ga)}) = x_N^{1-2\ga} \((\mu+2\ga)(n-\mu) + O(|x|)\) |x|^{-(\mu+2)} \({x_N \over |x|}\)^{2\ga}.
\end{equation}
From \eqref{eq_mfl_1}, \eqref{eq_mfl_2} and the computation
\begin{align*}
\mfb_m(|x|^{-\mu} - \zeta x_N^{2\ga} |x|^{-(\mu+2\ga)})
&= |x|^{-(\mu+2\ga)} \left[ 2\ga\kappa_{\ga} \zeta + f_m^{-\delta_m} u_m^{p_m-1} (\zeta x_N^{2\ga} - |x|^{2\ga}) \right] \\
&= |x|^{-(\mu+2\ga)} \left[ 2\ga\kappa_{\ga} \zeta + O(R_m^{-2\ga+o(1)}) \right] \quad (\text{by \eqref{eq_U_m^p_m}})
\end{align*}
in $B^N_+(0,r_3) \setminus B^N_+(0, R_m \hep_m)$ for a fixed $\zeta \in \mr$, we observe that the function
\[\Phi_{1m}(x) = L_m (|x|^{-\eta} - \zeta x_N^{2\ga} |x|^{-(\eta+2\ga)}) + L_0 M_m^{-\lambda_m} (|x|^{-(n-2\ga)+\eta} - \zeta x_N^{2\ga} |x|^{-n+\eta}),\]
with a suitable choice of $L_0,\, L_m > 0$ large and $\zeta,\, \eta > 0$ small, satisfies
\[\mfl_m(\Phi_{1m}) \ge 0 \quad \text{in } \Gamma_2, \quad \mfb_m(\Phi_{1m}) \ge 0 \quad \text{on } \pa \Gamma_2 \cap \mr^n
\quad \text{and} \quad U_m \le \Phi_{1m} \quad \text{on } \pa \Gamma_2 \setminus \mr^n.\]
Consequently, the generalized maximum principle (Lemma \ref{lemma_max}) leads us to conclude $U_m \le \Phi_{1m}$ in $\Gamma_2$.
From this and the assumption that $y_m \to y_0$ is an isolated simple blow-up point of $\{U_m\}_{m \in \mn}$, we infer that \eqref{eq_U_m_est_3} or \eqref{eq_U_m_est} holds.

\medskip \noindent \textsc{Step 2 (Lower decay estimate of $\{\wtu_m\}_{m \in \mn}$).}
We claim that there is a large constant $C > 0$ such that
\begin{equation}\label{eq_iso_2}
M_m \wtu_m(x) \ge C^{-1} |x|^{-(n-2\ga)} \quad \text{in } \Gamma_2
\end{equation}
where the magnitude of $r_4'$ is reduced if necessary.

Let $G_m$ be the Green's function that solves \eqref{eq_Green} provided $\bg = \bg_m$ and $B_R = B^N_+(0, r_4')$.
By \eqref{eq_Yamabe_3}, \eqref{eq_conv} and \eqref{eq_Green_dec}, we find that $U = M_m U_m - C^{-1}G_m$ solves
\[\begin{cases}
-\text{div}_{\bg_m} (x_N^{1-2\ga}\nabla U) + x_N^{1-2\ga} A_m U = 0 &\text{in } \Gamma_2,\\
\pa^{\ga}_{\nu} U = M_m f_m^{-\delta_m} u_m^{p_m} \ge 0 &\text{on } \pa \Gamma_2 \cap \mr^n,\\
U \ge 0 &\text{on } \pa \Gamma_2 \setminus \mr^n
\end{cases}\]
provided that $C > 0$ is large enough.
Hence the weak maximum principle depicted in Remark \ref{rmk_max_p} shows that $U \ge 0$ in $\Gamma_2$.
Inequality \eqref{eq_iso_2} now follows from the inequality $U_m \le C \wtu_m$ in $B^N_+(0,r_4')$ and \eqref{eq_Green_dec}.

\medskip \noindent \textsc{Step 3 (Rough upper decay estimate of derivatives of $\{\wtu_m\}_{m \in \mn}$).} We assert
\begin{equation}\label{eq_U_m_est_2}
\begin{cases}
M_m^{\lambda_m} \left| \nabla_{\bx}^{\ell} \wtu_m(x) \right| \le C \big( 1+\hep_m^{o(1)} \big) |x|^{-(n-2\ga+\ell)+\eta} \ (\ell = 1, 2),\\
M_m^{\lambda_m} \left| x_N^{1-2\ga} \pa_N \wtu_m(x) \right| \le C \big( 1+\hep_m^{o(1)} \big) |x|^{-n+\eta}
\end{cases}
\quad \text{in } \Gamma_2.
\end{equation}

We apply the standard rescaling argument depicted in, e.g., the proof of \cite[Lemma 2.6]{FO2}.
Given any $m \in \mn$ and $R \in [2R_m \hep_m, r_4'/2]$, set
\[\mcu_R(x) = M_m^{\lambda_m} R^{n-2\ga-\eta} \wtu_m(Rx) \quad \text{in } \Gamma_3 = B^N_+(0,2) \setminus B^N_+\(0,\frac{1}{2}\),\]
which solves
\begin{equation}\label{eq_U_m_est_21}
\begin{cases}
-\text{div}_{\bg_m(R\cdot)} (x_N^{1-2\ga} \nabla \mcu_R) + E_{\bg_m(R\cdot)}(x_N)\, \mcu_R = 0 &\text{in } \Gamma_3,\\
\pa^{\ga}_{\nu} \mcu_R = \(\dfrac{\hep_m}{R}\)^{2\ga\lambda_m} \tif_m^{-\delta_m}(R\cdot)\, \mcu_R^{p_m} &\text{on } \pa \Gamma_3' = \pa \Gamma_3 \cap \mr^n.
\end{cases}
\end{equation}
Inequalities \eqref{eq_U_m_est_2} will be valid if there exists $C > 0$ independent of $m$ and $R$ such that
\begin{equation}\label{eq_U_m_est_22}
\sum_{\ell=1}^2 \left| \nabla_{\bx}^{\ell} \mcu_R(x) \right| + \left| x_N^{1-2\ga} \pa_N \mcu_R(x) \right| \le C \big( 1+\hep_m^{o(1)} \big) \quad \text{for } |x| = 1.
\end{equation}
Because of relatively poor regularity property of degenerate elliptic equations,
especially when $\ga \in (0,1)$ is small, derivation of \eqref{eq_U_m_est_22} is rather technical.
In particular, as we will see shortly, it requires the lower estimate \eqref{eq_iso_2} of $M_m^{\lambda_m} \wtu_m$ in contrast to the local case $\ga = 1$.

In light of \eqref{eq_U_m_est}, it holds that $\mcu_R(x) \le C$ in $\Gamma_3$.
Applying the H\"older estimate \eqref{eq_DNM}, a bootstrap argument with the Schauder estimate \eqref{eq_Sch},
and the derivative estimate \eqref{eq_reg_2} to \eqref{eq_U_m_est_21}, we obtain
\begin{equation}\label{eq_U_m_est_23}
\|\nabla_{\bx} \mcu_R\|_{C^0(K)} \le C \left[ 1 + \(\dfrac{\hep_m}{R}\)^{2\ga\lambda_m} \right] \le C
\end{equation}
for any proper compact subset $K$ of $\Gamma_3 \cup \pa \Gamma_3'$.
Furthermore, \eqref{eq_iso_2} yields
\[\(\dfrac{\hep_m}{R}\)^{2\ga\lambda_m} \mcu_R^{p_m-2}(x) \le
\begin{cases}
C &\text{for } n < 6\ga,\\
\(\dfrac{\hep_m}{R}\)^{2\ga\lambda_m-\eta(2-p_m)} \hep_m^{o(1)} \le C \hep_m^{o(1)} &\text{for } n \ge 6\ga
\end{cases}\]
on $\pa \Gamma_3'$.
This together with \eqref{eq_Sch} and \eqref{eq_U_m_est_23} 
gives that
\begin{equation}\label{eq_U_m_est_24}
\begin{aligned}
\|\mcu_R\|_{C^{\beta'}(\wtk')} &\le C \left[ 1 + \(\dfrac{\hep_m}{R}\)^{2\ga\lambda_m}
\(1 + \|\mcu_R\|_{C^{\beta}(\wtk)} + \left\| \mcu_R^{p_m-2} \nabla_{\bx} \mcu_R \right\|_{C^0(\wtk)} \) \right] \\
&\le C \big( 1+\hep_m^{o(1)} \big)
\end{aligned}
\end{equation}
for all proper compact subsets $\wtk' \subsetneq \wtk$ of $\Gamma_3'$ and exponents $1 \le \beta < \beta' \le \min\{2,\beta+2\ga\}$.
The desired inequality \eqref{eq_U_m_est_22} is now derived from \eqref{eq_U_m_est_23}, \eqref{eq_U_m_est_24}, \eqref{eq_reg_2} with $\ell_0 = 2$ and \eqref{eq_reg_3}.

\medskip \noindent \textsc{Step 4 (Estimate of $\delta_m$).}
For $\delta_m = (2_{n,\ga}^*-1) - p_m \ge 0$, it holds that
\begin{equation}\label{eq_delta_m}
\delta_m = O\(M_m^{-{2 \over n-2\ga} + o(1)}\) \quad \text{and} \quad M_m^{\delta_m} \to 1 \quad \text{as } m \to \infty.
\end{equation}
The proof makes use of Pohozaev's identity in Lemma \ref{lemma_poho} and is analogous to that in \cite[Lemma 4.8]{JLX}. Hence we omit it.

\medskip \noindent \textsc{Step 5 (Estimate of $\{\wtu_m\}_{m \in \mn}$ on \{$|x| = r_4$\}).}
We demonstrate
\begin{equation}\label{eq_M_m_U_m}
\max_{|x| = r_4} M_m \wtu_m(x) \le C(r_4)
\end{equation}
for any sufficiently small $r_4 \in (0, r_4']$.

Suppose that it does not hold.
Then there exists a sequence $\{z_m\}_{m \in \mn}$ of points on the half-sphere $\{x \in \mr^N_+: |x| = r_4\}$ such that
\[M_m U_m(z_m) \to \infty \quad \text{as } m \to \infty.\]
Let us write $E_{\bg_m}(x_N) = x_N^{1-2\ga} A_m$ so that $\{A_m\}_{m \in \mn}$ is a family of functions whose $C^2$-norm is uniformly bounded.
We divide the cases according to the sign of $A_m$.

\medskip \noindent {\sc Case 1.} Suppose that $A_m \ge 0$ in $B^N_+(0, r_4')$ for all $m \in \mn$.

In this case, one can argue as in \cite[Proposition 4.5]{Ma} 
or \cite[Proposition 4.3]{Al} 
to reach a contradiction. The proof is omitted.

\medskip \noindent {\sc Case 2.} Assume the condition that $A_m \ge 0$ in $B^N_+(0, r_4')$ is violated for some $m \in \mn$.

To handle this situation, we will recover positivity of $A_m$ by employing conformal change of the metric $\bg_m$ on $M$.
Owing to \eqref{eq_hyp} (or the condition $H = 0$ on $M$), \eqref{eq_E_exp} and Lemma \ref{lemma_metric}, it holds that
\begin{align*}
x_N^{1-2\ga} A_m = E_{\bg_m}(x_N)
&\ge \({n-2\ga \over 2}\) x_N^{1-2\ga} \left[ 1 - 2r_4'\left\| \nabla \sqrt{|\bg_m|} \right\|_{L^{\infty}(B^N_+(0,r_4'))} \right] \\
&\ \times \left[\|\pi_m(0)\|^2 + R_{NN}[\bg_m](0) - r_4'\left\|\nabla \(x_N^{-1} \pa_N \sqrt{|\bg_m|}\) \right\|_{L^{\infty}(B^N_+(0,r_4'))} \right]
\end{align*}
in $B^N_+(0,r_4')$, where $\pi_m$ is the second fundamental form of $(M, \bh_m) \subset (\ox, \bg_m)$.
Also, inspecting the proof of \cite[Lemma 2.4]{KMW2}, we see
\[\|\pi_m(0)\|^2 + R_{NN}[\bg_m](0) = {1 \over 2(n-1)} \(R\,[\bh_m](0) - \|\pi_m(0)\|^2\).\]
We want to find a representative $\chh_m$ of the conformal class $[\bh_m]$ and a small number $r_4'' \in (0, r_4']$ such that
\[E_{\chg_m}(x_N) = x_N^{1-2\ga} \check{A}_m \ge 0 \quad \text{in } B^N_+(0, r_4'')\]
for all $m \in \mn$, where $\chg_m$ is the metric on $\ox$ defined via the geodesic boundary defining function associated to $\chh_m$.
To this end, it suffices to confirm that given a fixed small number $\vep > 0$,
\begin{equation}\label{eq_ver_1}
R\,[\chh_m](0) - \|\chpi_m(0)\|^2 \ge {1 \over \vep}
\end{equation}
and
\begin{equation}\label{eq_ver_2}
r_4'' \left\| \nabla \sqrt{|\chg_m|} \right\|_{L^{\infty}(B^N_+(0,r_4''))} \le \vep, \quad
r_4'' \left\| \nabla \(x_N^{-1} \pa_N \sqrt{|\chg_m|}\) \right\|_{L^{\infty}(B^N_+(0,r_4''))} \le {1 \over 2\vep}.
\end{equation}
Here $\chpi_m$ is the second fundamental form of $(M, \chh_m) \subset (\ox, \chg_m)$.

Set $f_m(\bx) = - K |\bx|^2$ in $B^n(0,r_4')$ for some large $K > 0$ and then extend it to $M$ suitably so that $f_m \in C^{\infty}(M)$.
If we let $\chh_m = e^{2f_m} \bh_m$, then the transformation law of the scalar curvature and the umbilic tensor under a conformal change (see (1.1) of \cite{Es} and (2.2) of \cite{KMW2}) gives
\begin{align*}
&\ R\,[\chh_m](0) - \|\chpi_m(0)\|^2 \\
&= e^{-2f_m(0)} \(R\,[\bh_m] - \|\pi_m\|^2 - 2(n-1) \Delta_{\bx} f_m - (n-1)(n-2) |\nabla_{\bx} f_m|^2\)(0) \\
&\ge 4n(n-1)K - \sup_{m \in \mn} \( \left| R\,[\bh_m] \right| + \|\pi_m\|^2\)(0) \ge 2n(n-1)K > 0,
\end{align*}
which establishes \eqref{eq_ver_1}.

Verifying \eqref{eq_ver_2} requires a little more work. We extend the function $f_m$ on $M$ to its collar neighborhood $M \times [0,r_4')$
by solving a first-order partial differential equation
\[\langle df_m, d\rho_m \rangle_{\bg_m} + \frac{\rho_m}{2} |df_m|_{\bg_m}^2 = 0 \quad \text{on } M \times [0,r_4').\]
Since it is non-characteristic, a solution exists and is unique provided $r_4'$ small. Locally, it is written as
\begin{equation}\label{eq_first_1}
{\pa f_m \over \pa x_N} + \frac{x_N}{2} \left[\bg_m^{ij} {\pa f_m \over \pa x_i} {\pa f_m \over \pa x_j} + \({\pa f_m \over \pa x_N}\)^2 \right] = 0 \quad \text{in } B^N_+(0,r_4').
\end{equation}
We easily see that $\chg_m = e^{2f_m} \bg_m$ on $M \times [0,r_4')$.
Hence, by the assumption $H = 0$ on $M$ and \eqref{eq_first_1}, it is sufficient to find a small number $r_4'' = r_4''(K) > 0$ such that
\[\|\nabla f_m\|_{L^{\infty}(B^N_+(0,r_4''))} \le \vep \quad \text{and} \quad \left\| \nabla^2 f_m \right\|_{L^{\infty}(B^N_+(0,r_4''))} \le {C \over \vep} \quad \text{for all } m \in \mn\]
so as to ensure the validity of \eqref{eq_ver_2}.
Given an arbitrary point $\bx \in B^n(0,2r_4'')$, the characteristic equation of \eqref{eq_first_1} is the system of $2N+1$ ordinary differential equations of functions
\[\mbp = \mbp(s; \bx) = (p_1, \cdots, p_N)(s; \bx), \quad z = z(s; \bx) \quad \text{and} \quad \mbx = \mbx(s; \bx) = (x_1, \cdots, x_N)(s; \bx)\]
defined as
\[\begin{cases}
\dot{\mbp} = - \(\dfrac{x_N}{2} \pa_1 \bg_m^{ij}(\mbx) p_ip_j, \cdots, \dfrac{x_N}{2} \pa_n \bg_m^{ij}(\mbx) p_ip_j, \dfrac{1}{2} \(\bg_m^{ij}(\mbx) p_i p_j + p_N^2\)\),\\
\dot{z} = x_N \bg_m^{ij}(\mbx) p_i p_j + p_N (1 + x_Np_N),\\
\dot{\mbx} = \(x_N \bg_m^{1i}(\mbx) p_i, \cdots, x_N \bg_m^{ni}(\mbx) p_i, 1 + x_Np_N\),\\
\mbp(0; \bx) = (-2K\bx, 0),\, z(0; \bx) = -K|\bx|^2,\, \mbx(0; \bx) = (\bx, 0).
\end{cases}\]
Here the dot notation stands for the differentiation with respect to $s$ and the domain of the functions $(\mbp,z,\mbx)(\cdot; \bx)$ is assumed to be the interval $[0,2r_4'')$.
Then asymptotic analysis on the system indicates
\[\|\nabla f_m\|_{L^{\infty}(B^N_+(0,r_4''))} \le \sup_{\bx \in B^n(0,2r_4'')} \|\mbp(\cdot; \bx)\|_{L^{\infty}([0,2r_4''))} \le 5K^{-1}\]
and
\[\left\| \nabla^2 f_m \right\|_{L^{\infty}(B^N_+(0,r_4''))} \le 2 \sup_{\bx \in B^n(0,2r_4'')} \(\|\nabla_{\bx} \mbp(\cdot; \bx)\|_{L^{\infty}([0,2r_4''))} + \|\dot{\mbp}(\cdot; \bx)\|_{L^{\infty}([0,2r_4''))} \) \le 5K\]
for any fixed $r_4'' \in (0, K^{-2})$, thereby establishing the desired inequalities.

\medskip
Now, with the fact that $x_N^{1-2\ga} \check{A}_m \ge 0$ in $B^N_+(0, r_4'')$,
we may consider the family $\{\check{U}_m\}_{m \in \mn}$ of solutions to \eqref{eq_Yamabe_3} in which the tildes are replaced with the checks.
Notice that since $H = 0$ is an intrinsic condition that comes from \eqref{eq_hyp},
the new metrics $\chg_m$ on $\ox$ still satisfy necessary conditions for the regularity results in Appendices \ref{sec_reg} and \ref{subsec_Bocher}.
Hence our situation is reduced to Case 1 and we get the same contradiction.

\medskip \noindent \textsc{Step 6 (Completion of the proof).}
Finally, reasoning as in \cite[Proposition 4.5]{Ma} or \cite[Proposition 4.3]{Al} with estimates \eqref{eq_delta_m} and \eqref{eq_M_m_U_m},
we get the desired inequality $M_m \wtu_m (x) \le C |x|^{-(n-2\ga)}$ in $B^N_+(0,r_4) \setminus \{0\}$.
The other estimates in \eqref{eq_iso} are established as in Step 3.
This finishes the proof.
\end{proof}

\section{Linear theory and refined blow-up analysis}\label{sec_lin}
\subsection{Linear theory}
Let $\chi: [0,\infty) \to [0,1]$ be a smooth function such that $\chi(t) = 1$ on $[0,1]$ and 0 in $[2, \infty)$.
Set also $\chi_{\ep}(t) = \chi(\ep t)$ for any $\ep > 0$.
\begin{prop}\label{prop_lin}
Let $\{\pi_m\}_{m \in \mn}$ and $\{\hep_m\}_{m \in \mn}$ be sequences of symmetric 2-tensors (that is, $n \times n$-matrices) and small positive numbers
such that the averaged trace $H_m = \textnormal{tr}(\pi_m)/n$ of $\pi_m$ is 0 for each $m \in \mn$ and $\hep_m \to 0$ as $m \to \infty$.
Also, $n > 2 + 2\ga$, $\chi_m = \chi_{\hep_m}$, and $W_{1,0}$ and $w_{1,0}$ are the standard bubbles appearing in \eqref{eq_bubble} and \eqref{eq_bubble_2}, respectively.
Then there exists a solution $\Psi_m \in D^{1,2}\mrg$ to the linear equation
\begin{equation}\label{eq_lin}
\begin{cases}
-\textnormal{div} (x_N^{1-2\ga} \nabla \Psi_m) = x_N^{1-2\ga} \(2 \hep_m x_N \chi_m(|x|) (\pi_m)_{ij}\, \pa_{ij} W_{1,0}\) &\text{in } \mr^N_+,\\
\pa^{\ga}_{\nu} \Psi_m = \({n+2\ga \over n-2\ga}\) w_{1,0}^{4\ga \over n-2\ga} \Psi_m &\text{on } \mr^n
\end{cases}
\end{equation}
such that
\begin{equation}\label{eq_lin_1}
\left| \nabla_{\bx}^{\ell} \Psi_m(x) \right| \le {C \hep_m |\pi_m|_{\infty} \over 1+|x|^{n-2\ga-1+\ell}},
\quad \left| x_N^{1-2\ga} \pa_N \Psi_m(x) \right| \le {C \hep_m |\pi_m|_{\infty} \over 1 + |x|^{n-1}}
\end{equation}
for any $x \in \mr^N_+$, $\ell \in \mn \cup \{0\}$ and some $C > 0$ independent of $m \in \mn$,
\begin{equation}\label{eq_lin_2}
\Psi_m(0) = {\pa \Psi_m \over \pa x_1}(0) = \cdots = {\pa \Psi_m \over \pa x_n}(0) = 0
\end{equation}
and
\begin{equation}\label{eq_lin_3}
\int_{\mr^N_+} x_N^{1-2\ga} \nabla \Psi_m \cdot \nabla W_{1,0} dx = \int_{\mr^n} w_{1,0}^{n+2\ga \over n-2\ga} \Psi_m\, d\bx = 0.
\end{equation}
Here $|\pi_m|_{\infty} = \max\limits_{i, j = 1, \cdots, n} |(\pi_m)_{ij}|$.
\end{prop}
\begin{proof}
Fix $m \in \mn$. In the proof, we will suppress the subscript $m$ for convenience.

\medskip
Given a fixed $\hep > 0$, let $Q \in D^{1,2}\mrg$ be the function defined as
\[Q(x) = 2 \hep x_N \chi_{\hep}(|x|) \pi_{ij}\, \pa_{ij} W_{1,0}(x) \quad \text{for all } x = (\bx, x_N) \in \mr^N_+.\]
By symmetry of the functions $Z_{1,0}^0,\, \cdots, Z_{1,0}^n$ given in \eqref{eq_Z} and the assumption that $H = \text{tr}(\pi)/n = 0$, we see
\[\int_{\mr^N_+} x_N^{1-2\ga} Q\, Z_{1,0}^0 dx = \int_{\mr^N_+} x_N^{1-2\ga} Q\, Z_{1,0}^1 dx = \cdots = \int_{\mr^N_+} x_N^{1-2\ga} Q\, Z_{1,0}^n dx = 0.\]
Therefore, from the nondegeneracy result in Lemma \ref{lemma_W} (4) and the Fredholm alternative, we get a unique solution $\wtps \in D^{1,2}\mrg$ to \eqref{eq_lin} satisfying
\[\int_{\mr^N_+} x_N^{1-2\ga} \wtps\, Z_{1,0}^0 dx = \int_{\mr^N_+} x_N^{1-2\ga} \wtps\, Z_{1,0}^1 dx = \cdots = \int_{\mr^N_+} x_N^{1-2\ga} \wtps\, Z_{1,0}^n dx = 0.\]
Furthermore, by repetitive applications of the maximum principle and the scaling method with the help of Lemmas \ref{lemma_reg_2} and \ref{lemma_reg_3}, we can prove that $\wtps$ satisfies \eqref{eq_lin_1}.
Refer the proof of \cite[Lemma 3.3]{DKP} for the details.
Multiplying the first equation and \eqref{eq_bubble_2}-\eqref{eq_CS} by $W_{1,0}$ and $\wtps$, respectively, also reveals that
\[\({n+2\ga \over n-2\ga}\) \int_{\mr^n} w_{1,0}^{n+2\ga \over n-2\ga} \wtps d\bx
= \kappa_{\ga} \int_{\mr^N_+} x_N^{1-2\ga} \nabla \wtps \cdot \nabla W_{1,0} dx = \int_{\mr^n} w_{1,0}^{n+2\ga \over n-2\ga} \wtps d\bx.\]
Hence \eqref{eq_lin_3} holds for the function $\wtps$.

Now, if we set
\[\Psi = \wtps - \left[{2\wtps(0) \over \alpha_{n,\ga}(n-2\ga)}\right] Z_{1,0}^0 - \sum_{i=1}^n \left[{\pa_i \wtps(0) \over \alpha_{n,\ga}(n-2\ga)}\right]  Z_{1,0}^i,\]
then it can be easily shown from \eqref{eq_W_dec} that $\Psi$ is the desired function that satisfies \eqref{eq_lin}-\eqref{eq_lin_3}.
This concludes the proof.
\end{proof}

\subsection{Refined blow-up analysis}\label{subsec_ref}
As before, let $y_m \to y_0 \in M$ be an isolated simple blow-up point of $\{U_m\}_{m \in \mn}$.
In view of Corollary \ref{cor_i_blow} (1) and \eqref{eq_w_m_2}, $y_m \to y_0$ is an isolated blow-up point of $\{\wtu_m\}_{m \in \mn}$
and $M_m = \wtu_m(y_m)$ ($= \wtu_m(0)$ if $\tig_m$-Fermi coordinate system around $y_m$ is used).
Also, Proposition \ref{prop_iso} ensures the validity of the pointwise estimate \eqref{eq_iso} for $\{\wtu_m\}_{m \in \mn}$ near $y_0$.
The objective of this subsection is to refine it by analyzing the $\ep_m$-order terms.
Recall the functions $W_{1,0}$ and $\Psi_m$ defined in \eqref{eq_bubble_2} and constructed in Proposition \ref{prop_lin}
(where the tensor $\pi_m$ is replaced by the second fundamental form $\tipi_m(y_m)$ at $y_m$ of $(M, \tih_m) \subset (\ox, \tig_m)$), respectively.
\begin{prop}\label{prop_ref}
Suppose that $n > 2 + 2\ga$. Let $\ep_m = M_m^{-{(p_m-1)/(2\ga)}}$, $\hep_m = \alpha_{n,\ga}^{(p_m-1)/(2\ga)} \ep_m$,
\begin{equation}\label{eq_V_m}
\wtv_m(x) = \hep_m^{2\ga \over p_m-1} \wtu_m\(\hep_m x\) \quad \text{in } B^N_+(0, r_4' \hep_m^{-1}),
\end{equation}
and $\alpha_{n,\ga} > 0$ and $r_4 > 0$ be the constants introduced in \eqref{eq_bubble} and Proposition \ref{prop_iso}, respectively.
Then one can find numbers $C > 0$ and $r_5 \in (0, r_4]$ independent of $m \in \mn$ such that
\begin{equation}\label{eq_V_m_est}
\left| \nabla_{\bx}^{\ell} \wtv_m - \nabla_{\bx}^{\ell}(W_{1,0} + \Psi_m) \right|(x) \le {C \ep_m^2 \over 1+|x|^{n-2\ga-2+\ell}}
\end{equation}
for $\ell = 0, 1, 2$ and
\begin{equation}\label{eq_V_m_est_2}
\left| x_N^{1-2\ga} \pa_N \wtv_m - x_N^{1-2\ga} \pa_N(W_{1,0} + \Psi_m) \right|(x) \le {C \ep_m^2 \over 1+|x|^{n-2}}
\end{equation}
for $|x| \le r_5 \hep_m^{-1}$.
\end{prop}
\begin{proof}
Our main tool to prove this proposition will be the maximum principle; compare the proof of the corresponding result \cite[Proposition 6.1]{Al} for the boundary Yamabe problem
which makes use of Green's representation formula. 
The proof is splitted into three steps.

\medskip \noindent \textsc{Step 1 (An estimate of $\wtv_m - (W_{1,0} + \Psi_m)$).} We assert that
\begin{equation}\label{eq_wtv_d}
\left|\wtv_m - (W_{1,0} + \Psi_m)\right| \le C \max\{\ep_m^2, \delta_m\} \quad \text{in } B^N_+(0, r_4 \hep_m^{-1})
\end{equation}
where $\delta_m = (2_{n,\ga}^*-1) - p_m$.

Set
\[\Lambda_m = \max_{|x| \le r_4 \hep_m^{-1}} \left| \wtv_m - (W_{1,0} + \Psi_m) \right|(x) = \left| \wtv_m - (W_{1,0} + \Psi_m) \right|(\hat{x}_m).\]
If $|\hat{x}_m| \ge \eta r_4 \hep_m^{-1}$ for any fixed small number $\eta \in (0,1)$, we obtain a stronger inequality
\[\Lambda_m \le C \hep_m^{n-2\ga} = o(\ep_m^2)\]
than \eqref{eq_wtv_d}. Thus we may assume that $|\hat{x}_m| \le \eta r_4 \hep_m^{-1}$.

Let
\[\Theta_m = \Lambda_m^{-1} \left[ \wtv_m - (W_{1,0} + \Psi_m) \right]\]
and
\[\hmfl_m(\Theta_m) = -\text{div}_{\hg_m} (x_N^{1-2\ga} \nabla \Theta_m) + \whe_m(x_N) \Theta_m\]
in $B^N_+(0,r_4 \hep_m^{-1})$. It solves
\begin{equation}\label{eq_wtv_d_2}
\begin{cases}
\hmfl_m(\Theta_m) = x_N^{1-2\ga} \whq_{1m} &\text{in } B^N_+(0,r_4 \hep_m^{-1}),\\
\pa^{\ga}_{\nu} \Theta_m - B_m \Theta_m = \whq_{2m} &\text{on } B^n(0,r_4 \hep_m^{-1})
\end{cases}
\end{equation}
where $\hg_m = \tig_m(\hep_m \cdot)$,
\begin{align}
\whe_m(x_N) &= \({n-2\ga \over 4n}\) \left[ R[\hg_m] - (n(n+1)+R[g^+](\hep_m \cdot)) x_N^{-2} \right] x_N^{1-2\ga} \label{eq_whe_m}\\
&= \({n-2\ga \over 4n}\) \hep_m^2 \left[ R[\tig_m](\hep_m \cdot) + o(1)\right] x_N^{1-2\ga} \quad \text{(by \eqref{eq_hyp})}, \nonumber\\
\whq_{1m} &= \Lambda_m^{-1} \left[(\hg_m^{ij} - \delta^{ij})\, \pa_{ij} (W_{1,0} + \Psi_m)
+ \({\pa_i \sqrt{|\hg_m|} \over \sqrt{|\hg_m|}}\) \hg_m^{ij}\, \pa_j(W_{1,0} + \Psi_m) \right. \nonumber \\
&\hspace{40pt} + \({\pa_N \sqrt{|\hg_m|} \over \sqrt{|\hg_m|}}\) \pa_N(W_{1,0} + \Psi_m) + \pa_i \hg_m^{ij} \, \pa_j(W_{1,0} + \Psi_m) \nonumber \\
&\hspace{40pt} \left. - 2 \hep_m x_N \chi_m(|x|) (\pi_m)_{ij}\, \pa_{ij} W_{1,0} \right] \nonumber \\
\whq_{2m} &= \Lambda_m^{-1} \left[ (\hf_m^{-\delta_m} - 1)\, (w_{1,0} + \Psi_m)^{p_m} + (w_{1,0} + \Psi_m)^{n+2\ga \over n-2\ga} \left\{ (w_{1,0} + \Psi_m)^{-\delta_m} - 1 \right\} \right. \nonumber\\
&\hspace{40pt} \left. + \left\{ (w_{1,0} + \Psi_m)^{n+2\ga \over n-2\ga} - w_{1,0}^{n+2\ga \over n-2\ga} - \({n+2\ga \over n-2\ga}\) w_{1,0}^{4\ga \over n-2\ga} \Psi_m \right\} \right], \nonumber\\
B_m &= \hf_m^{-\delta_m} \left[{\wtv_m^{p_m} - (w_{1,0} + \Psi_m)^{p_m} \over \wtv_m - (w_{1,0} + \Psi_m)}\right] \nonumber
\end{align}
and $\hf_m = f_m(\hep_m \cdot)$.

As a preliminary step, we first deduce pointwise estimates of the functions $\whq_{1m}$, $\whq_{2m}$ and $B_m$.
By \eqref{eq_whu_4}, \eqref{eq_reg_2}, \eqref{eq_iso}, \eqref{eq_W_dec} and \eqref{eq_lin_1},
\begin{equation}\label{eq_wtv_d_0}
\left| \nabla_{\bx}^{\ell} \wtv_m (x) \right| + \left| \nabla_{\bx}^{\ell} W_{1,0}(x) \right| \le {C \over 1+|x|^{n-2\ga+\ell}}
\quad \text{and} \quad
\left| \nabla_{\bx}^{\ell} \Psi_m (x) \right| \le {C \hep_m \over 1+|x|^{n-2\ga-1+\ell}}
\end{equation}
in $B^N_+(0, r_4 \hep_m^{-1})$ for $\ell = 0, 1$. Moreover, by \eqref{eq_conv} and \eqref{eq_iso},
\begin{equation}\label{eq_wtv_d_01}
\wtv_m (\bx) \ge {C \over 1+|\bx|^{n-2\ga}} \quad \text{in } B^n(0, r_4 \hep_m^{-1}).
\end{equation}
From these inequalities and Lemmas \ref{lemma_metric}-\ref{lemma_W} (especially, $H = 0$ on $M$), we discover
\begin{equation}\label{eq_wtv_d_3}
\left| \whq_{1m}(x) \right| \le {C \Lambda_m^{-1} \ep_m^2 \over 1+|x|^{n-2\ga}} \quad \text{in } B^N_+(0, r_4 \hep_m^{-1}),
\end{equation}
\begin{equation}\label{eq_wtv_d_4}
\left| \whq_{2m}(\bx) \right| \le C \Lambda_m^{-1} \left[ {\delta_m \log(1+|\bx|) \over 1+|\bx|^{n+2\ga+o(1)}}
+ {\ep_m^2 \over 1+|\bx|^{n-2+2\ga}} \right] \quad \text{on } B^n(0, r_4 \hep_m^{-1})
\end{equation}
and
\begin{equation}\label{eq_B_m}
|B_m(\bx)| \le C \left[ \wtv_m^{p_m-1} + (w_{1,0} + \Psi_m)^{p_m-1} \right] \le {C \over 1+|\bx|^{4\ga + o(1)}} \quad \text{on } B^n(0, r_4 \hep_m^{-1}).
\end{equation}
Reducing the magnitude of $r_4$ if necessary, we have that $w_{1,0}(\bx) \ge 2|\Psi_m(\bx)|$ on $B^n(0,r_4\hep_m^{-1})$.
Using this fact, \eqref{eq_wtv_d_0}, \eqref{eq_wtv_d_01} and the inequality
\[|x^p-px+(p-1)| \le \begin{cases}
C(1-x)^2 \min\{1,x^{p-2}\} &\text{for } p \in (1, 2),\\
C(1-x)^2 (1+x^{p-2}) &\text{for } p \ge 2
\end{cases}\]
in $(0,\infty)$, 
we also deduce that
\begin{equation}\label{eq_B_m_der_1}
|\nabla_{\bx} B_m(\bx)| \le C\left[ {\delta_m \over 1+|\bx|^{4\ga + o(1)}} + {1 \over 1+|\bx|^{1 + 4\ga + o(1)}} \right]
\end{equation}
and
\begin{equation}\label{eq_B_m_der_2}
\left| \nabla_{\bx} \whq_{2m}(\bx) \right| \le C\Lambda_m^{-1} \left[ {\delta_m \ep_m \over 1+|\bx|^{n+2\ga+o(1)}}
+ {\delta_m \log(1+|\bx|) \over 1+|\bx|^{n+1+2\ga+o(1)}} + {\ep_m^2 \over 1+|\bx|^{n-1+2\ga}} \right]
\end{equation}
on $B^n(0,r_4 \hep_m^{-1})$.

Next, we claim that there is a small number $\eta \in (0,1)$ such that
\begin{equation}\label{eq_wtv_d_1}
|\Theta_m(x)| \le C \left[ {1 \over 1+|x|^{\ga}} + {\Lambda_m^{-1} (\ep_m^2 + \delta_m) \over 1+|x|^{n-2-2\ga}} \right] \quad \text{in } B^N_+(0, \eta r_4 \hep_m^{-1}).
\end{equation}
To verify it, we construct a barrier function
\begin{align*}
&\ \Phi_{2m}(x) \\
&= \begin{cases}
L \(1+ \Lambda_m^{-1} (\ep_m^2 + \delta_m)\) \left[\(2-|x|^2\) - \zeta x_N^{2\ga} \(2-|x|^{2-2\ga}\) \right] \hspace{41pt} \text{for } |x| \le 1,\\
L \left[ \(|x|^{-\ga} - \zeta x_N^{2\ga} |x|^{-3\ga}\) + \Lambda_m^{-1} (\ep_m^2 + \delta_m) \(|x|^{-(n-2-2\ga)} - \zeta x_N^{2\ga} |x|^{-(n-2)}\) \right] &\\\
\hspace{300pt} \text{for } 1 < |x| < \eta r_4 \hep_m^{-1}
\end{cases}
\end{align*}
with positive numbers $L$ large and $\zeta$ small. Indeed, we see from that $H = 0$ on $M$, \eqref{eq_wtv_d_2} and \eqref{eq_wtv_d_0}-\eqref{eq_B_m} that
\begin{equation}\label{eq_comp}
\begin{cases}
\begin{aligned}
\hmfl_m(\Phi_{2m}) &\ge {C x_N^{1-2\ga} \Lambda_m^{-1} \ep_m^2 \over 1+|x|^{n-2\ga}} \ge \pm x_N^{1-2\ga} \whq_{1m} = \hmfl_m(\pm \Theta_m)
\end{aligned} &\text{in } B^N_+(0, \eta r_4 \hep_m^{-1}), \\
\begin{aligned}
\pa^{\ga}_{\nu} \Phi_{2m} &\ge C \left[ {1 \over 1+|\bx|^{3\ga}} + {\Lambda_m^{-1} (\ep_m^2 + \delta_m) \over 1+|\bx|^{n-2}} \right]
\ge |B_m| + \left| \whq_{2m} \right| \\
&\ge \pm \pa^{\ga}_{\nu} \Theta_m
\end{aligned} &\text{on } B^n(0, \eta r_4 \hep_m^{-1}), \\
\Phi_{2m} \ge C \Lambda_m^{-1} \ep_m^{n-2\ga} \ge \pm \Theta_m &\text{on } \pa_I B^N_+(0, \eta r_4 \hep_m^{-1})
\end{cases}
\end{equation}
for sufficiently small $\eta \in (0,1)$. 
Thus, rescaling \eqref{eq_comp} and employing the weak maximum principle in Remark \ref{rmk_max_p},
we establish that $|\Theta_m| \le \Phi_{2m}$ in $B^N_+(0, \eta r_4 \hep_m^{-1})$. This implies \eqref{eq_wtv_d_1}.

Suppose now that $\Lambda_m^{-1} (\ep_m^2 + \delta_m) \to 0$ as $m \to \infty$.
By Lemmas \ref{lemma_conv}, \ref{lemma_reg_0}, \ref{lemma_Sch} and \eqref{eq_wtv_d_3}-\eqref{eq_B_m_der_2}, there exist a function $\Theta_0$ and a number $\beta \in (0,1)$ such that
\begin{equation}\label{eq_wtv_d_5}
\Theta_m \to \Theta_0 \quad \text{in } W^{1,2}_{\loc}(\mr^N_+; x_N^{1-2\ga}) \text{ weakly and }
C^{\beta}_{\loc}(\overline{\mr^N_+}) \cap C^1_{\loc}(\mr^n)
\end{equation}
along a subsequence, and so
\[\begin{cases}
-\text{div} (x_N^{1-2\ga} \nabla \Theta_0) = 0 &\text{in } \mr^N_+,\\
|\Theta_0| \le \dfrac{C}{1+|x|^{\ga}} &\text{in } \mr^N_+,\\
\pa^{\ga}_{\nu} \Theta_0 = \({n+2\ga \over n-2\ga}\) w_{1,0}^{4\ga \over n-2\ga} \Theta_0 &\text{on } \mr^n.
\end{cases}\]
Consequently, from the fact that $\wtv_m(0) = \nabla_{\bx} \wtv_m(0) = 0$, Lemma \ref{lemma_W} (1), \eqref{eq_lin_2} and \eqref{eq_wtv_d_5}, we see
\[\Theta_0(0) = {\pa \Theta_0 \over \pa x_1}(0) = \cdots = {\pa \Theta_0 \over \pa x_n}(0) = 0.\]
In view of Lemma \ref{lemma_W} (4), it should hold that $\Theta_0 = 0$ in $\mr^N_+$ and thus $|\hat{x}_m| \to \infty$ as $m \to \infty$.
However, this induces a contradiction because $1 = \Theta_m(\hat{x}_m) \to 0$ as $m \to \infty$.
Estimate \eqref{eq_wtv_d} must be true.

\medskip \noindent \textsc{Step 2 (Estimate of $\delta_m$).}
We assert
\begin{equation}\label{eq_d_m_est}
\delta_m \le C \ep_m^2.
\end{equation}
Its proof can be done as in \cite[Lemma 6.2]{Al} with minor modifications, so is omitted.

\medskip
As a particular consequence of \eqref{eq_wtv_d} and \eqref{eq_d_m_est}, we get
\begin{equation}\label{eq_wtv_d_6}
\left|\wtv_m - (W_{1,0} + \Psi_m)\right|(x) \le C \ep_m^2 \quad \text{for all } x \in B^N_+(0, r_4 \hep_m^{-1}).
\end{equation}
\noindent \textsc{Step 3 (Completion of the proof).}
We can now deduce \eqref{eq_V_m_est} with $\ell = 0$. Redefine
\[\Theta_m = \ep_m^{-2} \left[ \wtv_m - (W_{1,0} + \Psi_m) \right] \quad \text{for } x \in B^N_+(0,r_4 \hep_m^{-1})\]
so that it solves equation \eqref{eq_wtv_d_2} once each quantity $\Lambda_m$ in the definition of $\whq_{1m}$ and $\whq_{2m}$ is replaced with $\ep_m^2$.
As in \eqref{eq_wtv_d_3} and \eqref{eq_wtv_d_4}, it holds that
\begin{equation}\label{eq_Q_1m}
\left| \whq_{1m}(x) \right| \le {C \over 1+|x|^{n-2\ga}}
\quad \text{and} \quad
\left| \whq_{2m}(\bx) \right| \le {C \over 1+|\bx|^{n+2\ga-2}}
\end{equation}
for $x \in B^N_+(0,r_4 \hep_m^{-1})$ and $\bx \in B^n(0,r_4 \hep_m^{-1})$. By \eqref{eq_wtv_d_6} and \eqref{eq_wtv_d_0},
\begin{equation}\label{eq_th_m_3}
|\Theta_m| \le C \quad \text{in } B^N_+(0, r_4 \hep_m^{-1}) \quad \text{and} \quad
|\Theta_m| \le C \ep_m^{n-2\ga-2} \quad \text{on } \pa_I B^N_+(0, r_4 \hep_m^{-1}).
\end{equation}

For any $0 < \mu \le n-2\ga-2$, we define a function
\[\Phi_{3m; \mu}(x) = \begin{cases}
L_{\mu} \left[\(2-|x|^2\) - \zeta_{\mu} x_N^{2\ga} \(2-|x|^{2-2\ga}\)\right] &\text{for } |x| \le 1,\\
L_{\mu} \(|x|^{-\mu} - \zeta_{\mu} x_N^{2\ga} |x|^{-\mu-2\ga}\) &\text{for } 1 < |x| < r_4 \hep_m^{-1}
\end{cases}\]
where $L_{\mu}$ and $\zeta_{\mu}$ are a large and a small positive number relying only on $\mu$, $n$ and $\ga$, respectively.
If we set $\mu_0 = \min\{\ga, n-2\ga-2\}$, a direct computation using $H = 0$ on $M$, \eqref{eq_wtv_d_2}, \eqref{eq_B_m}, \eqref{eq_Q_1m} and \eqref{eq_th_m_3} shows
\begin{equation}\label{eq_phi_m}
\begin{cases}
\begin{aligned}
\hmfl_m(\Phi_{3m;\mu_0}) &\ge {C x_N^{1-2\ga} \over 1+|x|^{n-2\ga}} \ge \hmfl_m(\pm \Theta_m)
\end{aligned} &\text{in } B^N_+(0, r_5' \hep_m^{-1}), \\
\pa^{\ga}_{\nu} \Phi_{3m;\mu_0} \ge \pm \pa^{\ga}_{\nu} \Theta_m &\text{on } B^n(0, r_5' \hep_m^{-1}), \\
\Phi_{3m;\mu_0} \ge \pm \Theta_m &\text{on } \pa_I B^N_+(0, r_5' \hep_m^{-1})
\end{cases}
\end{equation}
for $r_5' \in (0,r_4]$ small enough. Hence we deduce from the weak maximum principle in Remark \ref{rmk_max_p} that
\begin{equation}\label{eq_th_m_2}
|\Theta_m| \le \Phi_{3m;\mu_0} \le {C \over 1+|x|^{\mu_0}} \quad \text{in } B^N_+(0, r_5' \hep_m^{-1}).
\end{equation}
If $\mu_0 = n-2\ga-2$, we are done. Otherwise, we put \eqref{eq_th_m_2} into the second inequality of \eqref{eq_phi_m} in order to improve itself so that
\[|\Theta_m| \le \Phi_{3m;\mu_1} \le {C \over 1+|x|^{\mu_1}} \quad \text{in } B^N_+(0, r_5'' \hep_m^{-1})\]
for $\mu_1 = \min\{2\ga, n-2\ga-2\}$ and $r_5'' \in (0, r_5']$.
Iterating this process, we can conclude the proof of \eqref{eq_V_m_est} with $\ell = 0$.

The remaining inequalities, i.e., \eqref{eq_V_m_est} for $\ell = 1, 2$ and \eqref{eq_V_m_est_2} are derived as in the justification of \eqref{eq_U_m_est_2}.
Indeed, a tedious but straightforward calculation shows that the second-order derivatives of the functions $B_m(\bx)$ and $\whq_{2m}(\bx)$ have the required decay rate as $|\bx| \to \infty$.
The proof is now completed.
\end{proof}

\section{Vanishing theorem of the second fundamental form}\label{sec_van}
Let us denote by $(\tig_0, \tih_0)$ the $C^4$-limit of the sequence $\{(\tig_m, \tih_m)\}_{m \in \mn}$ given in Subsection \ref{subsec_blow}.
For each $m \in \mn$, let $\tipi_m$ and $\tipi_0$ be the second fundamental forms of $(M, \tih_m) \subset (\ox, \tig_m)$ and $(M, \tih_0) \subset (\ox, \tig_0)$, respectively.
By employing the sharp pointwise estimate presented in Proposition \ref{prop_ref},
we now prove that $\tipi_0 = 0$ at an isolated simple blow-up point $y_0 \in M$ of a sequence of solutions $\{U_m\}_{m \in \mn}$ to \eqref{eq_Yamabe_3}.
\begin{prop}\label{prop_van}
Suppose that $\ga \in (0,1)$, $n \in \mn$ satisfies the dimension restriction \eqref{eq_n} and $y_m \to y_0$ is an isolated simple blow-up point of the sequence $\{U_m\}_{m \in \mn}$
so that the description for $\{\wtu_m\}_{m \in \mn}$ in the first paragraph of Subsection \ref{subsec_ref} holds. Then
\begin{equation}\label{eq_van}
\|\tipi_m(y_m)\| \to 0 \quad \text{as } m \to \infty.
\end{equation}
Particularly, $\tipi_0(y_0) = 0$.
\end{prop}
\begin{proof}
In the proof, we will use Lemma \ref{lemma_rep} with $\tig = \tig_m$, $\tih = \tih_m$ and $y = y_m$,
and think as if $\wtu_m$ is a function in $\mr^N_+$ near the origin by appealing $\tig_m$-Fermi coordinates on $\ox$ around $y_m$.

\medskip
Denoting $\hg_m = \tig_m(\hep_m \cdot)$ and $\hf_m = \tif_m(\hep_m \cdot)$, we set
\[Q_{0m}(U) = \({n-2\ga \over 4n}\) \hep_m^2 (R[\tig_m](\hep_m \cdot) + o(1)) U\]
and
\begin{equation}\label{eq_Q_m}
\begin{aligned}
Q_m(U) &= (\hg_m^{ij} - \delta^{ij})\, \pa_{ij} U + \left[ \({\pa_i \sqrt{|\hg_m|} \over \sqrt{|\hg_m|}}\) \hg_m^{ij}\, \pa_jU + \({\pa_N \sqrt{|\hg_m|} \over \sqrt{|\hg_m|}}\) \pa_NU \right] + \pa_i \hg_m^{ij} \, \pa_j U \\
&= Q_{1m}(U) + Q_{2m}(U) + Q_{3m}(U).
\end{aligned}
\end{equation}
Also, let $\whe_m$ be the functions introduced in \eqref{eq_whe_m} so that $\wtv_m$ in \eqref{eq_V_m} is a solution of
\[\begin{cases}
-\text{div} (x_N^{1-2\ga} \nabla \wtv_m) + \whe_m(x_N) \wtv_m = x_N^{1-2\ga} Q_m(\wtv_m) &\text{in } B^N_+(0,r_5 \hep_m^{-1}),\\
\pa^{\ga}_{\nu} \wtv_m = \hf_m^{-\delta_m} \wtv_m^{p_m} &\text{on } B^n(0,r_5 \hep_m^{-1}).
\end{cases}\]
Thus in view of Pohozaev's identity in Lemma \ref{lemma_poho}, one can write
\[\mcp\(\wtv_m, r\hep_m^{-1}\) = \mcp_{1m}\(\wtv_m, r\hep_m^{-1}\) + {\delta_m \over p_m+1} \mcp_{2m}\(\wtv_m, r\hep_m^{-1}\) \quad \text{for any } r \in (0, r_5]\]
where
\begin{equation}\label{eq_F_R_m}
\begin{aligned}
\mcp_{1m}(U, r) &= \kappa_{\ga} \int_{B^N_+(0,r)} x_N^{1-2\ga} \left[ Q_{0m}(U) - Q_m(U) \right] \\
&\hspace{160pt} \times \left[x_i \pa_i U + x_N \pa_N U + \({n-2\ga \over 2}\) U \right] dx,\\
\mcp_{2m}(U, r) &= - \int_{B^n(0,r)} x_i \pa_i \hf_m \hf_m^{-(\delta_m+1)} u^{p_m+1} d\bx + \({n-2\ga \over 2}\) \int_{B^n(0,r)} \hf_m^{-\delta_m} u^{p_m+1} d\bx
\end{aligned}
\end{equation}
for $u = U$ on $B^n(0,r)$.

\medskip
For a fixed number $r \in (0, r_5]$, let
\begin{multline*}
\whf_m(V_1, V_2) = \kappa_{\ga} \int_{B^N_+(0,r\hep_m^{-1})} x_N^{1-2\ga} \left[Q_{0m}(V_1) - Q_m(V_1) \right] \\
\times \left[x_i \pa_i V_2 + x_N \pa_N V_2 + \({n-2\ga \over 2}\) V_2 \right] dx.
\end{multline*}
Owing to \eqref{eq_V_m_est} and \eqref{eq_V_m_est_2}, we are led to
\begin{multline}\label{eq_whf_m_0}
\mcp_{1m}\(\wtv_m, r\hep_m^{-1}\) \\
= \whf_m(W_{1,0}, W_{1,0}) + \whf_m(W_{1,0}, \Psi_m) + \whf_m(\Psi_m, W_{1,0}) + \whf_m(\Psi_m, \Psi_m) + o(\ep_m^2)
\end{multline}
provided that $n > 2 + 2\ga$.

We first estimate $\whf_m(W_{1,0}, W_{1,0})$. This amounts to calculate the integrals
\[\wtf_{\ell m} = \kappa_{\ga} \int_{\mr^N_+} x_N^{1-2\ga} Q_{\ell m}(W_{1,0})\, Z_{1,0}^0 dx \quad \text{for } \ell = 0,\, 1,\, 2,\, 3\]
where $Z_{1,0}^0 = x \cdot \nabla W_{1,0} + (n-2\ga)/2 \cdot W_{1,0}$ is the function described in \eqref{eq_Z}.
For the value $\wtf_{1m}$, we discover from Lemmas \ref{lemma_W} (1), \ref{lemma_metric} and \ref{lemma_rep} that
\begin{align}
\wtf_{1m} &= \kappa_{\ga} \left[ 2 \hep_m (\tipi_m)_{ij} \int_{\mr^N_+} x_N^{2-2\ga} \pa_{ij}W_{1,0} Z_{1,0}^0 dx \nonumber
+ {1 \over 3} \hep_m^2 R_{ikjl}[\tih_m] \int_{\mr^N_+} x_N^{1-2\ga} x_k x_l\, \pa_{ij} W_{1,0} Z_{1,0}^0 dx \right. \nonumber \\
&\hspace{30pt} + \hep_m^2 (\tig_m)^{ij}_{\phantom{ij},Nk} \int_{\mr^N_+} x_N^{2-2\ga} x_k \, \pa_{ij} W_{1,0} Z_{1,0}^0 dx \nonumber \\
&\hspace{30pt} \left. + \hep_m^2 \(3 (\tipi_m)_{ik} (\tipi_m)_{kj} + R_{iNjN}[\tig_m]\) \int_{\mr^N_+} x_N^{3-2\ga} \pa_{ij}W_{1,0} Z_{1,0}^0 dx + o(\ep_m^2) \right]  \label{eq_wtf_1} \\
&= \kappa_{\ga} \left[0 + 0 + 0 + \hep_m^2 \(3 \|\tipi_m\|^2 + R_{NN}[\tig_m]\) \cdot {1 \over n} \int_{\mr^N_+} x_N^{3-2\ga} \Delta_{\bx} W_{1,0} Z_{1,0} dx + o(\ep_m^2) \right] \nonumber \\
&\hspace{270pt} \text{(since $\widetilde{H}_m = R[\tih_m] = 0$)} \nonumber \\
&= \hep_m^2 \|\tipi_m\|^2 \kappa_{\ga} \left[{4n-5 \over 2n(n-1)}\right] \int_{\mr^N_+} x_N^{3-2\ga} \Delta_{\bx} W_{1,0} Z_{1,0}^0 dx + o(\ep_m^2) \nonumber
\end{align}
where $\widetilde{H}_m = \text{tr}(\tipi_m)/n$. Similarly, we have
\begin{align}
\wtf_{2m} &= \hep_m^2 \|\tipi_m\|^2 \kappa_{\ga} \left[{1 \over 2(n-1)}\right] \int_{\mr^N_+} x_N^{2-2\ga} \pa_N W_{1,0} Z_{1,0}^0 dx + o(\ep_m^2) \label{eq_wtf_31}
\intertext{and}
\wtf_{3m} &= o(\ep_m^2). \label{eq_wtf_32}
\end{align}
Moreover, the Gauss-Codazzi equation and Lemma \ref{lemma_rep} yield
\[R[\tig_m] = 2R_{NN}[\tig_m] + \|\tipi_m\|^2 + R[\tih_m] - \widetilde{H}_m^2 = - \({n \over n-1}\) \|\tipi_m\|^2 \quad \text{at } y_m \in M.\]
Hence we find
\begin{equation}\label{eq_wtf_2}
\wtf_{0m} = -\hep_m^2 \|\tipi_m\|^2 \kappa_{\ga} \left[{n-2\ga \over 4(n-1)}\right] \int_{\mr^N_+} x_N^{1-2\ga} W_{1,0} Z_{1,0}^0 dx + o(\ep_m^2).
\end{equation}
Consequently, by combining \eqref{eq_wtf_1}-\eqref{eq_wtf_2} and employing \eqref{eq_W_int}, we obtain that
\begin{equation}\label{eq_whf_m_2}
\whf_m(W_{1,0}, W_{1,0}) = \hep_m^2 \|\tipi_m\|^2 \kappa_{\ga} \left[{3n^2 + n(16\ga^2-22)+20(1-\ga^2) \over 8n(n-1)(1-\ga^2)}\right] \mcc_0 + o(\ep_m^2)
\end{equation}
for some constant $\mcc_0 > 0$ relying only on $n$ and $\ga$.
We note that the coefficient of $\ep_m^2$ (or $\hep_m^2$) is positive if and only if $n$ satisfies \eqref{eq_n} for each $\ga \in (0,1)$.

On the other hand, it holds that
\begin{equation}\label{eq_whf_m_1}
\whf_m(W_{1,0}, \Psi_m) + \whf_m(\Psi_m, W_{1,0}) \ge o(\ep_m^2)
\end{equation}
whose verification is deferred to the end of the proof, and
\begin{equation}\label{eq_whf_m_3}
\whf_m(\Psi_m, \Psi_m) = o(\ep_m^2).
\end{equation}
By plugging \eqref{eq_whf_m_2}-\eqref{eq_whf_m_3} into \eqref{eq_whf_m_0}, we reach at
\[\mcp_{1m}\(\wtv_m, r\hep_m^{-1}\) \ge \hep_m^2 \|\tipi_m\|^2 \kappa_{\ga} \left[{3n^2 + n(16\ga^2-22)+20(1-\ga^2) \over 8n(n-1)(1-\ga^2)}\right] \mcc_0 + o(\ep_m^2).\]

\medskip
Using \eqref{eq_iso}, \eqref{eq_delta_m} and \eqref{eq_V_m}, we deduce \[\begin{cases}
\left| \nabla_{\bx}^{\ell} \wtv_m(x) \right| \le C |x|^{-(n-2\ga+\ell)}  \ (\ell = 0, 1, 2),\\
\left| x_N^{1-2\ga} \pa_N \wtv_m(x) \right| \le C|x|^{-n}
\end{cases}
\quad \text{in } B^N_+(0, r_4 \hep_m^{-1}).\]
Thus, it follows from \eqref{eq_F_R_m} that
\[\left| \mcp\(\wtv_m, r\hep_m^{-1}\) \right| \le C \ep_m^{n-2\ga} \quad \text{and} \quad \mcp_{2m}\(\wtv_m, r\hep_m^{-1}\) \ge 0\]
if $r > 0$ is selected to be small enough. As a result, estimate \eqref{eq_van} follows.

\medskip \noindent \textbf{Derivation of \eqref{eq_whf_m_1}.}
Since $\tipi_m \to \tipi_0$ in $C^1(M)$, the norm $|\tipi_m|_{\infty}$ (see the statement of Proposition \ref{prop_lin}) is uniformly bounded in $m \in \mn$.
Thus, by virtue of \eqref{eq_Q_m}, \eqref{eq_W_dec}, \eqref{eq_lin_1}, Lemma \ref{lemma_metric} and integration by parts, we observe
\begin{align*}
&\ \whf_m(W_{1,0}, \Psi_m) + \whf_m(\Psi_m, W_{1,0}) \\
&= - 2 \hep_m (\tipi_m)_{ij}\, \kappa_{\ga} \int_{\mr^N_+} x_N^{2-2\ga} \left[\pa_{ij} W_{1,0} \left\{x_i \pa_i \Psi_m + x_N \pa_N \Psi_m + \({n-2\ga \over 2}\) \Psi_m \right\} \right. \\
&\hspace{130pt} \left. + \pa_{ij} \Psi_m \left\{x_i \pa_i W_{1,0} + x_N \pa_N W_{1,0} + \({n-2\ga \over 2}\) W_{1,0}\right\} \right] dx + o(\ep_m^2) \\
&= 2 \hep_m (\tipi_m)_{ij}\, \kappa_{\ga} \int_{\mr^N_+} x_N^{2-2\ga} \left[ (n-2\ga+2)\, \pa_i W_{1,0} \pa_j \Psi_m
+ \pa_i W_{1,0} (x_k \pa_{jk} \Psi_m + x_N \pa_{jN} \Psi_m) \right. \\
&\hspace{200pt} \left. + (x_k \pa_{ik} W_{1,0} + x_N \pa_{iN} W_{1,0}) \pa_j \Psi_m \right] dx + o(\ep_m^2) \\
&= - 2 \hep_m (\tipi_m)_{ij}\, \kappa_{\ga} \int_{\mr^N_+} x_N^{2-2\ga} \pa_i W_{1,0} \pa_j \Psi_m dx + o(\ep_m^2)
\end{align*}
provided that $n > 2 + 2\ga$. On the other hand, applying another integration by parts and testing $\Psi_m$ in \eqref{eq_lin} lead to
\begin{multline*}
- 2 \hep_m (\tipi_m)_{ij}\, \kappa_{\ga} \int_{\mr^N_+} x_N^{2-2\ga} \pa_i W_{1,0} \pa_j \Psi_m dx \\
= \underbrace{\kappa_{\ga} \int_{\mr^N_+} x_N^{1-2\ga} |\nabla \Psi_m|^2 dx - \({n+2\ga \over n-2\ga}\) \int_{\mr^n} w_{1,0}^{4\ga \over n-2\ga} \Psi_m^2 d\bx}_{= \mathcal{I}} + o(\ep_m^2).
\end{multline*}
It is well-known that the Morse index of $w_{1,0} \in H^{\ga}(\mr^n)$ is 1 due to the contribution of $w_{1,0}$ itself.
Hence we see from \eqref{eq_lin_3} that $\mathcal{I} \ge 0$; see the proof of \cite[Lemma 4.5]{DKP} for more explanation.
This completes the proof.
\end{proof}

\section{Proof of the main theorems}\label{sec_main}
\subsection{Exclusion of bubble accumulation}
Set
\begin{equation}\label{eq_mcp'}
\mcp'(U, r) = \kappa_{\ga} \int_{\pa_I B^N_+(0,r)} x_N^{1-2\ga} \left[\({n-2\ga \over 2}\) u {\pa u \over \pa r} - {r \over 2} |\nabla u|^2
+ r \left| {\pa u \over \pa r} \right|^2 \right] d\sigma_{x},
\end{equation}
which is a part of the function $\mcp$ defined in \eqref{eq_poho_1}.
\begin{lemma}\label{lemma_lsign}
Assume that $\ga \in (0,1)$ and the dimension condition \eqref{eq_n} holds.
Let $y_m \to y_0$ be an isolated simple blow-up point of $\{U_m\}_{m \in \mn}$, and $\{\wtu_m\}_{m \in \mn}$ be the sequence of the functions constructed in Subsection \ref{subsec_blow}.
Suppose further that $\tipi_0(y_0) \ne 0$.
Then, given $m \in \mn$ large and $r > 0$ small, there exist universal constants $\mcc_1, \cdots, \mcc_4 > 0$ such that
\begin{equation}\label{eq_lsign}
\hep_m^{n-2\ga} \mcp'\(\wtu_m(0)\wtu_m, r\) \ge \hep_m^2 \mcc_1 - \hep_m^{2+\eta} r^{2-\eta} \mcc_2 - \hep_m^{n-2\ga} r^{-n+2\ga+1} \mcc_3 - {\hep_m^n r^n \mcc_4 \over \hep_m^{2n+o(1)} + r^{2n+o(1)}}
\end{equation}
in $\bg_m$-Fermi coordinates centered in $y_m$. Here, $\eta > 0$ is an arbitrarily small number.
\end{lemma}
\begin{proof}
It holds that
\[\mcp'\(\wtu_m(0)\wtu_m, r\) 
= \hep_m^{-(n-2\ga)+o(1)} \left[ \mcp\(\wtv_m, r\hep_m^{-1}\) -
{r\hep_m^{-1} \over p+1} \int_{\pa B^n(0,r\hep_m^{-1})} \hf_m^{-\delta_m} \wtv_m^{p_m+1} d\sigma_{\bx} \right]\]
where $\wtv_m$ is the function defined in \eqref{eq_V_m} and $\hep_m^{o(1)} \to 1$ as $m \to \infty$.
Inspecting the proof of Proposition \ref{prop_van} and using the assumption that $\tipi_0(y_0) \ne 0$, we obtain
\[\mcp\(\wtv_m, r\hep_m^{-1}\) \ge \hep_m^2 \mcc_1 - \hep_m^{2+\eta} r^{2-\eta} \mcc_2 - \hep_m^{n-2\ga} r^{-n+2\ga+1} \mcc_3\]
for some $\mcc_1, \mcc_2, \mcc_3 > 0$ and small $\eta > 0$. Also, by \eqref{eq_wtv_d_0},
\[\left| r\hep_m^{-1} \int_{\pa B^n(0,r\hep_m^{-1})} \hf_m^{-\delta_m} \wtv_m^{p_m+1} d\sigma_{\bx} \right| \le {\mcc_4 \hep_m^n r^n \over \hep_m^{2n+o(1)} + r^{2n+o(1)}}\]
for some $\mcc_4 > 0$. Therefore \eqref{eq_lsign} holds.
\end{proof}

We shall use the following Liouville type lemma to prove Lemma \ref{lemma_sim}.
\begin{lemma}\label{lemma_li}
If $U \in W^{1,2}_{\loc}(\mr_+^N; x_N^{1-2\ga})$ is a solution to
\begin{equation}\label{eq_har}
\begin{cases}
-\textnormal{div} (x_N^{1-2\ga} \nabla U) = 0 &\text{in } \mr^N_+,\\
\pa^{\ga}_{\nu} U = 0 &\text{on } \mr^n,\\
\liminf_{|x| \to \infty} U(x) \ge 0,
\end{cases}
\end{equation}
then it should be a nonnegative constant.
\end{lemma}
\begin{proof}
According to H\"older estimate \eqref{eq_DNM} and the asymptotic condition in \eqref{eq_har},
$U$ is contained in $C^{\beta}_{\loc}(\overline{\mr^N_+})$ and bounded from below.
Let us denote $m_0 = \inf_{\mr^N_+} U$ and $U_{m_0} = U - m_0 \ge 0$.
By the Harnack inequality \eqref{eq_har_2} and scaling invariance of \eqref{eq_har}, we have
\begin{equation}\label{eq_har_3}
\sup_{B^N_+(0,r)} U_{m_0} \le C \inf_{B^N_+(0,r)} U_{m_0} \quad \text{for any } r > 0
\end{equation}
where $C > 0$ is independent of $r$.
Letting $r \to \infty$ in \eqref{eq_har_3}, we obtain that $U_{m_0} = 0$ or $U = m_0$ in $\mr^N_+$.
One more application of the asymptotic condition on $U$ forces $m_0 \ge 0$. The proof is concluded.
\end{proof}

\begin{lemma}\label{lemma_sim}
Assume that $\ga \in (0,1)$ and the dimension condition \eqref{eq_n} holds.
Let $y_m \to y_0 \in M$ be an isolated blow-up point of the sequence $\{U_m\}_{m \in \mn}$ and $\pi_0(y_0) \ne 0$.
Then $y_0$ is an isolated simple blow-up point of $\{U_m\}_{m \in \mn}$.
\end{lemma}
\begin{proof}
Thanks to Corollary \ref{cor_i_blow} (2), the weighted average $\bar{u}_m$ of $u_m = U_m|_M$ (see \eqref{eq_bu} for its precise definition)
has exactly one critical point in $(0, R_m \hep_m)$ for large $m \in \mn$.
To the contrary, suppose that there exists the second critical point $\vrh_m$ of $\bar{u}_m$ such that $R_m \hep_m \le \vrh_m \to 0$ as $m \to \infty$. Define
\[\wtt_m(x) = \vrh_m^{2\ga \over p_m-1} \wtu_m(\vrh_m x) \quad \text{in } B^N_+(0, \vrh_m^{-1}r_5).\]
Then, using Propositions \ref{prop_iso} and \ref{prop_boc}, Lemma \ref{lemma_li} and \eqref{eq_iso_2}, one can verify the existence of $c_1 > 0$ and $\beta \in (0,1)$ such that
\begin{equation}\label{eq_wtt}
\wtt_m(0)\, \wtt_m \to c_1 \(|x|^{-(n-2\ga)} + 1\) \quad \text{in } C^{\beta}_{\loc}(\overline{\mr^N_+} \setminus \{0\}) \cap C^1_{\loc}(\mr^n \setminus \{0\})
\end{equation}
up to a subsequence; see \cite[Proposition 8.1]{Al}.
It follows from \eqref{eq_lsign} that
\begin{multline}\label{eq_wtt_1}
\hep_m^{-(n-2\ga)} \( \hep_m^2 \mcc_1 - \hep_m^{2+\eta} \vrh_m^{2-\eta} \mcc_2 - \hep_m^{n-2\ga} \vrh_m^{-n+2\ga+1} \mcc_3\) - {\hep_m^{2\ga} \vrh_m^n \mcc_4 \over \hep_m^{2n+o(1)} + \vrh_m^{2n+o(1)}} \\
\le \mcp'\(\wtu_m(0)\wtu_m, \vrh_m\) = \vrh_m^{-(n-2\ga)+o(1)} \mcp'\(\wtt_m(0)\wtt_m, 1\).
\end{multline}
Note that
\begin{equation}\label{eq_wtt_21}
\vrh_m^{n-2\ga+o(1)} \cdot \hep_m^{-(n-2\ga)} \( \hep_m^2 \mcc_1 - \hep_m^{2+\eta} \vrh_m^{2-\eta} \mcc_2 - \hep_m^{n-2\ga} \vrh_m^{-n+2\ga+1} \mcc_3\) \ge - 2 \vrh_m \mcc_3 \to 0
\end{equation}
and
\begin{equation}\label{eq_wtt_22}
\vrh_m^{n-2\ga+o(1)} \cdot {\hep_m^{2\ga} \vrh_m^n \over \hep_m^{2n+o(1)} + \vrh_m^{2n+o(1)}} \le C \({\hep_m \over \vrh_m}\)^{2\ga+o(1)} \le {C \over R_m^{2\ga+o(1)}} \to 0
\end{equation}
as $m \to \infty$. Hence, taking the limit on \eqref{eq_wtt_1} and employing \eqref{eq_wtt_21}, \eqref{eq_wtt_22}, \eqref{eq_wtt} and \eqref{eq_mcp'}, we obtain
\begin{align*}
0 \le \lim_{m \to \infty} \mcp'\(\wtt_m(0)\wtt_m, 1\) &= \mcp'\(c_1 \(|x|^{-(n-2\ga)} + 1\), 1\) \\
&= - \kappa_{\ga} c_1^2 \({n-2\ga \over 2}\) \int_{\pa_I B^N_+(0,1)} x_N^{1-2\ga} d\sigma_{x} < 0,
\end{align*}
which is a contradiction. The assertion in the statement must be true.
\end{proof}

We rule out bubble accumulation by applying Lemma \ref{lemma_sim}.
\begin{prop}\label{prop_sim}
Assume the hypotheses of Theorem \ref{thm_main_2}.
Let $\vep_0, \vep_1,\, R,\, C_0,\, C_1$ be positive numbers in the statement of Proposition \ref{prop_blow_a}.
Suppose that $U \in W^{1,2}(X; \rho^{1-2\ga})$ is a solution to \eqref{eq_Yamabe_3} and $\{y_1, \cdots, y_{\mcn}\}$ is the set of its local maxima on $M$.
Then there exists a constant $C_2 > 0$ depending only on $(X, g^+)$, $\bh$, $n$, $\ga$, $\vep_0$, $\vep_1$ and $R$ such that if $\max_M U \ge C_0$,
then $d_{\bh}(y_{m_1}, y_{m_2}) \ge C_2$ for all $1 \le m_1 \ne m_2 \le \mcn(U)$.
\end{prop}
\begin{proof}
By appealing to Propositions \ref{prop_blow_a}, \ref{prop_iso} and \ref{prop_boc}, Lemmas \ref{lemma_lsign} and \ref{lemma_sim},
the maximum principle and the Hopf lemma in Lemma \ref{lemma_hopf},
one can argue as in \cite[Proposition 8.2]{Al} or \cite[Proposition 5.2]{JLX}.
\end{proof}

By Proposition \ref{prop_sim}, $\sup_{m \in \mn} \mcn(U_m)$ is bounded. Therefore we have
\begin{cor}\label{cor_sim}
Assume the hypotheses of Theorem \ref{thm_main_2}.
Then the set of blow-up points of $\{U_m\}_{m \in \mn}$ is finite and it consists of isolated simple blow-up points.
\end{cor}

\subsection{Proof of the main theorems}\label{subsec_pf_m}
We are now ready to complete the proof of the main theorems.
All notations used in the proofs are borrowed from Subsection \ref{subsec_blow}.
\begin{proof}[Proof of Theorem \ref{thm_main_1}]
According to Corollary \ref{cor_sim}, any blow-up point $y_m \to y_0 \in M$ of $\{U_m\}_{m \in \mn}$ is isolated simple.
Therefore Proposition \ref{prop_van} implies the validity of Theorem \ref{thm_main_1}.
\end{proof}

\begin{proof}[Proof of Theorem \ref{thm_main_2}]
We first claim that $u \le C$ on $M$.
If it does not hold, then by Proposition \ref{prop_CG}, there is a sequence $\{U_m\}_{m \in \mn} \subset W^{1,2}(X;\rho^{1-2\ga})$ of solutions to \eqref{eq_Yamabe_2} which blows-up at a point $y_0 \in M$.
By applying Theorem \ref{thm_main_1}, we conclude that $\pi(y_0) = 0$.
However, it is contradictory to the assumption that $\pi$ never vanishes on $M$, so our claim should be true.

A combination of \eqref{eq_Y_cons}, Lemma \ref{lemma_reg_0} and Proposition \ref{prop_reg} now yields the other estimates in \eqref{eq_u_est},
that is to say, the lower and $C^{2+\beta}$-estimates of $u$ on $M$.
\end{proof}

At this stage, only Theorem \ref{thm_main_3} is remained to be verified.
To define the Leray-Schauder degree $\deg(\mcf_p, D_{\Lambda},0)$ for all $1 \le p \le 2_{n,\ga}^*-1$ and apply its homotopy invariance property, we need the following result.
\begin{lemma}\label{lemma_main_3}
Assume the hypotheses of Theorem \ref{thm_main_2}.
Then one can choose a constant $C = C(X^{n+1}, g^+, \bh, \ga) > 1$ such that
\[C^{-1} \le u \le C \quad \text{on } M\]
for all $1 \le p \le 2_{n,\ga}^*-1$ and $u > 0$ satisfying \eqref{eq_Yamabe_1}.
\end{lemma}
\begin{proof}
We consider the extension problem \eqref{eq_Yamabe_2} where the fourth line is substituted with $\pa^{\ga}_{\nu} U = \mce(u) u^p$ on $M$.
By adapting the proof of Lemmas 4.1 and 6.5 in \cite{FO} and applying Theorem \ref{thm_main_2}, we get the result.
\end{proof}

\begin{proof}[Proof of Theorem \ref{thm_main_3}]
From the previous lemma, we find that $0 \notin \mcf_p(\pa D_{\Lambda})$ for all $1 \le p \le 2_{n,\ga}^*-1$ provided $\Lambda > 0$ large enough. Therefore
\[\deg(\mcf_p, D_{\Lambda}, 0) = \deg(\mcf_1, D_{\Lambda}, 0) \quad \text{for all } 1 \le p \le 2_{n,\ga}^*-1.\]

Since $\Lambda^{\ga}(M,[\bh]) > 0$, the first $L^2(M)$-eigenvalue of $P^{\ga}_{\bh}$ must be positive, for
\[\int_M u P^{\ga}_{\bh} u\, dv_{\bh} \ge \Lambda^{\ga}(M, [\bh])\, \|u\|_{L^{2n \over n-2\ga}(M)}^2 \ge \Lambda^{\ga}(M, [\bh]) |M|^{-{2\ga \over n}} \, \|u\|_{L^2(M)}^2, \quad u \in H^{\ga}(M).\]
Also, in \cite[Section 4]{GQ} and \cite[Section 1]{Cs}, it was proved that the first eigenspace of $P^{\ga}_{\bh}$ is one-dimensional and spanned by a positive function on $M$.
By the $L^2(M)$-orthogonality, the other eigenfunctions must change their signs.
Using these characterizations, one can follow the argument in \cite{Sc2}, up to minor modifications, to derive $\deg(\mcf_1, D_{\Lambda}, 0) = -1$.
The proof of Theorem \ref{thm_main_3} is completed.
\end{proof}

\appendix
\section{Elliptic regularity}\label{sec_reg}
For a fixed point $x_0 \in \mr^n \simeq \pa \mr^N_+$ and $R > 0$, let
\begin{equation}\label{eq_B_R}
B_R = B^N_+((x_0,0),R) \subset \mr^N_+, \quad \pa B_R' = B^n(x_0,R) \subset \mr^n, \quad
\pa B_R'' = \pa_I B^N_+((x_0,0),R)
\end{equation}
so that $\pa B_R = \pa B_R' \cup \pa B_R''$. Suppose also that $\ga \in (0,1)$,
\begin{enumerate}
\item[(g1)] $\bg$ is a smooth metric on $\obr$ such that $\bg_{iN} = 0$, $\bg_{NN} = 1$
    and
    $\lambda_{\bg} |\xi|^2 \le \bg_{ij}(x) \xi_i \xi_j \le \Lambda_{\bg} |\xi|^2$ on $\obr$
    for some positive numbers $\lambda_{\bg} \le \Lambda_{\bg}$ and all vectors $\xi \in \mr^n$;
\item[(A1)] $A \in L^{2(n-2\ga+2)/(n-2\ga+4)}(B_R; x_N^{1-2\ga})$,  $Q \in L^1(B_R; x_N^{1-2\ga})$ and $F = (F_1, \cdots, F_n, F_N) \in L^1(B_R; x_N^{1-2\ga})$;
\item[(a1)] $a \in L^{2n/(n+2\ga)}(\pa B_R')$ and $q \in L^1(\pa B_R')$.
\end{enumerate}
In this section, we will examine regularity of a weak solution $U \in W^{1,2}(B_R; x_N^{1-2\ga})$ to a degenerate elliptic equation
\begin{equation}\label{eq_deg}
\begin{cases}
-\text{div}_{\bg} (x_N^{1-2\ga} \nabla U) + x_N^{1-2\ga} A U = x_N^{1-2\ga} Q + \text{div}(x_N^{1-2\ga} F) &\text{in } B_R,\\
U = u &\text{on } \pa B_R',\\
\pa^{\ga}_{\nu, F} U = au + q &\text{on } \pa B_R'
\end{cases}
\end{equation}
where
\[\pa^{\ga}_{\nu, F} U = \kappa_{\ga} \lim_{x_N \to 0+} x_N^{1-2\ga} \(F_N - {\pa U \over \pa x_N} \) = \pa^{\ga}_{\nu} U + \kappa_{\ga} \lim_{x_N \to 0+} x_N^{1-2\ga} F_N.\]
Notice that $\pa^{\ga}_{\nu} = \pa^{\ga}_{\nu,0}$.
We first recall the precise meaning of a weak solution to \eqref{eq_deg}.
\begin{defn}
Assume that (g1), (A1) and (a1) are valid. A function $U \in W^{1,2}(B_R; x_N^{1-2\ga})$ is called a {\it weak solution} to \eqref{eq_deg} if it satisfies
\begin{multline*}
\kappa_{\ga} \int_{B_R} x_N^{1-2\ga} \(\la \nabla U, \nabla \Phi \ra_{\bg} + AU\Phi\) dv_{\bg} \\
= \kappa_{\ga} \int_{B_R} x_N^{1-2\ga} \(Q\Phi - F_i \pa_i\Phi - F_N \pa_N \Phi\) dv_{\bg} + \int_{\pa B_R'} (au + q) \phi\, dv_{\bh}
\end{multline*}
for every $\Phi \in C^1(\obr)$ such that $\Phi = \phi$ on $\pa B_R'$ and $\Phi = 0$ on $\pa B_R''$.
Here $u$ is the trace of $U$ on $\pa B_R'$ and $\bh = \bg|_{\pa B_R'}$.
\end{defn}

\subsection{H\"older estimates}
By applying the Moser iteration technique, we can deduce H\"older estimates of weak solutions to \eqref{eq_deg}.
\begin{lemma}\label{lemma_reg_0}
Assume that the metric $\bg$ satisfies \textnormal{(g1)} and $U \in W^{1,2}(B_R; x_N^{1-2\ga})$ is a weak solution to \eqref{eq_deg}.
Suppose also that
\begin{enumerate}
\item[(A2)] $A, Q \in L^{q_1}(B_R; x_N^{1-2\ga})$ and $F \in L^{q_2}(B_R; x_N^{1-2\ga})$ for $q_1 > (n-2\ga+2)/2$ and $q_2 > n-2\ga+2$;
\item[(a2)] $a, q \in L^{q_3}(\pa B_R')$ for $q_3 > n/(2\ga)$.
\end{enumerate}
Then $U \in C^{\beta}(\obrt)$ and
\begin{multline}\label{eq_DNM}
\|U\|_{C^{\beta}(\obrt)} \le C \(\|U\|_{L^2(B_R; x_N^{1-2\ga})} + \|Q\|_{L^{q_1}(B_R; x_N^{1-2\ga})} \right. \\
\left. + \|F\|_{L^{q_2}(B_R; x_N^{1-2\ga})} + \|q\|_{L^{q_3}(\pa B_R')}\)
\end{multline}
where $C > 0$ and $\beta \in (0,1)$ depend only on $n$, $\ga$, $R$, $\lambda_{\bg}$, $\Lambda_{\bg}$, $\|A\|_{L^{q_1}(B_R; x_N^{1-2\ga})}$ and $\|a\|_{L^{q_3}(\pa B_R')}$.
Besides, if $U$ is nonnegative on $\obr$, then it is also true that
\begin{equation}\label{eq_har_2}
\sup_{B_{R/2}} U \le C \(\inf_{B_{R/2}} U + \|Q\|_{L^{q_1}(B_R; x_N^{1-2\ga})} + \|F\|_{L^{q_2}(B_R; x_N^{1-2\ga})} + \|q\|_{L^{q_3}(\pa B_R')}\)
\end{equation}
for some $C > 0$ depending only on $n$, $\ga$, $R$, $\lambda_{\bg}$, $\Lambda_{\bg}$, $\|A\|_{L^{q_1}(B_R; x_N^{1-2\ga})}$ and $\|a\|_{L^{q_3}(\pa B_R')}$.
\end{lemma}
\begin{proof}
Derivation of \eqref{eq_DNM} and \eqref{eq_har_2} can be found in Lemma 5.1 and Remark 5.2 of \cite{Ki}.
\end{proof}

\subsection{Derivative estimates}
If the functions $A, Q, a$ and $q$ have classical derivatives in the tangential direction,
weak solutions to \eqref{eq_deg} have higher differentiability in the same direction.
The following result is a huge improvement of \cite[Lemma 5.3]{Ki} in that a much milder condition on $\bg$ is imposed.
The reader is advised to see carefully why handling \eqref{eq_deg} becomes more difficult if $\bg$ is non-Euclidean and how it is resolved in the proof.
\begin{lemma}\label{lemma_reg_2}
Assume that the metric $\bg$ satisfies \textnormal{(g1)} and $U \in W^{1,2}(B_R; x_N^{1-2\ga})$ is a weak solution to \eqref{eq_deg}.
Suppose also that $\ell_0$ is $1$, $2$ or $3$, and
\begin{enumerate}
\item [(g2)] it holds that
\begin{equation}\label{eq_sqrt_bg}
\left| \nabla_{\bx}^{\ell} \pa_N \sqrt{|\bg|}(x) \right| \le C x_N \quad \text{on } \obr
\end{equation}
for any $\ell = 0, \cdots, \ell_0$;
\item[(A3a)] $A, \cdots, \nabla_{\bx}^{\ell_0-1} A, Q, \cdots, \nabla_{\bx}^{\ell_0-1} Q \in L^{\infty}(B_R)$
    and $F, \cdots, \nabla_{\bx}^{\ell_0-1}F \in C^{\beta'}(\obr)$ for $\beta' \in (0,1)$;
\item[(A3b)] $\nabla_{\bx}^{\ell_0} A, \nabla_{\bx}^{\ell_0} Q \in L^{q_1}(B_R; x_N^{1-2\ga})$
    and $\nabla_{\bx}^{\ell_0} F \in L^{q_2}(B_R; x_N^{1-2\ga})$ for $q_1 > (n-2\ga+2)/2$ and $q_2 > n-2\ga+2$;
\item[(a3)] $a, \cdots, \nabla_{\bx}^{\ell_0} a, q, \cdots, \nabla_{\bx}^{\ell_0} q \in L^{q_3}(\pa B_R')$ for $q_3 > n/(2\ga)$.
\end{enumerate}
Then $\nabla^{\ell_0}_{\bx} U \in C^{\beta}(\obrt)$ and
\begin{equation}\label{eq_reg_2}
\begin{aligned}
\|\nabla^{\ell_0}_{\bx} U\|_{C^{\beta}(\obrt)} &\le C \(\|U\|_{L^2(B_R; x_N^{1-2\ga})}
+ \sum_{\ell=1}^{\ell_0-1} \|\nabla_{\bx}^{\ell} A\|_{L^{\infty}(B_R)} + \|\nabla_{\bx}^{\ell_0} A\|_{L^{q_1}(B_R; x_N^{1-2\ga})} \right.\\
&\ + \sum_{\ell=0}^{\ell_0-1} \|\nabla_{\bx}^{\ell} Q\|_{L^{\infty}(B_R)}
+ \|\nabla_{\bx}^{\ell_0} Q\|_{L^{q_1}(B_R; x_N^{1-2\ga})} + \sum_{\ell=0}^{\ell_0-1} \|\nabla_{\bx}^{\ell} F\|_{C^{\beta'}(\obr)} \\
&\ \left. + \|\nabla_{\bx}^{\ell_0} F\|_{L^{q_2}(B_R)} + \sum_{\ell=1}^{\ell_0} \|\nabla_{\bx}^{\ell} a\|_{L^{q_3}(\pa B_R')}
+ \sum_{\ell=0}^{\ell_0} \|\nabla_{\bx}^{\ell} q\|_{L^{q_3}(\pa B_R')} \)
\end{aligned}
\end{equation}
for $C > 0$ and $\beta \in (0,1)$ relying only on $n$, $\ga$, $R$, $\bg$, $A$, $\|a\|_{L^{q_3}(\pa B_R')}$ and $\sum_{\ell=0}^{\ell_0-1} \|\nabla_{\bx}^{\ell} U\|_{L^{\infty}(B_R)}$.
\end{lemma}
\begin{proof}
Assuming that $\ell_0 = 1$, we shall derive \eqref{eq_reg_2}.
Given any vector $h \in \mr^n$ with small magnitude $|h|$, we define the {\it difference quotient} $D^h U$ of $U$ by
\[D^h U(\bx, x_N) = {U(\bx+h, x_N) - U(\bx, x_N) \over |h|} \quad \text{for } (\bx, x_N) \in B_{3R/4}.\]
Then it weakly solves
\begin{equation}\label{eq_reg_20}
\begin{cases}
-\text{div}_{\bg} (x_N^{1-2\ga} \nabla (D^h U)) + x_N^{1-2\ga} A (D^h U) = x_N^{1-2\ga} Q^* + \text{div} (x_N^{1-2\ga} F^*) &\text{in } B_{3R/4},\\
\pa^{\ga}_{\nu, F^*} (D^h U) = a (D^h U) + q^* &\text{on } \pa B_{3R/4}'
\end{cases}
\end{equation}
where $U^h(\bx, x_N) = U(\bx+h, x_N)$,
\begin{align*}
F^* &= \(D^h F_i + D^h \(\sqrt{|\bg|} \bg^{ij}\) \cdot \pa_j U^h, D^h F_N + D^h\sqrt{|\bg|} \cdot \pa_N U^h\), \\
Q^* &= D^hQ - D^hB \cdot U^h
\quad \text{and} \quad
q^* = D^h a \cdot U^h + D^h q.
\end{align*}

The most problematic term in analyzing \eqref{eq_reg_20} turns out to be $\text{div} (x_N^{1-2\ga} F^*)$,
especially, its subterm $\pa_N (x_N^{1-2\ga} D^h \sqrt{|\bg|} \cdot \pa_N U^h)$.
Let us concern it in depth.
If we fix a small number $\vep > 0$ and write
\[B_{3R/4, \vep} = B_{3R/4} \cap \{x_N > \vep\} \quad \text{and} \quad \pa B_{3R/4, \vep}' = B_{3R/4} \cap \{x_N = \vep\},\]
then an integration by parts shows
\begin{multline}\label{eq_reg_21}
\int_{B_{3R/4, \vep}} x_N^{1-2\ga} \(D^h\sqrt{|\bg|} \cdot \pa_N U^h\) \pa_N\Phi dx
= - \int_{B_{3R/4, \vep}} \(D^h\sqrt{|\bg|}\) \pa_N (x_N^{1-2\ga} \pa_NU^h) \Phi dx \\
- \int_{B_{3R/4, \vep}} x_N^{1-2\ga} \(\pa_N D^h\sqrt{|\bg|}\) \pa_NU^h \Phi dx
- \int_{\pa B_{3R/4, \vep}'} x_N^{1-2\ga} \(D^h\sqrt{|\bg|}\) \pa_NU^h \phi \, d\bx
\end{multline}
for any function $\Phi \in C^1(\overline{B_{3R/4}})$ such that $\Phi = \phi$ on $\pa B_{3R/4}'$ and $\Phi = 0$ on $\pa B_{3R/4}''$.
On the other hand, we obtain from \eqref{eq_deg} that
\begin{align*}
&\ \int_{B_{3R/4, \vep}} \sqrt{|\bg|}^h \pa_N (x_N^{1-2\ga} \pa_NU^h) \Phi dx \\
&= - \int_{B_{3R/4, \vep}} x_N^{1-2\ga} \(\sqrt{|\bg|} \bg^{ij}\)^h \pa_iU^h \pa_j \Phi dx
- \int_{B_{3R/4, \vep}} x_N^{1-2\ga} A^h U^h \Phi dx + \int_{B_{3R/4, \vep}} x_N^{1-2\ga} Q^h \Phi dx \\
&\ - \int_{B_{3R/4, \vep}} x_N^{1-2\ga} (F_i)^h \pa_i \Phi dx - \int_{B_{3R/4, \vep}} x_N^{1-2\ga} (F_N)^h \pa_N \Phi dx + \int_{\pa B_{3R/4, \vep}'} x_N^{1-2\ga} \sqrt{|\bg|}^h \pa_NU^h \phi d\bx \\
&\ + \int_{B_{3R/4, \vep}} x_N^{1-2\ga} \(\pa_N \sqrt{|\bg|}^h\) \pa_N U^h \Phi dx
+ \int_{\pa B_{3R/4, \vep}'} x_N^{1-2\ga}((F_N)^h - \pa_NU^h) \phi \, d\bx
\end{align*}
where $\sqrt{|\bg|}^h(\bx, x_N) = \sqrt{|\bg|}(\bx+h, x_N)$ and so on.
Consequently, after substituting
\[\Xi = \(\sqrt{|\bg|}^h\)^{-1} \(D^h\sqrt{|\bg|}\) \Phi\]
for $\Phi$ in the above identity,
combining the result with \eqref{eq_reg_21} and then taking $\vep \to 0$, we get
\begin{align}
&\ \int_{B_{3R/4}} x_N^{1-2\ga} \(D^h\sqrt{|\bg|} \cdot \pa_N U^h\) \pa_N\Phi dx \nonumber \\
&= - \int_{B_{3R/4}} x_N^{1-2\ga} \(\sqrt{|\bg|} \bg^{ij}\)^h \pa_iU^h \pa_j \Xi dx
- \int_{B_{3R/4}} x_N^{1-2\ga} A^h U^h \Xi dx + \int_{B_{3R/4}} x_N^{1-2\ga} Q^h \Xi dx \nonumber \\
&\ - \int_{B_{3R/4}} x_N^{1-2\ga} (F_i)^h \pa_i \Xi dx - \int_{B_{3R/4}} x_N^{1-2\ga} (F_N)^h \pa_N \Xi dx + \int_{B_{3R/4}} x_N^{1-2\ga} \(\pa_N \sqrt{|\bg|}\)^h \pa_N U^h \Xi dx \nonumber \\
&\ - \int_{B_{3R/4}} x_N^{1-2\ga} \(\pa_N D^h\sqrt{|\bg|}\) \pa_NU^h \Phi dx
+ \kappa_{\ga}^{-1} \int_{\pa B_{3R/4}'} (au+q) \phi\, d\bx. \label{eq_reg_22}
\end{align}
We have two remarks on \eqref{eq_reg_22}: First, $\Xi$ has the same regularity as that of $\Phi$ and vanishes on $\pa B_{3R/4}''$.
Second, by virtue of (g2), there exists a constant $C > 0$ such that
\[\left| \(\pa_N \sqrt{|\bg|}^h\) \pa_N U^h \right| + \left|\(\pa_N D^h\sqrt{|\bg|}\) \pa_NU^h \right| \le C(x_N \pa_N U)^h \quad \text{in } B_{3R/4}.\]
Furthermore, as pointed out in the proof of \cite[Lemma 5.3]{Ki}, an application of the rescaling argument gives
\[\|x_N \pa_N U\|_{L^{\infty}(B_{2R/3})} \le C\(\|U\|_{L^{\infty}(B_{3R/4})} + \|F\|_{C^{\beta'}(\obr)} + \|Q\|_{L^{\infty}(B_R)}\)\]
where $C > 0$ depends only $n$, $R$, $\bg$ and $\|A\|_{L^{\infty}(B_R)}$.
Therefore no terms in the right-hand side of \eqref{eq_reg_22} are harmful.

Now, we introduce a number
\[k = \begin{cases}
\begin{aligned}
\|\nabla_{\bx} A\|_{L^{q_1}(B_R; x_N^{1-2\ga})} +
\|\nabla_{\bx} Q\|_{L^{q_1}(B_R; x_N^{1-2\ga})} +
\|\nabla_{\bx} F\|_{L^{q_2}(B_R; x_N^{1-2\ga})} \\
\|x_N \pa_NU\|_{L^{\infty}(B_{2R/3})} + \|\nabla_{\bx} a\|_{L^{q_3}(\pa B_R')} + \|\nabla_{\bx} q\|_{L^{q_3}(\pa B_R')}
\end{aligned} &\text{if it is nonzero},\\
\text{any positive number} &\text{otherwise}.
\end{cases}\]
In the latter case, we send $k \to 0$ at the last stage. For a fixed $K > 0$ and $m \ge 0$, we define
\[V_h = (D^h U)_+ + k,\quad V_{h,K} = \min\{V_h, K\} \quad \text{and} \quad Z_{h,m} = V_h^{(m+2)/2}.\]
We test \eqref{eq_reg_20} with $\Phi = \tich (V_{h,K}^m V_h - k^{m+1})$ where $\tich \in C^{\infty}(\obr)$ denotes a suitable cut-off function.
Employing \eqref{eq_reg_22}, H\"older's inequality, Young's inequality,
the weighted Sobolev inequality and the weighted Sobolev trace inequality (see \eqref{eq_emb})
and then taking $K \to \infty$, we derive
\begin{multline}\label{eq_reg_23}
\|\nabla(\tich Z_{h,m})\|_{L^2(B_R; x_N^{1-2\ga})}^2
\le Cm^{\eta} \left[ \int_{B_R} x_N^{1-2\ga} (\tich^2 + |\nabla \tich|^2) Z_{h,m}^2 dx \right. \\
\left. + \int_{B_R} x_N^{1-2\ga} \tich^2 V_h^m (|\nabla_{\bx} U_h|^2 + \|x_N \pa_N U\|_{L^{\infty}(B_{2R/3})}^2) dx \right]
\end{multline}
for some $C > 0$ depending only on $n$, $\ga$, $R$, $\bg$, $A$, $\|a\|_{L^{q_3}(\pa B_R')}$ and $\|U\|_{L^{\infty}(B_R)}$,
and $\eta > 1$ depending only on $n$ and $\ga$;
refer to the proofs of \cite[Proposition 1]{FF} and \cite[Lemma 5.3]{Ki} which provide more detailed descriptions.
Combining \eqref{eq_reg_23} with the corresponding inequality for $(D^hU)_- + k$ and letting $h \to 0$, we see
\[\||\nabla_{\bx} U| + k\|_{L^{(m+2)\({n-2\ga+2 \over n-2\ga}\)}(B_R \cap \{\tich = 1\}; x_N^{1-2\ga})}^{m+2}
\le Cm^{\eta} \|\nabla \tich\|_{L^{\infty}(B_R)}^2 \||\nabla_{\bx} U| + k\|_{L^{(m+2)}(B_R; x_N^{1-2\ga})}^{m+2}.\]
Hence the Moser iteration argument implies that \[\|\nabla_{\bx} U\|_{L^{\infty}(B_{R/2})} \le \text{(the right-hand side of \eqref{eq_reg_2} with } \ell_0 = 1).\]

Similarly, one can obtain the weak Harnack inequality as well as the H\"older estimate for $\nabla_{\bx} U$.
The cases $\ell_0 = 2$ or $3$ can be also treated. We omit the details.
\end{proof}

In the following lemma, we take into account H\"older regularity of the weighted derivative $x_N^{1-2\ga} \pa_N U$ of a weak solution $U$ to \eqref{eq_deg}.
\begin{lemma}\label{lemma_reg_3}
Suppose that the metric $\bg$ satisfies \textnormal{(g1)} and $U \in W^{1,2}(B_R; x_N^{1-2\ga})$ is a weak solution to \eqref{eq_deg}
such that $U, \nabla_{\bx}U, \nabla_{\bx}^2 U \in C^{\beta}(\obr)$ for some $\beta \in (0,1)$.
Furthermore, assume that the following conditions hold:
\begin{enumerate}
\item[(A4)] $A \in C^{\beta}(\obr)$, $\sup_{x_N \in (0,R)} (\|Q(\cdot, x_N)\|_{C^{\beta}(\opbr)}
    + \|\pa_i F_i(\cdot, x_N)\|_{C^{\beta}(\opbr)}) < \infty$ and $F_N = 0$;
\item[(a4)] $a, q \in C^{\beta}(\opbr)$.
\end{enumerate}
Then $x_N^{1-2\ga} \pa_N U \in C^{\min\{\beta, 2-2\ga\}}(\obrt)$ and
\begin{multline}\label{eq_reg_3}
\left\| x_N^{1-2\ga} \pa_N U \right\|_{C^{\min\{\beta, 2-2\ga\}}(\obrt)}
\le C \( \sum_{\ell = 0}^2 \|\nabla_{\bx}^{\ell} U\|_{C^{\beta}(\obr)} \right. \\
\left. + \sup_{x_N \in (0,R)} (\|Q(\cdot, x_N)\|_{C^{\beta}(\opbr)} + \|\pa_i F_i(\cdot, x_N)\|_{C^{\beta}(\opbr)}) + \|q\|_{C^{\beta}(\opbr)}\)
\end{multline}
for $C > 0$ depending only on $n$, $\ga$, $R$, $\bg$, $\|A\|_{C^{\beta}(\obr)}$ and $\|a\|_{C^{\beta}(\opbr)}$.
\end{lemma}
\begin{proof}
Refer to \cite[Lemma 5.5]{Ki}.
\end{proof}

\subsection{Two maximum principles}
In this part, we list two maximum principles which are used throughout the paper.

\medskip
The following lemma describes the generalized maximum principle for degenerate elliptic equations.
\begin{lemma}\label{lemma_max}
Suppose that $A \in L^{\infty}(B_R)$, $a \in L^{\infty}(\pa B_R'')$ and there exists a function $V \in C^0(\obr) \cup C^2(B_R)$
such that $\nabla_{\bx} V,\, x_N^{1-2\ga} \pa_N V \in C^0(\obr)$ and
\[\begin{cases}
-\textnormal{div} (x_N^{1-2\ga}\nabla V) + x_N^{1-2\ga} A V \ge 0 &\text{in } B_R,\\
V > 0 &\text{on } \obr,\\
\pa_{\nu}^{\ga} V + aV \ge 0 &\text{on } \pa B_R'.
\end{cases}\]
If $U \in C^0(\obr) \cup C^2(B_R)$ satisfies $\nabla_{\bx} U,\, x_N^{1-2\ga} \pa_N U \in C^0(\obr)$ and solves
\begin{equation}\label{eq_max}
\begin{cases}
-\textnormal{div} (x_N^{1-2\ga}\nabla U) + x_N^{1-2\ga} A U \ge 0 &\text{in } B_R,\\
\pa_{\nu}^{\ga} U + aU \ge 0 &\text{on } \pa B_R',\\
U \ge 0 &\text{on } \pa B_R'',
\end{cases}
\end{equation}
then $U \ge 0$ on $\obr$.
Furthermore, the same conclusion holds if $B_R$ and $\pa B_R'$ are substituted by $\mr^N_+$ and $\mr^n$, respectively,
and the third inequality in \eqref{eq_max} is replaced with the condition that $|U(x)|/V(x) \to 0$ uniformly as $|x| \to \infty$.
\end{lemma}
\begin{proof}
One can obtain it by suitably modifying the proofs of \cite[Lemma A.3]{JLX} and \cite[Lemma 3.7]{KMW}.
\end{proof}

The next remark concerns on the weak maximum principle when the size of the domain is sufficiently small.
\begin{rmk}\label{rmk_max_p}
For any fixed $R > 0$, we introduce the space
\begin{equation}\label{eq_mcw}
\mcw_0^{1,2}(B_R; x_N^{1-2\ga}) = \left\{ U \in W^{1,2}(B_R; x_N^{1-2\ga}): U = 0 \text{ on } \pa B_R'' \right\},
\end{equation}
endowed with the standard $W^{1,2}(B_R; x_N^{1-2\ga})$-norm.
As shown in \cite[Lemma 2.1.2]{DMV}, the map $D^{1,2}\mrg \hookrightarrow L^2(B_R; x_N^{1-2\ga})$ is compact.
Therefore a minimizer of the Rayleigh quotient
\[\lambda_1(R) = \inf_{U \in \mcw_0^{1,2}(B_R; x_N^{1-2\ga}) \setminus \{0\}} {\int_{B_R} x_N^{1-2\ga} |\nabla U|^2 dx \over \int_{B_R} x_N^{1-2\ga} U^2 dx}\]
is always attained, and so $\lambda_1(R) > 0$.
Moreover, we see from the dilation symmetry that
\[\lambda_1(R) = \({R' \over R}\)^2 \lambda_1(R') \quad \text{for any } 0 < R < R'.\]
Therefore if $|A| \le \mcm$ for some constant $\mcm > 0$, then there exists $R_0' = R_0'(\mcm, \bg) > 0$ such that
\begin{equation}\label{eq_*}
\|U\|_* = \(\int_{B_R} x_N^{1-2\ga} \(|\nabla U|_{\bg}^2 + A U^2\) dv_{\bg}\)^{1 \over 2} \quad \text{for } U \in \mcw_0^{1,2}(B_R; x_N^{1-2\ga})
\end{equation}
is a norm equivalent to the $\mcw_0^{1,2}(B_R; x_N^{1-2\ga})$-norm for $R \in (0,R_0')$.

In particular, we have a weak maximum principle:
Given any $R \in (0,R_0')$, suppose that $U \in W^{1,2}(B_R; x_N^{1-2\ga})$ satisfies
\begin{equation}\label{eq_ineq}
\begin{cases}
-\text{div}_{\bg} (x_N^{1-2\ga}\nabla U) + x_N^{1-2\ga} A U \ge 0 &\text{in } B_R,\\
\pa^{\ga}_{\nu} U \ge 0 &\text{on } \pa B_R',\\
U \ge 0 &\text{on } \pa B_R''.
\end{cases}
\end{equation}
Then $U \ge 0$ in $B_R$. To check it, we just put a test function $U_- \in W^{1,2}(B_R; x_N^{1-2\ga})$ into \eqref{eq_ineq} and use the equivalence between the $*$-norm and the $\mcw_0^{1,2}(B_R; x_N^{1-2\ga})$-norm.
\end{rmk}

\subsection{Schauder estimates}
In this subsection, we prove the Schauder estimate for solutions to
\begin{equation}\label{eq_deg_2}
\begin{cases}
-\text{div}_{\bg} (x_N^{1-2\ga} \nabla U) + x_N^{1-2\ga} A U = 0 &\text{in } B_R,\\
U = u > 0 &\text{on } \pa B_R',\\
\pa^{\ga}_{\nu} U = q &\text{on } \pa B_R',
\end{cases}
\end{equation}
which is a special case of \eqref{eq_deg}.
\begin{lemma}\label{lemma_Sch}
Assume that the metric $\bg$ satisfies \textnormal{(g1)} and \textnormal{(g2)}, and $q \in C^{\beta}(\overline{B_R})$ for some $0 < \beta \notin \mn$.
Suppose also that
\begin{enumerate}
\item[(g3)] $\bg_{ij}(0) = \delta_{ij}$ where $\delta_{ij}$ is the Kronecker delta;
\item[(A3c)] $A, \cdots, \nabla_{\bx}^{\lceil \beta \rceil} A \in L^{\infty}(B_R)$ and $\beta' \in (0, \beta+2\ga] \cap (0, \lceil \beta \rceil)$.
\end{enumerate}
If $U \in W^{1,2}(B_R; x_N^{1-2\ga})$ is a weak solution to \eqref{eq_deg_2}, then $u \in C^{\beta'}(\overline{\pa B_{R/2}'})$ and
\begin{equation}\label{eq_Sch}
\|u\|_{C^{\beta'}(\overline{\pa B_{R/2}'})}
\le C \(\|U\|_{L^2(B_R; x_N^{1-2\ga})} + \|q\|_{C^{\beta}(\overline{\pa B_R'})} + \sum_{\ell=0}^{\lceil \beta \rceil} \|\nabla_{\bx}^{\ell} A\|_{L^{\infty}(B_R)}\)
\end{equation}
Here $\lceil \beta \rceil$ is the smallest integer exceeding $\beta$ and $C > 0$ depends only on $n$, $\gamma$, $\beta$, $R$, $\bg$ and $A$.
\end{lemma}
\begin{proof}
We shall follow closely the argument in the proof of \cite[Theorem 2.14]{JLX}.

Considering a finite open cover of $B_R$ which consists of balls and half-balls with small diameters,
we may assume that $R > 0$ is so small that the $*$-norm in \eqref{eq_*} is equivalent to the standard $W^{1,2}(B_R; x_N^{1-2\ga})$-norm.
For simplicity, we set
\[M = \|q\|_{C^{\beta}(\opbr)} + \sum_{\ell=1}^{\lceil \beta \rceil} \|\nabla_{\bx}^{\ell} A\|_{L^{\infty}(B_R)}.\]

Assume that $\beta \in (0,1)$. For $m \in \mn$, let $W_m$ be the unique solution in $W^{1,2}(B_{R/2^m}; x_N^{1-2\ga})$ to
\begin{equation}\label{eq_Sch_2}
\begin{cases}
-\text{div}_{\bg}\(x_N^{1-2\ga}\nabla W_m\) + x_N^{1-2\ga} A W_m = 0 &\text{in } B_{R/2^m},\\
\pa^{\ga}_{\nu} W_m = q(0) - q(\bx) &\text{on } \pa B_{R/2^m}',\\
W_m = 0 &\text{on } \pa B_{R/2^m}''.
\end{cases}
\end{equation}
Then an application of the weak maximum principle (Remark \ref{rmk_max_p}) to the equation of the  function
\[2^{2\ga m} W_m\({x \over 2^m}\) \pm {M R^\beta \over 2^{\beta m}} \left[{2R^2-|x|^2 \over n+2-2\ga} + {2R^{2\ga}-x_N^{2\ga} \over 2\ga \kappa_{\ga}} \right]\]
shows
\begin{equation}\label{eq_Sch_3}
\|W_m\|_{L^{\infty}(B_{R/2^m})} \le {C M \over 2^{(\beta + 2\ga)m}}
\end{equation}
for every $m \in \mn$. Define $h_m = W_{m+1} - W_m$.
Thanks to Lemmas \ref{lemma_reg_0} and \ref{lemma_reg_2}, we have
\begin{equation}\label{eq_Sch_4}
\|\nabla_{\bx}^{\ell} (U+W_0)\|_{L^{\infty}(B_{R/2})} \le C(\|U\|_{L^2(B_R;x_N^{1-2\ga})} + M)
\end{equation}
and
\begin{equation}\label{eq_Sch_5}
\|\nabla_{\bx}^{\ell} h_m\|_{L^{\infty}(B_{R/2^{m+2}})} \le {C M \over 2^{(\beta + 2\ga-\ell)m}}
\end{equation}
for $\ell = 0,\, 1$ and all $m \in \mn$. By \eqref{eq_Sch_3}-\eqref{eq_Sch_5} and the mean value theorem,
\begin{align*}
|u(\bx) - u(0)| &\le |W_m(0,0)| + |W_m(\bx,0)| + |(U+W_0)(\bx,0) - (U+W_0)(0,0)| \\
&\ + \sum_{j=0}^{m-1} |h_j(\bx,0)-h_j(0,0)|  \\
&\le C(\|U\|_{L^2(B_R;x_N^{1-2\ga})} + M) |\bx|^{\min\{\beta + 2\ga,1\}} \quad \text{for all } \bx \in \pa B_{R/2}'
\end{align*}
so that $u \in C^{\beta'}(\overline{\pa B_{R/2}'})$ and \eqref{eq_Sch} holds.

Suppose $\beta \in (1,2)$. In this case, we modify $W_m$ by replacing the second equation of \eqref{eq_Sch_2} with
\[\pa^{\ga}_{\nu} W_m = q(0) + \nabla_{\bx} q(0) \cdot \bx - q(\bx) \quad \text{on } \pa B_{R/2^m}'.\]
Then we adopt the above argument to conclude that $u \in C^{\beta'}(\overline{\pa B_{R/2}'})$ and \eqref{eq_Sch} holds.

The case $\beta > 2$ can be similarly treated. This finishes the proof.
\end{proof}

\subsection{Conclusion}
From the results obtained in the previous subsections, one gets the following regularity property of solutions to \eqref{eq_Yamabe} and its extension problem \eqref{eq_Yamabe_2}.
\begin{prop}\label{prop_reg}
Suppose that $U \in W^{1,2}(X;\rho^{1-2\ga})$ is a weak solution of \eqref{eq_Yamabe_2} with a fixed $p \in (1, 2_{n,\ga}^*-1]$ and condition \eqref{eq_hyp} holds.
Then the functions $U$, $\nabla_{\bx} U$, $\nabla_{\bx}^2 U$ and $x_N^{1-2\ga} \pa_N U$ are H\"older continuous on $\ox$.
In particular, the trace $u \in C^2(M)$ of $U$ on $M$ satisfies \eqref{eq_Yamabe} in the classical sense.
\end{prop}
\begin{proof}
The above regularity properties are local. Therefore Lemmas \ref{lemma_reg_0}, \ref{lemma_reg_2}, \ref{lemma_reg_3} and \ref{lemma_Sch} ensure their validity.
\end{proof}

\section{Green's function and B\^ocher's theorem}
In this section, we keep using the notations given in \eqref{eq_B_R}.

\subsection{Green's function}\label{subsec_Green}
Given a small number $R > 0$, we shall examine the existence and the growth rate of the {\it Green's function $G$} in $B_R$, a solution of
\begin{equation}\label{eq_Green}
\begin{cases}
-\text{div}_{\bg} (x_N^{1-2\ga} \nabla G) + E_{\bg}(x_N)\, G = 0 &\text{in } B_R,\\
\pa_{\nu}^{\ga} G = \delta_0 &\text{on } \pa B_R',\\
G = 0 &\text{on } \pa B_R''
\end{cases}
\end{equation}
where $\bg$ is the metric satisfying (g1) and $\delta_0$ is the Dirac measure centered at $0 \in \mr^N$.
Our argument is based on elliptic regularity theory and does not rely on parametrices.

\medskip
We start with deriving an auxiliary lemma.
\begin{lemma}\label{lemma_reg_1}
Given a small $R > 0$, suppose that $a = q = 0$ on $\pa B_R'$, $A \in L^{\infty}(B_R)$ and $Q \in L^{q_1}(B_R; x_N^{1-2\ga})$ for
\begin{equation}\label{eq_q_1}
q_1 \in \left[{2(n-2\ga+2) \over n-2\ga+4}, {n-2\ga+2 \over 2}\).
\end{equation}
Assume also that $U \in \mcw_0^{1,2}(B_R; x_N^{1-2\ga})$ is a weak solution to \eqref{eq_deg}
and the $*$-norm is equivalent to the $\mcw_0^{1,2}(B_R; x_N^{1-2\ga})$-norm; see \eqref{eq_mcw} and \eqref{eq_*}).
Then
\[\|U\|_{L^{q_4}(B_R; x_N^{1-2\ga})} \le C \|Q\|_{L^{q_1}(B_R; x_N^{1-2\ga})}\]
for any pair $(q_1, q_4)$ satisfying $1/q_1 = 1/q_4 + 2/(n-2\ga+2)$
and some constant $C > 0$ depending only on $n$, $\ga$, $R$, $\lambda_{\bg}$, $\Lambda_{\bg}$, $q_1$ and $q_4$.
\end{lemma}
\begin{proof}
One can follow the lines in the proof of \cite[Lemma 3.3]{CK}.
The main difference is that we need to apply the Sobolev inequality here instead of the Sobolev traced inequality as used in the reference; see \eqref{eq_emb}.
\end{proof}

Appealing to the previous lemma, we prove the main result in this subsection.
\begin{prop}\label{prop_loc_G}
Assume that $n \ge 2 + 2\ga$.
Then there exists $0 < R_0 \le \min\{R_0', r_1\}$ small (refer to Remark \ref{rmk_max_p} and Subsection \ref{subsec_poho})
such that \eqref{eq_Green} possesses a unique solution $G \in W^{1,2}_{\loc}(B_R \setminus \{0\}; x_N^{1-2\ga})$ satisfying
\begin{equation}\label{eq_Green_dec}
|x|^{n-2\ga} G(x) \to g_{n,\ga} \quad \text{uniformly as } |x| \to 0
\quad \text{where } g_{n,\ga} = {\Gamma\({n-2\ga \over 2}\) \over \pi^{n/2} 2^{2\ga} \Gamma(\ga)} > 0
\end{equation}
for any fixed $R \in (0, R_0)$.
\end{prop}
\begin{proof}
The proof is divided into four steps.

\medskip \noindent \textsc{Step 1 (Existence).}
By using \eqref{eq_hyp} and \eqref{eq_E_exp}, we rewrite the first equation of \eqref{eq_Green} as
\[-\text{div}_{\bg} (x_N^{1-2\ga} \nabla G) + x_N^{1-2\ga} A G = 0 \quad \text{in } B_R\]
where $A \in C^2(\obr)$.
In view of Remark \ref{rmk_max_p}, there exists $R_0 > 0$ such that $\|\cdot\|_{*}$ in \eqref{eq_*} serves as a norm
equivalent to the standard $\mcw^{1,2}(B_R; x_N^{1-2\ga})$-norm for all $R \in (0, R_0)$.
Then a duality argument in the proof of \cite[Lemma 4.2]{KMW2} shows that the desired function $G$ exists
and is contained in $W^{1,2}_{\loc}(B_R \setminus \{0\}; x_N^{1-2\ga}) \cap W^{1,q}(B_R; x_N^{1-2\ga})$ for any $1 < q < (n-2\ga+2)/(n-2\ga+1)$.

\medskip \noindent \textsc{Step 2 (Regularity).} Recall that
\begin{equation}\label{eq_Green_mr^N}
G_{\mr^N_+}(x) = {g_{n,\ga} \over |x|^{n-2\ga}} \quad \text{in } \mr^N_+
\end{equation}
solves
\[\begin{cases}
-\text{div} (x_N^{1-2\ga} \nabla G_{\mr^N_+}) = 0 &\text{in } \mr^N_+,\\
\pa_{\nu}^{\ga} G_{\mr^N_+} = \delta_0 &\text{on } \mr^n.
\end{cases}\]
Hence $G$ satisfies \eqref{eq_Green} if and only if $\mch = G_{\mr^N_+} - G - g_{n,\ga} R^{-(n-2\ga)}$ is a solution of
\begin{equation}\label{eq_mch}
\begin{cases}
-\text{div} (x_N^{1-2\ga} \nabla \mch) + x_N^{1-2\ga} A \mch = x_N^{1-2\ga} \( \mcq(G_{\mr^N_+}) - A g_{n,\ga} R^{-(n-2\ga)}\) &\text{in } B_R,\\
\pa_{\nu}^{\ga} \mch = 0 &\text{on } \pa B_R',\\
\mch = 0 &\text{on } \pa B_R''
\end{cases}
\end{equation}
where
\begin{multline*}
\mcq(U) = - (\bg^{ij} - \delta^{ij})\, \pa_{ij} U - \({\pa_i \sqrt{|\bg|} \over \sqrt{|\bg|}}\) \bg^{ij}\, \pa_jU - \({\pa_N \sqrt{|\bg|} \over \sqrt{|\bg|}}\) \pa_NU \\
- \pa_i \bg^{ij} \, \pa_jU + \({n-2\ga \over 4n}\) (R[\bg] + o(1)) U.
\end{multline*}
By a direct calculation, we see
\begin{equation}\label{eq_mch_mcq_2}
\mch \in W^{1,2}_{\loc}(B_R \setminus \{0\}; x_N^{1-2\ga}) \cap W^{1,q}(B_R; x_N^{1-2\ga})
\end{equation}
and
\begin{equation}\label{eq_mch_mcq}
|\mcq(G_{\mr^N_+})| \le {C \over |x|^{n-2\ga+1}} \in L^q(B_R; x_N^{1-2\ga})
\end{equation}
for all $1 < q < (n-2\ga+2)/(n-2\ga+1)$.
We claim that
\begin{equation}\label{eq_mch_i}
\|\mch\|_{L^{q'}(B_R; x_N^{1-2\ga})} < \infty
\quad \text{whenever } n > 2+2\ga \text{ and } 1 < q' < \frac{n-2\ga+2}{n-2\ga-1}.
\end{equation}

To justify it, we consider the formal adjoint of \eqref{eq_mch}
\begin{equation}\label{eq_dual}
\begin{cases}
-\text{div} (x_N^{1-2\ga}\nabla U) + x_N^{1-2\ga} A U = x_N^{1-2\ga} Q &\text{in } B_R,\\
\pa^{\ga}_{\nu} U = 0 &\text{on } \pa B_R',\\
U = 0 &\text{on } \pa B_R''
\end{cases}
\end{equation}
where $Q$ is an arbitrary function of class $C^1(\obr)$. Then
\begin{align*}
(x_N^{1-2\ga} \pa_N U)(\bx, x_N) - (x_N^{1-2\ga} \pa_N U)(\bx, \vep)
&= \int_{\vep}^{x_N} \pa_N (x_N^{1-2\ga} \pa_N U) dx_N \\
&= - \int_{\vep}^{x_N} x_N^{1-2\ga} \(\Delta_{\bx} U - AU + Q\) dx_N
\end{align*}
for any small $\vep > 0$. Since $(x_N^{1-2\ga} \pa_N U)(\bx, \vep) \to 0$ as $\vep \to 0$ thanks to the boundary condition
and $\Delta_{\bx} U \in C^0(\obr)$ in light of Lemma \ref{lemma_reg_2}, we observe
\[|\pa_N U (\bx, x_N)| \le Cx_N \quad \text{in } B_R\]
and in particular $U \in C^1(\obr)$.
Thus one may use $\mch$ and $U$ as a test function for \eqref{eq_dual} and \eqref{eq_mch}, respectively.
As a consequence, it holds that
\begin{align*}
\int_{B_R} x_N^{1-2\ga} \mch Q\, dx
&= \int_{B_R} x_N^{1-2\ga} \(\nabla \mch \cdot \nabla U + A \mch U\) dx \\
&= \int_{B_R} x_N^{1-2\ga} \(\mcq(G_{\mr^N_+}) - A g_{n,\ga} r^{-(n-2\ga)}\) U dx \\
&\le \left\|\mcq(G_{\mr^N_+}) - A g_{n,\ga} r^{-(n-2\ga)}\right\|_{L^{{n-2\ga+2 \over n-2\ga+1}-\eta}(B_R; x_N^{1-2\ga})}
\|U\|_{L^{n-2\ga+2+\eta'}(B_R; x_N^{1-2\ga})} \\
&\le C \|Q\|_{L^{{n-2\ga+2 \over 3} + \eta''}(B_R; x_N^{1-2\ga})}
\end{align*}
for any $Q \in C^1(\obr)$ and small $\eta > 0$.
Here $\eta'$ and $\eta''$ are small positive numbers depending only on $\eta$.
Also, the last inequality is due to \eqref{eq_mch_mcq} and Lemma \ref{lemma_reg_1},
and the assumption $n \ge 2+2\ga$ is required to ensure that $q_1 = (n-2\ga+2)/3 + \eta''$ satisfies condition \eqref{eq_q_1}.
By duality, the assertion follows.

\medskip \noindent \textsc{Step 3 (Blow-up rate).}
We use the rescaling argument in the proof of \cite[Proposition B.1]{LZhu2}. For $0 < R' < R/3$, we define
\[\wth(x) = (R')^{n-2\ga} \mch(R'x) \quad \text{in } B_3 \setminus \overline{B_{1/3}}.\]
It clearly solves
\[\begin{cases}
-\text{div} (x_N^{1-2\ga} \nabla \wth) + x_N^{1-2\ga} (R')^2 A(R'x) \wth = x_N^{1-2\ga} O(R') &\text{in } B_3 \setminus \overline{B_{1/3}},\\
\pa^{\ga}_{\nu} \wth = 0 &\text{on } \pa B_3' \setminus \overline{\pa B_{1/3}'}
\end{cases}\]
where $O(R') \le CR'$.
The local integrability condition \eqref{eq_mch_mcq_2}, Lemma \ref{lemma_reg_1} and estimate \eqref{eq_mch_i} show that
\begin{align*}
\|\wth\|_{L^{\infty}(B_2 \setminus B_{1/2})}
&\le C \(\|\wth\|_{L^{n-2\ga+2+\eta \over n-2\ga}(B_3 \setminus \overline{B_{1/3}}; x_N^{1-2\ga})} + O(R')\) \\
&= C\((R')^{\eta'} \|\mch\|_{L^{n-2\ga+2+\eta \over n-2\ga}(B_R; x_N^{1-2\ga})} + O(R')\) \le C(R')^{\eta'}
\end{align*}
for a small $\eta > 0$ and $\eta' = (n-2\ga)\eta/(n-2\ga+2+\eta)$. Therefore we have
\begin{equation}\label{eq_mch_dec}
|\mch(x)| \le C |x|^{-(n-2\ga)+\eta'} \quad \text{in } B_{2R/3}.
\end{equation}
On the other hand, by virtue of \eqref{eq_mch_mcq_2} and \eqref{eq_mch_mcq}, we can apply Lemma \ref{lemma_reg_1} to \eqref{eq_mch}, getting
\begin{equation}\label{eq_mch_dec_2}
|\mch(x)| \le C \quad \text{in } B_R \setminus \overline{B_{2R/3}}.
\end{equation}
Putting \eqref{eq_Green_mr^N}, \eqref{eq_mch_dec} and \eqref{eq_mch_dec_2} together gives the blow-up rate \eqref{eq_Green_dec} of $G$.

\medskip \noindent \textsc{Step 4 (Uniqueness).}
The uniqueness of $G$ follows from B\^ocher's theorem stated in Proposition \ref{prop_boc}.
\end{proof}

\begin{cor}
Assume that $n \ge 2 + 2\ga$.
The regular part $\mch$ of the Green's function $G$ defined in the proof of the previous proposition satisfies
\[|\nabla_{\bx} \mch(x)| \le C {|x|^{\eta'} \over |x|^{n-2\ga+1}}
\quad \text{and} \quad
\left| x_N^{1-2\ga} \pa_N \mch(x) \right| \le C {|x|^{\eta'} \over |x|^n}\]
for any fixed $R \in (0, R_0)$ and small $\eta' > 0$.
\end{cor}
\begin{proof}
The result follows immediately from \eqref{eq_mch_dec} and the rescaling argument.
\end{proof}

\subsection{The proof of B\^ocher's theorem}\label{subsec_Bocher}
We present the following version of a fractional B\^ocher's theorem, which is needed in the proof of Proposition \ref{prop_iso}.
The Euclidean case was considered in \cite[Lemma 4.10]{JLX} and \cite[Proposition 3.4]{NPX}.
\begin{prop}\label{prop_boc}
Fix any $R \in (0,R_0)$.
Suppose that the metric $\bg$ satisfies \textnormal{(g1)} and \textnormal{(g2)},
$\|A\|_{L^{\infty}(B_{R_0})} \le \mcm$, and $\nabla_{\bx} A \in L^{q_1}(B_{R_0}; x_N^{1-2\ga})$ for $q_1 > (n-2\ga+2)/2$.
If a function $U$ is nonnegative in $B_R \setminus \{0\}$, belongs to $W^{1,2}(B_R \setminus \overline{B_{\vth}'}; x_N^{1-2\ga})$ and weakly solves
\[\begin{cases}
-\textnormal{div}_{\bg} (x_N^{1-2\ga}\nabla U) + x_N^{1-2\ga} A U = 0 &\text{in } B_R \setminus \overline{B_{\vth}'},\\
\pa^{\ga}_{\nu} U = 0 &\text{on } \pa B_R' \setminus \overline{\pa B_{\vth}'}
\end{cases}\]
for all $\vth \in (0,R)$, then for any $R' \in (0,R)$
\[U = c_1 G + E \quad \text{in } B_{R'} \setminus \{0\}.\]
Here $c_1$ is a nonnegative constant,
$G$ is the Green's function that satisfies \eqref{eq_Green} (where $R$ is replaced with $R'$)
and $E \in W^{1,2}(B_R; x_N^{1-2\ga})$ solves
\[\begin{cases}
-\textnormal{div}_{\bg} (x_N^{1-2\ga}\nabla E) + x_N^{1-2\ga} A E = 0 &\text{in } B_{R'},\\
\pa^{\ga}_{\nu} E = 0 &\text{on } \pa B_{R'}'.
\end{cases}\]
\end{prop}
\noindent The numbers $R_0$ and $\mcm$ were chosen in Proposition \ref{prop_loc_G} and Remark \ref{rmk_max_p}.
To prove the proposition, we will use the strategy from \cite[Section 9]{LZhu} and \cite[Section 3]{NPX}.
As a preliminary step, we derive two results.
\begin{lemma}\label{lemma_boc_1}
Assume that $U$ satisfies all the conditions in Proposition \ref{prop_boc}. If $U(x) = o(|x|^{-(n-2\ga)})$ as $|x| \to 0$, then 0 is a removable singularity of $U$
and there exists $\beta \in (0,1)$ such that $U \in W^{1,2}(B_{R'}; x_N^{1-2\ga}) \cap C^{\beta}(\overline{B_{R'}})$ for any $R' \in (0,R)$.
\end{lemma}
\begin{proof}
We argue as in \cite[Lemma 3.6]{NPX} with minor modifications.
The maximum principle in Remark \ref{rmk_max_p} combined with the asymptotic behavior \eqref{eq_Green_dec} of $G$ near the origin shows that $U$ is bounded in $B_{R'}$.
Owing to the regularity hypotheses on $\bg$ and $A$, we can apply the scaling method with Lemmas \ref{lemma_reg_2} and \ref{lemma_reg_3},
deducing $U \in C^{\beta}(\overline{B_{R'}})$ for some $\beta \in (0,1)$ and
\[|x||\nabla_{\bx} U(x)| + |x|^{2\ga} \left| x_N^{1-2\ga}\pa_N U(x) \right| \le C \quad \text{in } \overline{B_{R'}}.\]
This in turn implies that $U \in W^{1,2}(B_{R'}; x_N^{1-2\ga})$.
\end{proof}
\begin{lemma}\label{lemma_boc_2}
Assume that $U$ satisfies all the conditions in Proposition \ref{prop_boc}. Then
\[\limsup_{r \to 0+} \max_{|x|=r} |x|^{n-2\ga} U(x) < \infty.\]
\end{lemma}
\begin{proof}
It can be checked as in \cite[Lemma 9.3]{LZhu} or \cite[Lemma 3.7]{NPX}.
\end{proof}

\begin{proof}[Proof of Proposition \ref{prop_boc}]
One can carry out the proof by adapting the ideas of \cite[Proposition 9.1]{LZhu} or \cite[Proposition 3.4]{NPX}.
Lemmas \ref{lemma_boc_1} and \ref{lemma_boc_2} are required.
\end{proof}

\section{Computation of the integrals involving the standard bubble}\label{sec_W_int}
We obtain the values of several integrals involving the standard bubble $W_{1,0}$ and its derivatives, which are needed in the proof of the vanishing theorem in Section \ref{sec_van}.

\begin{prop}\label{prop_W_int}
Suppose $\ga \in (0,1)$ and $n > 2 + 2\ga$. Then there exists a constant $\mcc_0 > 0$ depending only on $n$ and $\ga$ such that
\begin{equation}\label{eq_W_int}
\begin{aligned}
\int_{\mr^N_+} x_N^{3-2\ga} \Delta_{\bx} W_{1,0} Z_{1,0}^0 dx &= \mcc_0,\\
\int_{\mr^N_+} x_N^{2-2\ga} \pa_N W_{1,0} Z_{1,0}^0 dx &= \left[{3 \over 2(1+\ga)}\right] \mcc_0,\\
\int_{\mr^N_+} x_N^{1-2\ga} W_{1,0} Z_{1,0}^0 dx &= -\left[{3 \over 2(1-\ga^2) }\right] \mcc_0.
\end{aligned}
\end{equation}
Here $W_{1,0}$ and $Z_{1,0}^0$ are the functions given in \eqref{eq_bubble_2} and \eqref{eq_Z}.
\end{prop}
\noindent Its proof is based on the Fourier transform technique which was introduced by Gonz\'alez and Qing \cite{GQ} and soon improved by Gonz\'alez and Wang \cite{GW} and Kim et al. \cite{KMW, KMW2},
where the authors studied the existence and non-compactness characteristics of the solution set of \eqref{eq_Yamabe}.
We first need to remind a lemma obtained in \cite[Section 7]{GQ} and \cite[Subsection 4.3]{KMW}.
\begin{lemma}\label{lemma_KMW}
(1) Assume that $n \ge 3$ and $\ga \in (0,1)$.
Let $\whw_{1,0}(\xi,x_N)$ be the Fourier transform of $W_{1,0}(\bx,x_N)$ in $\bx \in \mr^n$ for each fixed $x_N > 0$
and $K_{\ga}$ the modified Bessel function of the second kind of order $\ga$.
Also choose appropriate numbers $d_1,\, d_2 > 0$ depending only on $n$ and $\ga$ so that the functions $\vp(t) = d_1 t^{\ga}K_{\ga}(t)$ and $\hw_{1,0}(t) = d_2 t^{-\ga}K_{\ga}(t)$ solve
\[\phi''(t) + {1-2\ga \over t} \phi'(t) - \phi(t) = 0, \quad \phi(0) = 1 \text{ and } \phi(\infty) = 0\]
and
\[\phi''(t) + {1+2\ga \over t} \phi'(t) - \phi(t) = 0 \quad \text{and} \quad  \lim_{t \to 0} t^{2\ga}\phi(t) + \lim_{t \to \infty} t^{\ga+{1 \over 2}} e^t \phi(t) \le C\]
for some $C > 0$, respectively. Then
\[\whw_{1,0}(\xi,x_N) = \hw_{1,0}(\xi)\, \vp(|\xi|x_N) \quad \text{for every } \xi \in \mr^n \text{ and } x_N > 0.\]

\medskip \noindent (2) Let
\[\mca_{\alpha} = \int_0^{\infty} t^{\alpha-2\ga}\vp^2(t)\, dt, \quad \mca_{\alpha}' = \int_0^{\infty} t^{\alpha-2\ga} \vp(t)\, \vp'(t)\, dt, \quad
\mca_{\alpha}'' = \int_0^{\infty} t^{\alpha-2\ga} (\vp'(t))^2\, dt,\]
\[\mcb_{\beta} = \int_0^{\infty} t^{-\beta+2\ga} \hw_{1,0}^2(t) t^{n-1} dt, \quad
\mcb_{\beta}' = \int_0^{\infty} t^{-\beta+2\ga} \hw_{1,0}(t)\, \hw_{1,0}'(t) t^{n-1} dt\]
and
\[\mcb_{\beta}'' = \int_0^{\infty} t^{-\beta+2\ga} (\hw_{1,0}'(t))^2 t^{n-1} dt\]
for $\alpha,\, \beta \in \mn \cup \{0\}$. Then we have
\begin{align*}
\mca_{\alpha} &= \({\alpha+2 \over \alpha+1}\) \cdot \left[\({\alpha+1 \over 2}\)^2 - \ga^2\right]^{-1} \mca_{\alpha+2} = -\({\alpha+1 \over 2} - \ga\)^{-1} \mca_{\alpha+1}' \\
&= \({\alpha+1 \over 2} - \ga\) \({\alpha-1 \over 2} + \ga\)^{-1} \mca_{\alpha}''
\end{align*}
for $\alpha$ odd, $\alpha \ge 1$ and
\[\mcb_{\beta} = {4(n-\beta+1) \mcb_{\beta-2} \over (n-\beta)(n+2\ga-\beta)(n-2\ga-\beta)} = - {2 \mcb_{\beta-1}' \over n+2\ga-\beta},
\quad \mcb_{\beta-2} = {(n-2\ga-\beta) \mcb_{\beta-2}'' \over n+2\ga-\beta+2}\]
for $\beta$ even, $\beta \ge 2$.
\end{lemma}
\noindent With the help of the previous lemma, we evaluate the following nine integrals.
\begin{lemma}\label{lemma_W_int_2}
Assume that $\ga \in (0,1)$ and $n > 2 + 2\ga$.
If we set $\mcc_0 = |\ms^{n-1}|\mca_3\mcb_2$, then
\begin{align*}
\mci_1 &= \int_{\mr^N_+} x_N^{1-2\ga} W_{1,0}^2 dx = \left[{3 \over 2\(1-\ga^2\)}\right] \mcc_0,\\
\mci_2 &= \int_{\mr^N_+} x_N^{1-2\ga} r W_{1,0} (\pa_r W_{1,0})\, dx = - \left[{3n \over 4\(1-\ga^2\)}\right] \mcc_0,\\
\mci_3 &= \int_{\mr^N_+} x_N^{2-2\ga} W_{1,0} (\pa_N W_{1,0})\, dx = - \left[{3 \over 2(1+\ga)}\right] \mcc_0,\\
\mci_4 &= \int_{\mr^N_+} x_N^{2-2\ga} r(\pa_r W_{1,0})(\pa_N W_{1,0})\, dx = \left[{3n-2(1+\ga) \over 4(1+\ga)}\right] \mcc_0,\\
\mci_5 &= \int_{\mr^N_+} x_N^{3-2\ga} W_{1,0} (\Delta_{\bx} W_{1,0})\, dx = - \mcc_0,\\
\mci_6 &= \int_{\mr^N_+} x_N^{3-2\ga} (\pa_r W_{1,0})^2 dx = \mcc_0,\\
\mci_7 &= \int_{\mr^N_+} x_N^{3-2\ga} (\pa_N W_{1,0})^2 dx = \({2-\ga \over 1+\ga}\) \mcc_0,\\
\mci_8 &= \int_{\mr^N_+} x_N^{3-2\ga} r (\pa_r W_{1,0})(\pa_{rr} W_{1,0})\, dx = -{n \over 2}\, \mcc_0,\\
\mci_9 &= \int_{\mr^N_+} x_N^{4-2\ga} (\pa_N W_{1,0}) (\Delta_{\bx} W_{1,0})\, dx = (2-\ga)\, \mcc_0
\end{align*}
for $r = |\bx|$.
\end{lemma}
\begin{proof}
The quantities $\mci_1,\, \mci_6,\, \mci_7$ were computed in \cite[Lemma 7.2]{GQ}.
Besides, \cite[Lemma B.4]{KMW2} provides the value of $\mci_8$, and its proof suggests a way to calculate $\mci_2,\, \mci_3$.
Accordingly we only take into account the others.
Throughout the proof, we agree that the variable of $\hw_{1,0}$ and $\hw_{1,0}'$ is $|\xi|$ and that of $\vp$ and $\vp'$ is $|\xi|x_N$.
Also, $'$ is used to represent the differentiation in the radial variable $|\xi|$.

\medskip
It follows from Parseval's theorem that
\begin{align*}
\mci_4 &= \int_0^{\infty} x_N^{2-2\ga} \(\int_{\mr^n} \widehat{x_i \pa_i W_{1,0}} \pa_N \whw_{1,0} d\xi\) dx_N\\
&= - \int_0^{\infty} x_N^{2-2\ga} \left[ \int_{\mr^n} \(n\whw_{1,0} + \xi_i \pa_i\whw_{1,0}\) \pa_N \whw_{1,0} d\xi \right] dx_N \\
&= - \int_0^{\infty} x_N^{2-2\ga} \left[ \int_{\mr^n} \(n\hw_{1,0}\vp + |\xi| \hw_{1,0}'\vp + x_N|\xi| \hw_{1,0}\vp' \) |\xi| \hw_{1,0} \vp' d\xi \right] dx_N \\
&= - |\ms^{n-1}| \(n \mca_2\mcb_2' + \mca_1'\mcb_2' + \mca_3''\mcb_3\).
\end{align*}
Therefore Lemma \ref{lemma_KMW} (2) gives the value in the statement.
On the other hand, we observe by applying the integration by parts that
\[\mci_5 = - \int_{\mr^N_+} x_N^{3-2\ga} |\nabla_{\bx} W_{1,0}|^2\, dx = - \mci_6.\]
Also one can verify
\[\mci_9 = -\int_0^{\infty} x_N^{4-2\ga} \left[\int_{\mr^n} |\xi|^2 \whw_{1,0} \(\pa_N \whw_{1,0}\) d\xi \right] dx_N = - |\ms^{n-1}| \mca_4'\mcb_2,\]
finishing the proof.
\end{proof}

\begin{proof}[Proof of Proposition \ref{prop_W_int}]
We have
\[Z_{1,0}^0 = r\, \pa_r W_{1,0} + x_N\, \pa_N W_{1,0} + \({n-2\ga \over 2}\) W_{1,0} \quad \text{and} \quad \Delta_{\bx} W_{1,0} = \pa_{rr} W_{1,0} + (n-1) r^{-1} \pa_r W_{1,0}\]
for $r = |\bx|$. Hence
\begin{align*}
\int_{\mr^N_+} x_N^{3-2\ga} \Delta_{\bx} W_{1,0} Z_{1,0}^0 dx &= \mci_8 + (n-1) \mci_6 + \mci_9 +\({n-2\ga \over 2}\) \mci_5,\\
\int_{\mr^N_+} x_N^{2-2\ga} \pa_N W_{1,0} Z_{1,0}^0 dx &= \mci_4 + \mci_7 + \({n-2\ga \over 2}\) \mci_3
\intertext{and}
\int_{\mr^N_+} x_N^{1-2\ga} W_{1,0} Z_{1,0}^0 dx &= \mci_2 + \mci_3 + \({n-2\ga \over 2}\) \mci_1.
\end{align*}
Now an easy application of Lemma \ref{lemma_W_int_2} completes the proof.
\end{proof}

\noindent \textbf{Acknowledgement}
S. Kim was supported by FONDECYT Grant 3140530 during his stay in the Pontifical Catholic University of Chile.
He is currently supported by Basic Science Research Program through the National Research Foundation of Korea(NRF) funded by the Ministry of Education (NRF2017R1C1B5076384).
M. Musso has been supported by Fondecyt grant 1160135.
The research of J. Wei is partially supported by NSERC of Canada.
Part of the paper was written when S. Kim was visiting the University of British Columbia and Wuhan University.
He appreciates the institutions for their hospitality and financial supports.

\end{document}